%% file: DNT.tex
\documentclass[11pt]{article}

\usepackage[utf8]{inputenc}
\usepackage{amsmath,amscd,amssymb,amsthm}
\usepackage{tikz}
\usepackage{euscript}
\usepackage{latexsym}
\usepackage{color}
\usepackage{cite}
\usepackage{hyperref}

\usepackage{soul}

\makeatletter

\renewcommand{\@oddfoot}{%
  \hfil\thepage\hfil\llap{\footnotesize (\today)}}

\allowdisplaybreaks[1]

\newcommand{\IR}{{\mathbb R}}
\newcommand{\IN}{\mathbb N}
\newcommand{\IZ}{\mathbb Z}

\newcommand{\C}{{\mathcal C}}
\newcommand{\E}{{\mathcal E}}
\newcommand{\intd}{\,\mathrm{d}}
\DeclareMathOperator{\e}{e}
\renewcommand{\i}{{\mathbf{i}}}
\newcommand{\abs}[1]{\mathopen|#1\mathclose|}
\newcommand{\norm}[1]{\mathopen\|#1\mathclose\|}
\newcommand{\intervaloo}[1]{\mathopen]#1\mathclose[}
\newcommand{\intervalco}[1]{\mathopen[#1\mathclose[}
\newcommand{\intervaloc}[1]{\mathopen]#1\mathclose]}
\newcommand{\intervalcc}[1]{\mathopen[#1\mathclose]}
\newcommand{\bigintervaloo}[1]{\bigl]#1\bigr[}
\@ifundefined{leqslant}{}{\let\le=\leqslant  \let\leq=\leqslant}
\@ifundefined{geqslant}{}{\let\ge=\geqslant  \let\geq=\geqslant}

\newtheorem{Thm}{Theorem}
\newtheorem{Lem}[Thm]{Lemma}

\newtheorem{Prop}[Thm]{Proposition}

\theoremstyle{definition}

\theoremstyle{remark}
\newtheorem{Rem}[Thm]{Remark}

\definecolor{verydarkgreen}{rgb}{0.4,0.47,0.42}
\definecolor{darkblue}{rgb}{0.23,0.33,0.95}
\definecolor{darkred}{rgb}{0.67,0.03,0.06}
\definecolor{darkgreen}{RGB}{55,160,80}

\@ifundefined{hypersetup}{}{%
  \hypersetup{%
    pdfauthor={C. De Coster \& S. Nicaise \& C. Troestler},
    pdftitle={Buckling spectrum},
    colorlinks=true,
    urlcolor=verydarkgreen,
    citecolor=darkblue,
    linkcolor=darkred,
    pdfstartview=FitH,
  }}

\definecolor{fixme}{RGB}{224,15,182}
\definecolor{Chris}{RGB}{180,125,25}

\makeatother

\numberwithin{equation}{section}

\numberwithin{Thm}{section}

\numberwithin{Def}{section}

\begin{document}

\title{Nodal properties of eigenfunctions of a generalized buckling problem on balls}

\author{%
  Colette De Coster,
  Serge Nicaise
  \medskip\\
  Universit\'e de Valenciennes et du Hainaut Cambr\'esis\\
  LAMAV,  FR CNRS 2956, \\
  Institut des Sciences et Techniques de Valenciennes\\
  F-59313  Valenciennes Cedex 9, France\\
  $\{$Colette.DeCoster,Serge.Nicaise$\}$@univ-valenciennes.fr
	 \medskip\\
  Christophe Troestler\thanks{%
    This author was partially supported by
    the program ``Qualitative study of solutions of variational elliptic
    partial differerential equations. Symmetries, bifurcations,
    singularities, multiplicity and numerics'' (2.4.550.10.F)
    and the project ``Existence and asymptotic behavior of solutions
    to systems of semilinear elliptic partial differential equations''
    (T.1110.14)
    of the \emph{Fonds de la Recherche Fondamentale Collective},
    Belgium. }
	\medskip\\
	Institut de math\'ematique, 
	\\
	Universit\'e de Mons-Hainaut, 
	\\
	place du parc 20, 
	\\
	B-7000 Mons, Belgium
	\\
	Christophe.Troestler@umons.ac.be
	}

\maketitle

\begin{abstract}
  In this paper we are interested in the following fourth order
  eigenvalue problem coming from the buckling of thin films on liquid
  substrates:
  \begin{equation*}
    \begin{cases}
      \Delta^2 u+ \kappa^2 u=-\lambda \Delta u &\text{in } B_1,\\
      u=\partial_r u= 0 &\text{on } \partial B_1,
    \end{cases}
  \end{equation*}
  where $B_1$ is the unit ball in~$\IR^N$.
  When $\kappa > 0$ is small, we show that the first eigenvalue is
  simple and the first eigenfunction, which gives the shape of the
  film for small displacements, is positive.  However, when $\kappa$
  increases, we establish that the first eigenvalue is not always simple
  and the first eigenfunction may change sign.
  More precisely, for any $\kappa \in
  \intervaloo{0,+\infty}$, we give the exact multiplicity of the first
  eigenvalue and the number of nodal regions of the first eigenfunction.
\end{abstract}



\section{Introduction}

This paper is motivated by the study of clamped thin elastic membranes
supported on a fluid substrate which can model
geological structures~\cite{Peletier},
biological organs (such as lungs, see~\cite{Zasadzinski}),
and water repellent surfaces.
A one-dimensional model of these films was given by Pocivavsek et
al.~\cite{pocivavsek} based on the principle that
the shape that the film takes must minimize the sum of the
elastic bending energy, measured by the curvature,
and the potential energy due to the vertical displacement
of the fluid column.
A detailed mathematical analysis of this problem was
performed in~\cite{wrinkling1D}.

Based on these ideas, a natural extension was proposed to higher
dimensions~\cite{wrinkling2D}.  More precisely let $\Omega$ be a
reference domain giving the shape of the film in the absence of
external forces and let $\Omega_\epsilon$ be a small compression of it
with $\Omega_\epsilon \to \Omega$ in some sense as $\epsilon \to 0$.
The shape of the film after the small compression is given by the
function $u_\epsilon : \Omega_\epsilon \to \IR$, giving the vertical
displacement of the film, which minimizes
\begin{equation*}
  \E_\epsilon :
  H^2_0(\Omega_\epsilon) \to \IR : v \mapsto
  \int_{\Omega_\epsilon} \abs{\Delta v}^2
  + \kappa^2 \int_{\Omega_\epsilon} v^2
\end{equation*}
under the constraint that the membrane can bend but not stretch,
thus that its total area does not change:
\begin{equation*}
  \int_{\Omega_\epsilon} \sqrt{1 + \abs{\nabla v}^2}
  = \abs{\Omega}.
\end{equation*}
The first term of $\E_\epsilon$ is the bending energy of the film, the
second accounts for the potential energy coming from the vertical
fluid displacement, and $\kappa$ is a constant expressing the relative
strength of these two energies.  It has been shown~\cite{wrinkling2D}
that, as $\epsilon \to 0$, minimizers $u_\epsilon$ of $\E_\epsilon$
behave like  $u_0$ where $u_0 \in
H^2(\Omega)\setminus\{0\}$ satisfies
\begin{equation}
  \label{eq:general problem}
  \begin{cases}
    \Delta^2 u+ \kappa^2 u = -\lambda_1 \Delta u &\text{in } \Omega,\\
    u = \frac{\partial u}{\partial \nu} = 0 &\text{on } \partial \Omega.
  \end{cases}
\end{equation}
Here $\Delta^2 u := \Delta(\Delta u)$
and $\lambda_1$ is the first buckling eigenvalue of $\Delta^2 +
\kappa^2$, namely
\begin{equation*}
  \lambda_1 :=
  \min_{u \in H^2_0(\Omega),\ \norm{\nabla u}_{L^2(\Omega)} = 1}
  \Bigl( \int_\Omega \abs{\Delta u}^2 + \kappa^2 \int_\Omega u^2
  \Bigr) .
\end{equation*}
As usual, we write $H^2_0(\Omega)$ for the
set of functions $u \in H^2(\Omega)$ that satisfy the clamped boundary
conditions $u = \frac{\partial u}{\partial \nu} = 0$ on $\partial \Omega$.
This first eigenvalue represents the minimal compression at which the
plate exhibits buckling (see \cite{KLV}).
The corresponding eigenfunction gives the shape of the membrane
when the compression is small.

In this work, we study the evolution of the spectrum with respect to
$\kappa\geq 0$ when $\Omega = B_1$ is the unit ball of~$\IR^N$.
More precisely, we determine values of
$\lambda$  and the shape of $u\ne 0$ satisfying the problem:
\begin{equation}
  \label{eq:pbm}
  \begin{cases}
    \Delta^2 u+ \kappa^2 u=-\lambda \Delta u &\text{in } B_1,\\
    u=\partial_r u= 0 &\text{on } \partial B_1.
  \end{cases}
\end{equation}
A special attention is devoted to the shape and nodal properties of the
first eigenfunction.

\medbreak

There is a large literature on the study of the positivity and of the
change of sign of the first eigenfunction for the eigenvalue problem
\begin{equation*}
  \begin{cases}
    \Delta^2 u= \lambda u &\text{in } \Omega,\\
    u= \frac{\partial u}{\partial \nu} = 0 &\text{on } \partial \Omega,
  \end{cases}
\end{equation*}
or for the buckling eigenvalue problem
\begin{equation*}
  \begin{cases}
    \Delta^2 u=-\lambda \Delta u &\text{in } \Omega,\\
    u= \frac{\partial u}{\partial \nu} = 0 &\text{on } \partial \Omega,
  \end{cases}
\end{equation*}
for different shapes of the domain $\Omega$ (see for example
\cite{BDJM, Coffman, CD, CDS, Duffin, GGS, GS, GS2, KKM, Sweers,
  Wieners}).  Roughly, these papers say that, except for $\Omega$
close to a disk in a suitable sense, the first eigenfunction changes
sign. The only reference that we know where the authors consider the
``mixed''
problem \eqref{eq:general problem} are \cite{KLV}, where the authors obtain
asymptotic estimate on the first eigenvalue of~\eqref{eq:general
  problem}, and
\cite{Chasman, Chasman-these} where the author considers the equation
$\Delta^2 u -\tau \Delta u = \omega u$ on a ball
with ``free'' boundary
conditions, where $\tau > 0$ is fixed and the eigenvalues $\omega > 0$
are sought.
In these latter works, L.~Chasman gives the structure of 
eigenfunctions but does not give sign information on them as she is
meanly interested in an isoperimetric inequality.
Note also that $\tau$ and $\omega$ give coefficients
of $\Delta u$ and $u$ of opposite sign compared to our case.

The paper is organized as follows.
In Section~\ref{Sect1}, we explain how we will find solutions
to~\eqref{eq:pbm} despite the fact that the method of separation of
variables is not directly applicable because of the presence of
``cross terms'' when we apply $\Delta^2$ to a function of the type
$R(r) S(\theta)$.
Section~\ref{Sect2} will deal with the easy case $\kappa = 0$
for which the eigenvalues are explicitly given in terms of
positive roots $(j_{\nu,\ell})_{\ell=1}^\infty$
of $J_\nu$ for some~$\nu$.
Recall that $J_\nu$ denotes the Bessel function of the First Kind of
order~$\nu$.

In Section~\ref{Sect3} and~\ref{sec:nodal prop}, we deal with
$\kappa > 0$.  First we show (see Theorem~\ref{thm:alpha kl}) that,
for all $k \in \IN$, there exists an increasing sequence
$\alpha_{k,\ell} = \alpha_{k,\ell}(\kappa) > \sqrt{\kappa}$, $\ell \ge
1$, such that $\lambda_{k,\ell} := \alpha_{k,\ell}^2 +
{\kappa^2}/{\alpha_{k,\ell}^2}$ is an eigenvalue of~\eqref{eq:pbm}
with corresponding eigenfunctions of the form $R_{k,\ell}(r)\e^{\pm\i
  k\theta}$ where 
\begin{equation*}
  R_{k,\ell}(r)
  := c J_k(\alpha_{k,\ell} \, r)
  + d J_k\Bigl(\frac{\kappa}{\alpha_{k,\ell}} \, r\Bigr)
\end{equation*}
for some $(c,d) \ne (0,0)$ suitably chosen (depending on $\kappa$,
$k$, and $\ell$).
The spectrum of~\eqref{eq:pbm} is exactly
$\{ \lambda_{k,\ell} \mid k \in \IN,\ \ell \ge 1 \}$.
Its minimal value $\lambda_1 = \lambda_1(\kappa)$
correspond the the minimum of
$\{ \alpha_{k,\ell} \mid k \in \IN,\ \ell \ge 1 \}$.
Contrarily to the standard case of second order elliptic operators,
the minimum is not always given by the same $\alpha_{k,\ell}$
but, depending on $\kappa$, is  $\alpha_{0,1}$ or $\alpha_{1,1}$ (see
Figure~\ref{figure 1}).
The main results of Section~\ref{Sect3} (see Theorems~\ref{1st eigen}
and~\ref{1st eigen2}) precisely describe this behavior
depending on the value of~$\kappa$ and explicitly give the
corresponding eigenspace which may be of dimension greater than~$1$.

In Section~\ref{sec:nodal prop} we show that, even when $\lambda_1$ is
simple, the first eigenfunction may change sign and can even possess
an arbitrarily large number of nodal domains.
More precisely, we prove the following theorem (see
Figure~\ref{fig:phi1} for a graphical illustration).

\begin{Thm}
  \label{1st eigen nodal}
  Denote $R_{k,\ell}$ a function defined by equation~\eqref{eq R} with
  $(c,d)$ a non-trivial solution of~\eqref{eq:R bd cond} and $\alpha =
  \alpha_{k,\ell}$ with $\alpha_{k,\ell}$ given by
  Theorem~\ref{thm:alpha kl}.
  \begin{itemize}
  \item If $\kappa\in \intervalco{0,j_{0,1}j_{0,2}}$, the first
    eigenvalue is simple and is given by
    $\lambda_1(\kappa) = \alpha_{0,1}^2(\kappa) +
    \kappa^2 / \alpha_{0,1}^2(\kappa)$ and the eigenfunctions
    $\varphi_1$ are radial, one-signed and $\abs{\varphi_1}$ is
    decreasing with respect to $r$.

  \item If $\kappa\in \intervaloo{j_{1,n}j_{1,n+1},
      \,j_{0,n+1}j_{0,n+2}}$, for some $n\ge 1$, the first eigenvalue
    is simple and given by
    $\lambda_1(\kappa) = \alpha_{0,1}^2(\kappa) + \kappa^2 /
    \alpha_{0,1}^2(\kappa)$ and the eigenfunctions are radial and have
    $n+1$ nodal regions.

  \item If $\kappa \in \intervaloo{j_{0,n+1}j_{0,n+2},
      \,j_{1,n+1}j_{1,n+2}}$, for some $n\ge 0$, the first eigenvalue
    is given by $\lambda_1(\kappa) = \alpha_{1,1}^2(\kappa) + \kappa^2 /
    \alpha_{1,1}^2(\kappa)$ and the eigenfunctions $\varphi_1$ have
    the form
    \begin{equation*}
      R_{1,1}(r) (c_1\cos\theta + c_2\sin\theta),
      \qquad
      c_1, c_2 \in \IR.
    \end{equation*}
    Moreover the function $R_{1,1}$ has $n$ simple zeros
    in $\intervaloo{0,1}$,
    i.e., $\varphi_1$ has $2(n+1)$ nodal regions.
  \end{itemize}
\end{Thm}

Information on the eigenspaces at the countably many $\kappa > 0$ not
considered in the previous theorem is also provided.  For these
$\kappa$, $\alpha_{0,1}(\kappa) = \alpha_{1,1}(\kappa)$ and the
eigenspaces have even larger dimensions (see Theorem~\ref{1st
  eigen2}).

For simplicity this paper is written for a two dimensional ball but,
in Section~\ref{sec:any dim}, we show how our results naturally extend
to any dimension.

\begin{figure}[ht]
  \begin{minipage}[b]{0.31\linewidth}
    \centering
    \includegraphics[width=\linewidth, bb=64 215 552 580]{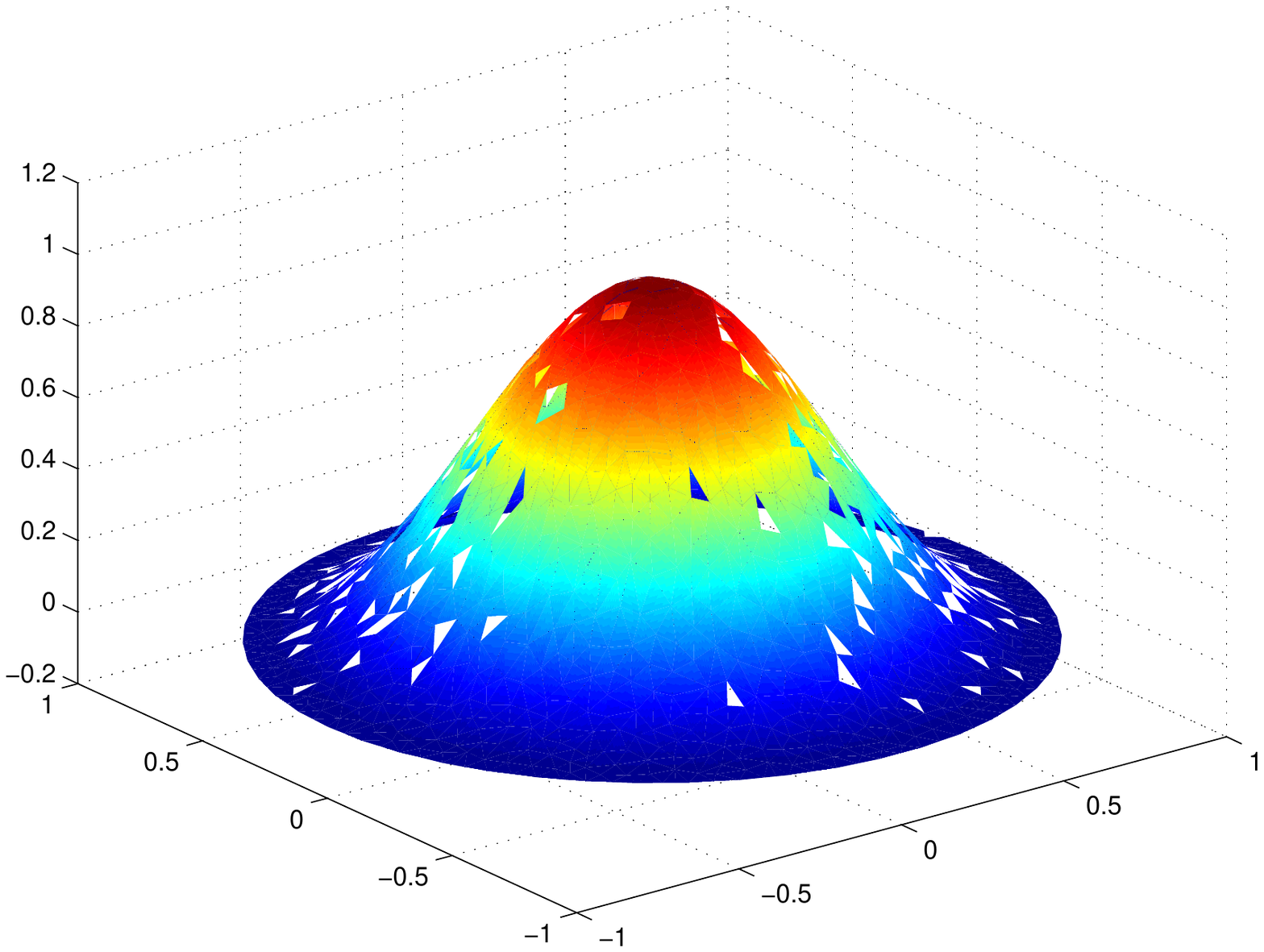}
    
    $\kappa\in \intervalco{0,j_{0,1}j_{0,2}}$
  \end{minipage}
  \hfill
  \begin{minipage}[b]{0.31\linewidth}
    \centering
    \includegraphics[width=\linewidth, bb=64 215 552 580]{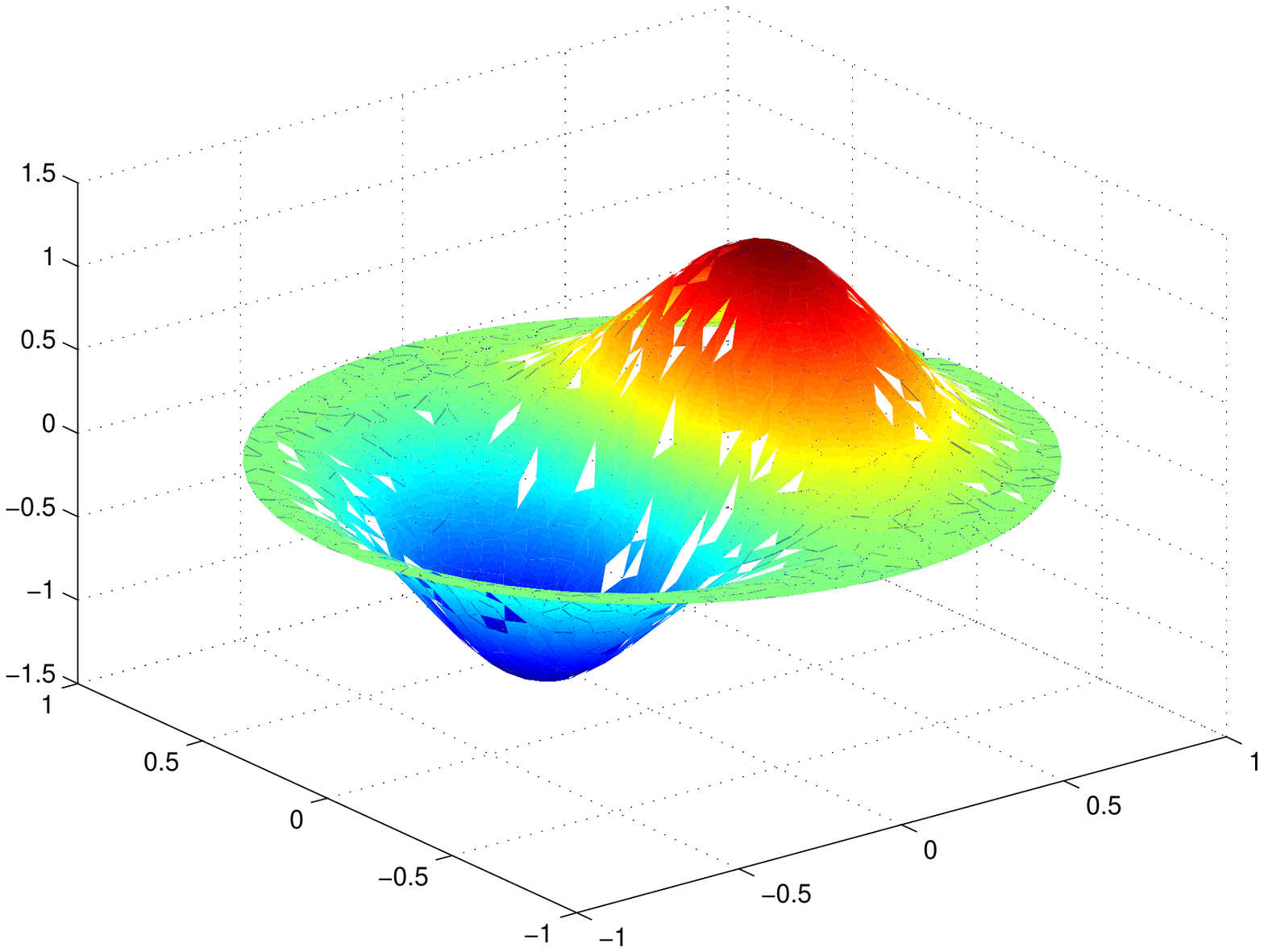}
    
    $\kappa\in \intervaloo{j_{0,1}j_{0,2},\, j_{1,1}j_{1,2}}$
  \end{minipage}
  \hfill
  \begin{minipage}[b]{0.31\linewidth}
    \centering
    \includegraphics[width=\linewidth, bb=64 214 552 580]{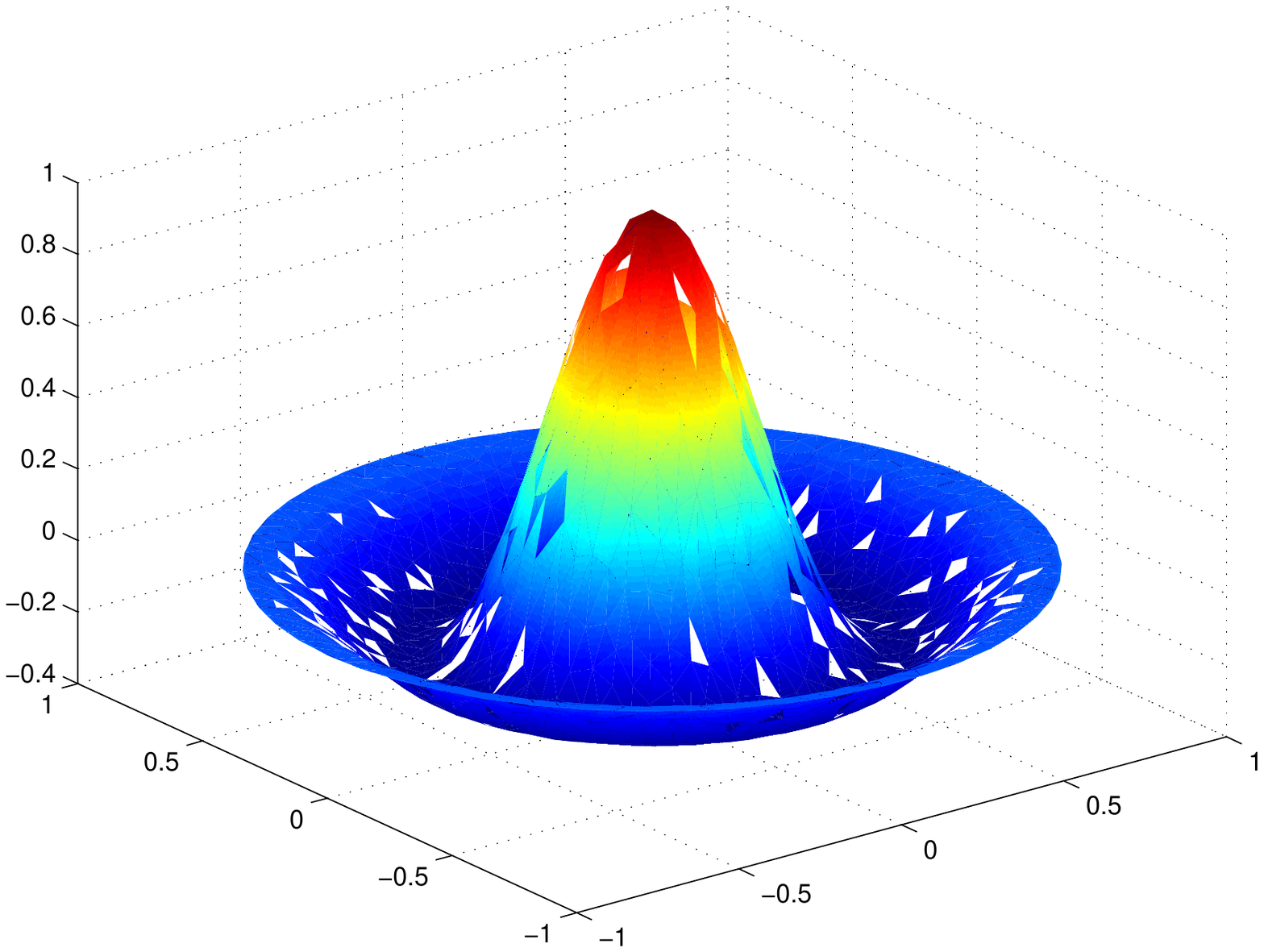}

    $\kappa\in \intervaloo{j_{1,1}j_{1,2},\, j_{0,2}j_{0,3}}$
  \end{minipage}
  \vspace{2ex}

  \begin{minipage}[b]{0.31\linewidth}
    \centering
    \includegraphics[width=\linewidth, bb=66 208 550 580]{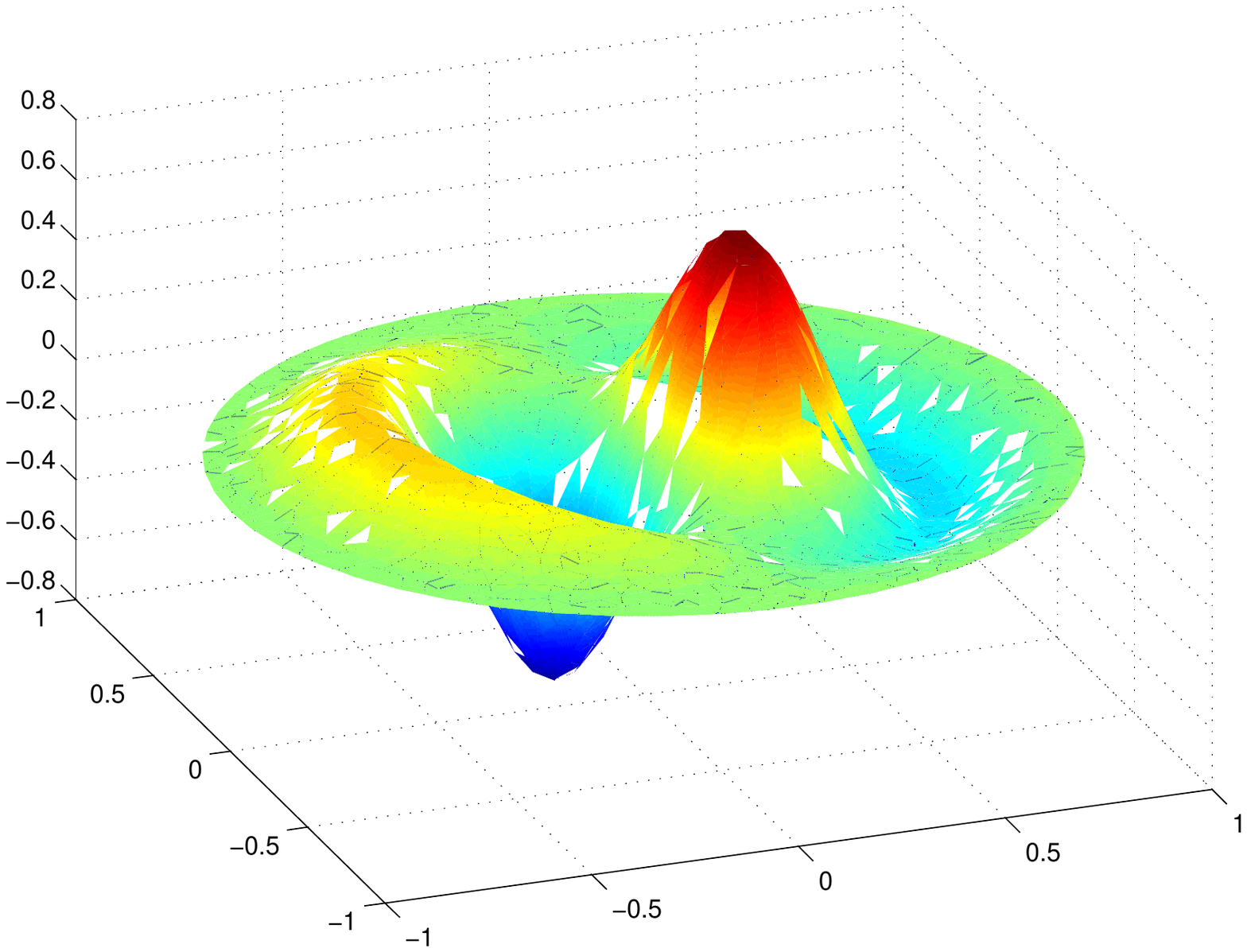}
  
    $\kappa \in \intervaloo{j_{0,2}j_{0,3},\, j_{1,2}j_{1,3}}$
  \end{minipage}
  \hfill
  \begin{minipage}[b]{0.31\linewidth}
    \centering
    \includegraphics[width=\linewidth, bb=68 206 548 580]{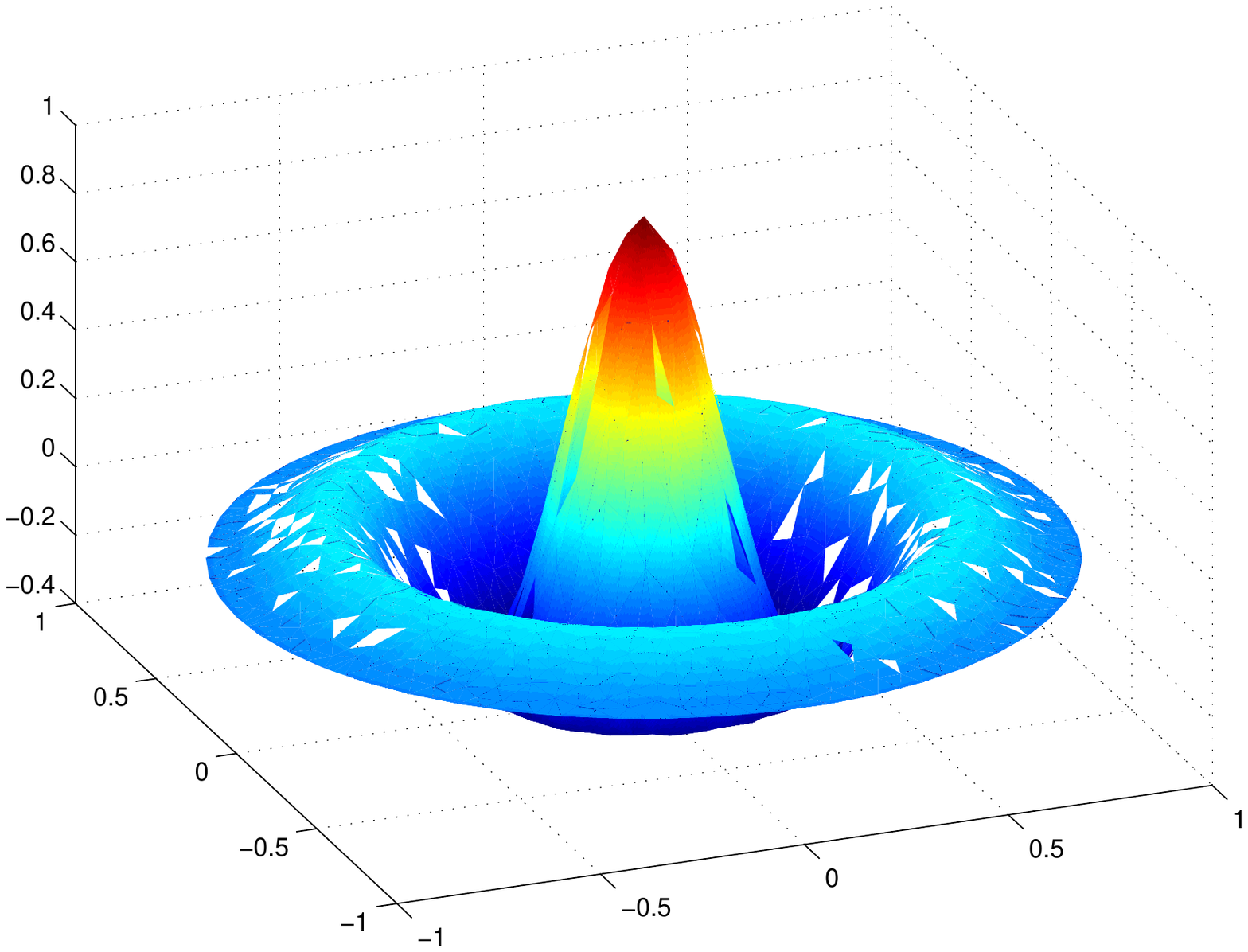}
    
    $\kappa\in \intervaloo{j_{1,2}j_{1,3},\, j_{0,3}j_{0,4}}$
  \end{minipage}
  \hfill
  \begin{minipage}[b]{0.31\linewidth}
    \centering
    \includegraphics[width=\linewidth, bb=70 205 546 589]{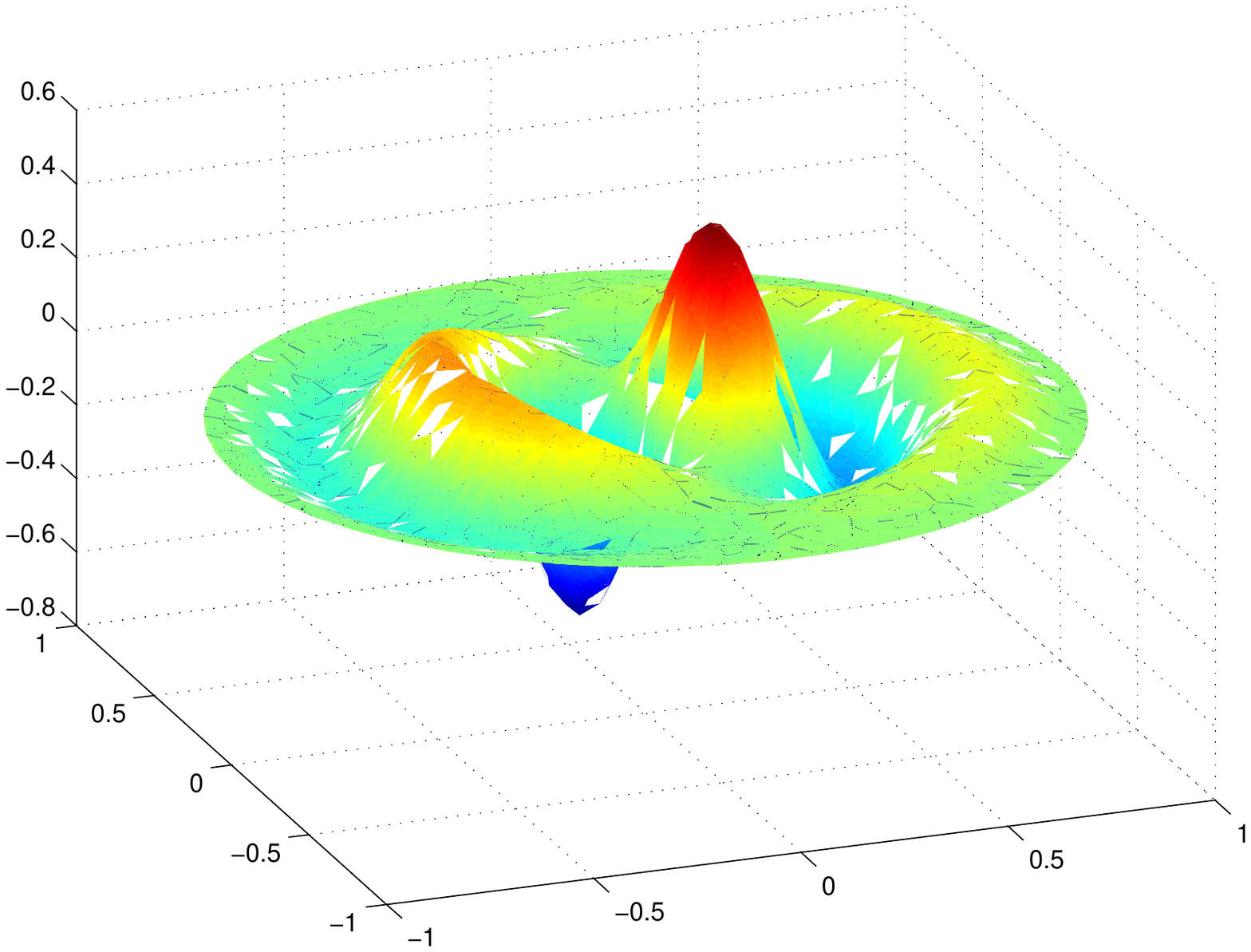}
    
    $\kappa\in \intervaloo{j_{0,3}j_{0,4},\, j_{1,3}j_{1,4}}$
  \end{minipage}

  \caption{Graphs of $\varphi_1$ for various values of $\kappa$.}
  \label{fig:phi1}
\end{figure}


In this paper, we use the following notations.  The set of natural
numbers is denoted $\IN=\{0,1,2,\ldots\}$, the set of positive
integers is $\IN^*=\{1,2,\ldots\}$, and $j_{\nu,\ell}$, $\ell \in
\IN^*$ denotes the $\ell$-th positive root of $J_\nu$, the Bessel
function of the First Kind of order~$\nu$.

\section{Preliminaries}
\label{Sect1}

Given two complex numbers $\alpha$, $\beta$, we look for special
solutions $u$ to the equation
\begin{equation}
  \label{EDP}
  (\Delta+\alpha^2) (\Delta+\beta^2)u=0.
\end{equation}
Such an equation is equivalent to
\begin{equation}\label{serge2}
  \Delta^2u + (\alpha^2+\beta^2)\Delta u + \alpha^2 \beta^2 u=0.
\end{equation}
Hence if we look for a solution to
\begin{equation}\label{serge3}
  \Delta^2u + \kappa^2 \, u = -\lambda \Delta u
\end{equation}
with $\kappa \ge 0$ fixed, it suffices to take $\alpha \beta = \kappa$ and
$\alpha^2+\beta^2 = \lambda$.

Given that we work in two dimensions,
we use the antsatz $u(r,\theta)=R(r) \e^{\i k\theta}$ with $k\in
\IZ$, where $(r,\theta)$ are the polar coordinates,
and notice that~\eqref{EDP} is equivalent to the fourth
order differential equation (in $\partial_r$)
\begin{equation}\label{serge1polar}
  L(\partial_r, r, \alpha,\beta, \abs{k}) R=0.
\end{equation}
We write $L(\partial_r, r, \alpha,\beta, \abs{k})$ to emphasize that the
coefficients of the differential operator depend continuously on $r$,
$\alpha$, $\beta$ and that the sign of $k$ does not matter.
Hence by the theory of ordinary
differential equations, $L$ has four linearly independent solutions.  To
find them it suffices to notice that
\begin{equation*}
  (\Delta+\alpha^2)u=0 \quad\Rightarrow\quad
  (\Delta+\alpha^2)(\Delta+\beta^2)u=0.
\end{equation*}
Thus if
\begin{equation}
  \label{serge4}
  (\Delta+\alpha^2)\bigl(R(r) \e^{\i k\theta}\bigr)=0
\end{equation}
then $R$ is a solution to~\eqref{serge1polar}.  But a solution
to~\eqref{serge4} is simpler to find. Indeed then $R$ satisfies the
Bessel equation
$$
(r\partial_r)^2 R+\alpha^2 r^2 R=k^2 R.
$$
Hence if $\alpha \ne 0$, $R$ is a linear combination of $J_{|k|}(\alpha
r)$ and of $Y_{|k|}(\alpha r)$.
On the contrary if $\alpha=0$ and $k \not=0$,
then $R$ is a linear combination of $r^k$ and of $r^{-k}$  while $R$ is a linear combination of $1$ and $\log r$ when $\alpha=0$ and $k=0$.
\medbreak

We have therefore proved the following result:

\begin{Lem}
\label{lserge1}
  Let $k\in \mathbb Z$.
  \begin{enumerate}
  \item 
	If $\alpha\ne\beta$ both non-zero, then the four linearly
    independent solutions to~\eqref{serge1polar} are $J_{|k|}(\alpha r)$,
    $Y_{|k|}(\alpha r)$, $J_{|k|}(\beta r)$, $Y_{|k|}(\beta r)$.
  \item 
	If $\alpha\ne0$ and $\beta=0$, then the four linearly independent
    solutions to~\eqref{serge1polar} are $J_{|k|}(\alpha r)$,
    $Y_{|k|}(\alpha r)$, $r^k$, $r^{-k}$ if $k \ne 0$
    and $J_{|k|}(\alpha r)$, $Y_{|k|}(\alpha r)$, $1$, $\log r$ if $k = 0$.
  \item 
    If $\alpha=\beta \ne 0$, then the four linearly independent
    solutions to~\eqref{serge1polar} are $J_{|k|}(\alpha r)$,
    $Y_{|k|}(\alpha r)$, $rJ'_{|k|}(\alpha r)$,
    $rY'_{|k|}(\alpha r)$.
  \end{enumerate}
\end{Lem}

\begin{proof}
  The first two points were already treated before, the linear
  independence coming easily from the asymptotic behavior of the
  Bessel functions and their derivatives at~$0$.
  For the third
  case, it suffices to notice that taking
  $\gamma\ne\alpha$, we see that
  $$
  \frac{J_{|k|}(\alpha r)-J_{|k|}(\gamma r)}{\alpha-\gamma}
  $$
  is a solution of
  $$
  L(\partial_r, r,\alpha,\gamma, \abs{k}) R=0.
  $$
  Letting $\gamma$ tend to $\alpha$, we prove that $R(r)=rJ'_{|k|}(\alpha r)$ is
	a solution of
  $$
	L(\partial_r, r,\alpha,\alpha, \abs{k}) R=0.
  $$
  The same argument holds for $Y_{\abs{k}}(\alpha r)$.
\end{proof}

\section{Eigenvalues in the case $\kappa=0$}
\label{Sect2}

Here we want to characterize the full spectrum of the buckling problem
with $\kappa = 0$ on the unit disk.  In other words, we look for a
non-trivial $u$ and $\lambda > 0$ such that
\begin{equation}\label{serge5}
  \begin{cases}
    \Delta^2 u=-\lambda \Delta u &\text{ in } D_1,\\
    u=\partial_r u= 0 &\text{ on } \partial D_1,
  \end{cases}
\end{equation}
where $D_1 = \bigl\{x\in \IR^2 \bigm| \abs{x} <1 \bigr\}$.
According to the previous section, we look for solutions $u$ in the form
\begin{equation*}
  u=R(r) \e^{\i k\theta}, \qquad
  \text{with } k\in \IZ,
\end{equation*}
to equation~\eqref{EDP} with $\alpha=\sqrt{\lambda}$ and $\beta=0$.
From Lemma \ref{lserge1}, we see that
$$
R(r)=c r^{|k|}+dJ_{|k|}(\alpha r),
$$
for some real numbers $c$ and $d$ (since $R$ and $R'$ are bounded around
$r=0$).  Hence the Dirichlet boundary conditions from \eqref{serge5}
yield
\begin{equation*}
  \begin{cases}
    c+d J_{|k|}(\alpha)=0,\\[1\jot]
    c |k|+d \alpha  J'_{|k|}(\alpha )=0.
  \end{cases}
\end{equation*}
This $2\times 2$ system has a non trivial solution if and only if its
determinant is zero, namely
\begin{equation}\label{D}
  \alpha  J'_{|k|}(\alpha)-|k|J_{|k|}(\alpha)=0.
\end{equation}
If a solution $\alpha$ exists (see Lemma~\ref{lserge2} below), then
$R$ has the form
\begin{equation*}
  R(r)=d \bigl(-J_{|k|}(\alpha )r^{|k|}+ J_{|k|}(\alpha r)\bigr),
\end{equation*}
for some $d\ne 0$.

\begin{Lem}
\label{lserge2}
  For all $k\in \IN$, there exists an increasing sequence
  $\alpha_{k,\ell}>0$, with $\ell\in {\IN}^{*}$, of solutions
  to~\eqref{D}. This sequence is formed by the positive zeros of~$J_{k+1}$.
\end{Lem}

\begin{proof}
  For $k\geq 0$, we use the formula \eqref{A:4}
  to find that
  $$
  \alpha  J'_{k}(\alpha )-k   J_{k}(\alpha )=-\alpha J_{k+1}(\alpha ).
  $$
  Therefore $\alpha>0$ is a solution of \eqref{D} if and only if
  $J_{k+1}(\alpha )=0$.
\end{proof}

We are ready to state the following result.

\begin{Thm}
  \label{tserge2}
  The spectrum of the buckling problem with $\kappa =0$ is given by
  $\bigl\{ \lambda_{k,\ell} := j_{|k|+1,\ell}^2 \bigm| \ell\in {\IN}^{*},\
  k\in \IZ\bigr\}$.
  A basis of the eigenfunctions is given by
  \begin{equation*}
    (r,\theta) \mapsto
    -J_{0}(j_{1,\ell}) + J_{0}(j_{1,\ell} \, r)
  \end{equation*}
  giving rise to the eigenvalue $\lambda_{0,\ell}$ and
  \begin{equation*}
    (r,\theta) \mapsto
    \bigl(-J_{|k|}(j_{|k|+1,\ell})r^{|k|}
    + J_{|k|}(j_{|k|+1,\ell} \, r)\bigr)
    \e^{\i k\theta},
    \quad k \ne 0,
  \end{equation*}
  giving rise to the eigenvalue $\lambda_{k,\ell}$.
\end{Thm}

\begin{proof}
  We have already showed that all $j_{|k|+1,\ell}^2$ are
  eigenvalues of the operator with the corresponding eigenvectors. It
  then remains to prove that we have found all eigenvalues.  The reason
  comes essentially from the fact that the functions
  $(\e^{\i k\theta})_{k\in \IZ}$ form an orthonormal basis of
  $L^2\bigl(\intervaloo{0,2\pi}\bigr)$.
  Indeed let $(u,\lambda)$ be a solution to~\eqref{serge5}.
  We write
  $$
  u=\sum_{k\in \IZ} u_k(r) \e^{\i k\theta},
  $$
  with
  $$
  \forall r>0,\quad
  u_k(r)=\int_0^{2\pi}  u(r,\theta) \e^{-\i k\theta} \intd\theta
  $$
  (this integral makes sense because $u$ is smooth).
  Now we check that
  \begin{equation}\label{serge6}
    \forall k\in \IZ,\quad
    L(\partial_r, r, \sqrt\lambda, 0, \abs{k}) u_k=0.
  \end{equation}
  Indeed using the differential equation in \eqref{serge5}, we see that
  \begin{equation*}
    \forall r>0,\quad
    0=\int_0^{2\pi} (\Delta^2+\lambda \Delta)
    u(r,\theta) \cdot \e^{-\i k\theta} \intd\theta.
  \end{equation*}
  Writing the operator $\Delta$ and $\Delta^2$ in polar coordinates
  and integrating by parts in $\theta$, we see that the previous
  identity is equivalent to \eqref{serge6}. At this stage we use
  point~2 of Lemma~\ref{lserge1} to deduce that $u_k$ is a linear
  combination of $r^{|k|}$ and of $J_{|k|}(\sqrt{\lambda} \, r)$ (due to the
  regularity of $u_k$ at $r=0$). 
	
	We therefore deduce that $\lambda$
  has to be a root of \eqref{D} and hence $\lambda=\lambda_{k,\ell}$,
  for some $k\in \mathbb Z$ and $\ell\in \IN^{*}$,
  and $u$ is a linear combination of the eigenfunctions
  given in the statement of this Theorem.
\end{proof}

\section{Eigenvalues in the case $\kappa >0$}
\label{Sect3}

In this section we characterize the eigenfunctions of the buckling
problem with $\kappa >0$ on the unit disk~$D_1$.  In other words, we
look for a non-trivial $u$ and $\lambda>0$ solving~\eqref{eq:pbm}.

First observe that $\lambda_1 \ge 2\kappa$.
As $\int_{D_1}(\Delta u + \kappa u)^2\ge 0$, we have
$$
\int_{D_1}(|\Delta u|^2+\kappa^2 u^2)\geq - 2 \kappa \int_{D_1} \Delta u \, u = 2 \kappa \int_{D_1} |\nabla u|^2.
$$
This implies that
$$
\lambda_1 =
\inf_{u\in H^2_0(\Omega)\setminus\{0\}} \frac{\int_{D_1}(|\Delta u|^2+\kappa^2 u^2)}{
  \int_{D_1} |\nabla u|^2}\geq 2\kappa.
$$
As a consequence, we can write~\eqref{eq:pbm} under the form
\eqref{EDP} with $\alpha$ and $\beta$ positive real numbers
satisfying $\alpha\beta=\kappa$ and $\alpha^2+\beta^2=\lambda$.
Following the same strategy as before, we look for solutions
$u=R(r) \e^{\i k\theta}$ with $k\in \IZ$.  Again due to the regularity of
$u$ at zero and eliminating $\beta={\kappa}/{\alpha}$, we deduce that,
if $\alpha\ne {\kappa}/{\alpha}$,
$R$ is in the form
\begin{equation}
  \label{eq R} 
  R(r) = c J_{|k|}(\alpha r)
  + dJ_{|k|}\Bigl(\frac{\kappa}{\alpha} r \Bigr),
\end{equation}
for some $c,d\in \IR$.  If instead
$\alpha = {\kappa}/{\alpha}$ (i.e., $\alpha = \sqrt{\kappa}$),
\begin{equation}
  \label{eq:R:2}
  R(r)=c J_{|k|}(\sqrt{\kappa} r)
  + d rJ'_{|k|}(\sqrt{\kappa} r),
\end{equation}
with $c,d\in \IR$.

\begin{Lem}
  \label{cserge1}%
  Let $k \in \IN$.  
	
	\begin{enumerate}
	
	\item
	The function
  \begin{math}
    \displaystyle
    \tilde H_k :
    \intervaloo{0, +\infty} \to \IR
  \end{math}
	defined by
	\begin{equation}
    \label{eq:tilde H}
    \tilde H_k(z) := (z^2 - k^2) \bigl(J_k(z)\bigr)^2
    + z^2 \bigl(J_k'(z)\bigr)^2
  \end{equation}
  is positive and increasing.  
  \item
    The function
    \begin{equation*}
      H_k :
      \intervaloo{0, +\infty}\setminus
      \bigl\{j_{k,\ell} \bigm| \ell\in \IN^{*}\bigr\} \to \IR :
      z \mapsto \frac{z J'_k(z)}{J_k(z)}
    \end{equation*}
    has a negative derivative 
  \begin{equation}
    \label{eq:tilde Hbis}
    H'_k(z) = \frac{-\tilde H_k(z)}{z J_k^2(z)}
  \end{equation}
  and thus
  is decreasing between any two consecutive roots of $J_k$.
  Moreover, for any $\ell \ge 1$, 
    \begin{equation*}
      \lim_{z \xrightarrow{>} j_{k,\ell}\hspace{-2ex}} H_k(z) = +\infty
      \qquad \text{and} \qquad
      \lim_{z \xrightarrow{<} j_{k,\ell+1}\hspace{-3.7ex}}
      H_k(z) = -\infty.
  \end{equation*}
  \end{enumerate}
\end{Lem}

\begin{proof}
  Let $\tilde H_k$ be defined by \eqref{eq:tilde H}. 
	 Differentiating $\tilde H_k$ and using  the equation
  satisfied by Bessel functions \eqref{A:6} gives $\tilde H_k'(z) = 2 z J_k^2(z)$
  which is positive for all $z > 0$ except at the (isolated)
  roots of $J_k$. Since  $\tilde H_k(0) = 0$,
  this proves the result concerning $\tilde H_k$. 

	Using again the differential equation satisfied by Bessel
  functions \eqref{A:6}, one easily gets \eqref{eq:tilde Hbis}.
  Since $\tilde H_k>0$, 
  the function $H_k$ decreases between two consecutive roots of $J_k$.
  Hence the limits are easy to compute once
  one remarks that the numerator of $H_k(z)$ does not vanish at $z =
  j_{k,m}$ for any~$m$ because the positive roots of $J_k$ are simple.
\end{proof}

\begin{Prop}
  \label{eigenfunctions}
  The eigenfunctions of the differential equation~\eqref{eq:pbm} are
  of the form $u=R(r) \e^{\i k\theta}$ with $k\in \IZ$ and $R$ given by
  \eqref{eq R}, where
  $\alpha \ne \sqrt{\kappa}$ and $\alpha$ is a positive solution of
  \begin{equation}\label{Dbis}
    F_k(\alpha) :=
    \frac{\kappa}{\alpha} J_{|k|}(\alpha)
    J'_{|k|} \Bigl(\frac{\kappa}{\alpha} \Bigr)
    - \alpha J_{|k|}\Bigl(\frac{\kappa}{\alpha}\Bigr)  J'_{|k|}(\alpha )
    =0.
  \end{equation}
  The corresponding eigenvalue is $\lambda = \alpha^2 +
  {\kappa^2}/{\alpha^2}$.
\end{Prop}

\begin{proof}
  First observe that,
in the case $\alpha = \sqrt{\kappa}$, there exists a non-trivial
function of the form~\eqref{eq:R:2} satisfying the boundary
conditions at $r=1$ if
and only if the system
\begin{equation*}
  \begin{cases}
    c\,J_{\abs{k}}(\sqrt{\kappa})+d \, J_{\abs{k}}'(\sqrt{\kappa})=0
    \\[0.5\jot]
    c\,\sqrt{\kappa}\,J_{\abs{k}}'(\sqrt{\kappa})
    + d \bigl(J_{\abs{k}}'(\sqrt{\kappa})
    + \sqrt{\kappa}\, J_{\abs{k}}''(\sqrt{\kappa})\bigr)=0
  \end{cases}
\end{equation*}
has a non trivial solution $(c,d)$. This holds if and only if
$\alpha = \sqrt{\kappa}$ is a solution of
\begin{equation}\label{Dter}
  D_k(\alpha) :=
  J_{|k|}(\alpha)  \bigl(J'_{|k|}(\alpha)
  + \alpha \, J''_{|k|}(\alpha) \bigr)
  - \alpha \bigl(J'_{|k|}(\alpha)\bigr)^2 =0.
\end{equation}
Note that, for all $\alpha > 0$, using  the equation \eqref{A:6}
  satisfied by Bessel functions,
\begin{math}
  D_k(\alpha) = -\frac{1}{\alpha} \tilde H_{\abs{k}}(\alpha) < 0
\end{math}
where $\tilde H_{\abs{k}}$ is defined by~\eqref{eq:tilde H}.
Consequently \eqref{Dter} possesses no solution $\alpha > 0$.
\medskip

In the case~\eqref{eq R}, the boundary conditions at $r=1$
lead to the system
\begin{equation}
  \label{eq:R bd cond}
  \begin{cases}
    c\, J_{|k|}(\alpha)
    + d\, J_{|k|}\bigl(\frac{\kappa}{\alpha} \bigr) = 0,
    \\[1\jot]
    c \,\alpha J'_{|k|}(\alpha )
    + d\, \frac{\kappa}{\alpha}
    J'_{|k|}\bigl(\frac{\kappa}{\alpha} \bigr) = 0.
  \end{cases}
\end{equation}
This $2\times 2$ system has a non-trivial solution if and only if its
determinant is equal to zero, namely if and only if \eqref{Dbis} is satisfied.

The same arguments than the ones used in Theorem \ref{tserge2}
allow to conclude that no other eigenvalues exist.
\end{proof}

\begin{Thm}
  \label{thm:alpha kl}
  For all $k\in \IN$ and $\kappa>0$,
  the roots of~$F_k$ (defined by~\eqref{Dbis}\/) can be ordered as
  an increasing sequence $\alpha_{k,\ell} = \alpha_{k,\ell}(\kappa) > 0$,
  with $\ell\in {\IZ}$, such that
  \begin{gather*}
    \forall \ell \ge 0,\qquad
    \alpha_{k,-\ell}=\frac{\kappa}{\alpha_{k,\ell}},\\
    \alpha_{k,0}=\sqrt{\kappa}
    \quad \text{and}\quad
    \forall \ell > 0,\quad
    \alpha_{k,\ell}>\sqrt{\kappa}>\alpha_{k,-\ell} \, ,\\
    \alpha_{k,\ell}\to +\infty \text{ as } \ell\to +\infty,\\
    \alpha_{k,\ell}\to 0 \text{ as } \ell\to -\infty.
  \end{gather*}
  Each $\ell \ne 0$ gives rise to the eigenvalue
  \begin{equation}
    \label{rel_lambda_alpha}
    \lambda_{k,\ell} = \alpha_{k,\ell}^2 +
    \frac{\kappa^2}{\alpha_{k,\ell}^2} = \alpha_{k,\ell}^2
    + \alpha_{k,-\ell}^2 \, ,
\end{equation}
  and a corresponding eigenfunction
  of the form $R_{k,\ell}(r)\e^{\i k\theta}$ with
  \begin{equation*}
    R_{k,\ell}(r)=cJ_k(\alpha_{k,\ell}\, r)+ d J_k(\alpha_{k,-\ell}\, r),
  \end{equation*}
  and $c, d$ solutions to~\eqref{eq:R bd cond} with $\alpha=\alpha_{k,\ell}$.
\end{Thm}

\begin{proof}
  First notice that
  \begin{equation*}
    \forall \alpha > 0,\quad
    F_k\Bigl(\frac{\kappa}{\alpha}\Bigr) = -F_k(\alpha),
  \end{equation*}
	where $F_k$ is defined in \eqref{Dbis}.
  As a consequence, $F_k(\sqrt\kappa) = 0$ and we set $\alpha_{k,0} :=
  \sqrt\kappa$.  Moreover it suffices to
  find the roots of $F_k$ in $\intervaloo{\sqrt\kappa, +\infty}$.  The function
  $F_k$ being continuous on $\intervaloo{0,\infty}$, it will possess infinitely
  many roots provided it changes sign infinitely
  many times when $\alpha \to +\infty$.

  Formula \eqref{A:4}  implies
  that,
  \begin{equation*}
    F_k(\alpha)
    = \alpha J_k\Bigl(\frac{\kappa}{\alpha}\Bigr) J_{k+1}(\alpha)
    - \frac{\kappa}{\alpha}J_k(\alpha)
    J_{k+1}\Bigl(\frac{\kappa}{\alpha}\Bigr).
  \end{equation*}
  Hence, noting that $\kappa/\alpha = o(1)$, if $\alpha\to\infty$, formulas \eqref{A:7} and \eqref{A:8} imply that
  \begin{equation*}
    F_k(\alpha)
    = \sqrt{\frac{2\alpha}{\pi}}\frac{1}{k!}
    \Bigl(\frac{\kappa}{2\alpha}\Bigr)^k
    \Bigl(
    \cos\bigl( \alpha - \tfrac{2k+3}{4} \pi \bigr)
    + o(1) \Bigr)
    \quad \text{as } \alpha \to +\infty.
  \end{equation*}
  Thus $F_k$ oscillates an infinite number of times
  as $\alpha \to +\infty$.
  This yields the sequence of $\alpha_{k,\ell} > 0$ with
  $\ell > 0$.

  Observe that the only possible accumulation points are $0$ and
  $+\infty$ as otherwise the corresponding eigenvalues
  $\lambda_{k,\ell} = \alpha_{k,\ell}^2 + \kappa^2 / \alpha_{k,\ell}^2$
  would have a finite accumulation point which contradicts the
  variational theory of eigenvalues.
\end{proof}

In order to better understand the behaviour of the eigenvalues and of the
corresponding eigenfunctions, we will now study the functions
$\alpha_{k,\ell}$.

\begin{Lem}
  \label{alpha increase}%
  For all $k \in \IN$ and $\ell \in \IZ$, the function $\alpha_{k,\ell}: \intervaloo{0,
    +\infty} \to \IR : \kappa \mapsto \alpha_{k,\ell}(\kappa)$ is of
  class $\C^1$ and $\partial_\kappa \alpha_{k,\ell} > 0$.
\end{Lem}

\begin{proof}
  Let us note $F_k(\alpha, \kappa)$ the function $F_k(\alpha)$
  defined by~\eqref{Dbis}
  where we have explicited the dependence on $\kappa$.  
  The assertion
  will result from the Implicit Function Theorem.  Let us fix
  $k \in \IN$, $\kappa^* > 0$
  and $\alpha^* = \alpha_{k,\ell}(\kappa^*) > 0$ and distinguish two cases.
  \begin{itemize}
  \item If $J_k(\alpha^*) = 0$ (resp.\
    $J_k(\frac{\kappa^*}{\alpha^*}) = 0$) then
    $J_k'(\alpha^*) \ne 0$ (resp.\
    $J_k'(\frac{\kappa^*}{\alpha^*}) \ne 0$) because the roots of
    the Bessel functions are simple.  But then, the fact that
    $F_k(\alpha^*, \kappa^*) = 0$ implies that
    $J_k(\frac{\kappa^*}{\alpha^*}) = 0$ (resp.\
    $J_k(\alpha^*) = 0$).  A direct computation,
    using the fact that both $J_k(\alpha^*)$ and
    $J_k(\frac{\kappa^*}{\alpha^*})$ vanish, shows
    \begin{align*}
      \partial_\kappa F_k(\alpha^*, \kappa^*)
      &= - J'_k(\alpha^*) J'_k\Bigl(\frac{\kappa^*}{\alpha^*}\Bigr),
      \\
      \partial_\alpha F_k(\alpha^*, \kappa^*)
      &= 2 \frac{\kappa^*}{\alpha^*} J'_k(\alpha^*)
      J'_k\Bigl(\frac{\kappa^*}{\alpha^*}\Bigr)
      \ne 0.
    \end{align*}
    \sloppy %
    Therefore the Implicit Function Theorem implies that there exists
		$\C^1$ curve $\beta_{\ell}$ around $\kappa^*$ such that, in a neighbourhood of
		$(\alpha^*,\kappa^*)$, 
		$$
		F_k(\alpha,\kappa)=0 \text{ if and only if }\alpha=\beta_{\ell}(\kappa).
		$$
		Moreover
    $$
		\partial_\kappa \beta_{\ell}(\kappa^*) = - \frac{\partial_\kappa
    F_k(\alpha^*, \kappa^*) }{ \partial_\alpha F_k(\alpha^*, \kappa^*)}
    = \frac{\alpha^* }{ 2 \kappa^*} > 0.
		$$
  \item Let us now suppose that $J_k(\alpha^*) \ne 0$ and
    $J_k(\frac{\kappa^*}{\alpha^*}) \ne 0$.  Around such $(\alpha^*,
    \kappa^*)$, one can write
    \begin{equation*}
      F_k(\alpha, \kappa)
      = J_k(\alpha) J_k\Bigl(\frac{\kappa}{\alpha}\Bigr)
      \tilde F_k(\alpha, \kappa)
      \quad\text{with}\quad
      \tilde F_k(\alpha, \kappa)
      := H_k\Bigl(\frac{\kappa}{\alpha}\Bigr) - H_k(\alpha),
    \end{equation*}
    where $H_k$ is defined in Lemma~\ref{cserge1}.  Using
    Lemma~\ref{cserge1}, one deduces
    \begin{align*}
      \partial_\kappa \tilde F_k(\alpha, \kappa)
      &= \frac{1}{\alpha} H'_k\Bigl(\frac{\kappa}{\alpha}\Bigr) < 0,
      \\[1\jot]
      \partial_\alpha \tilde F_k(\alpha, \kappa)
      &= - \frac{\kappa}{\alpha^2}
      H'_k\Bigl(\frac{\kappa}{\alpha}\Bigr)
      - H'_k(\alpha) > 0.
    \end{align*}
    Therefore the Implicit Function Theorem applies to $\tilde F_k$
    and there exists a
    $\C^1$ curve $\beta_{\ell}$ defined around $\kappa^*$ such that,
    in a neighbourhood of~$(\alpha^*,\kappa^*)$, 
		$$
		F_k(\alpha,\kappa)=0 \text{ if and only if }\alpha=\beta_{\ell}(\kappa).
		$$
		Moreover
    $$
		\partial_\kappa \beta_{\ell}(\kappa^*) =
        - \frac{\partial_\kappa
      \tilde F_k(\alpha^*, \kappa^*)}{\mathstrut\partial_\alpha \tilde
      F_k(\alpha^*, \kappa^*)}  > 0.
			$$
  \end{itemize}
  This argument can be done for all $\ell$.
  Thus, for all $\ell$, we have a $\C^1$-curve emanating 
	from $\alpha_{k,\ell}(\kappa^*)$ such that, in a neighbourhood $U_{\ell}$ of $(\alpha_{k,\ell}(\kappa^*), \kappa^*)$,
	$$
	F_k(\alpha,\kappa)=0 \text{ if and only if }\alpha=\beta_{\ell}(\kappa).
	$$
  Moreover, as $F_k(\alpha, \kappa^*)\not=0$ for $\alpha\notin
  \{\alpha_{k,\ell}(\kappa^*)\mid \ell\in \IZ\}$,
  the continuity of $F_k$ implies the existence of
  a neighbourhood $V_{\ell}$ of
  $\{(\alpha,\kappa^*)\mid
  \alpha_{k,\ell-1}(\kappa^*)<\alpha<\alpha_{k,\ell}(\kappa^*),\linebreak[2]\
  (\alpha,\kappa^*)\notin U_{\ell-1}\cup U_{\ell}\}$ such that
  $F_k(\alpha,\kappa)\ne 0$ for $(\alpha, \kappa)\in V_{\ell}$.
  In this way,
  one shows that there is a
  neighbourhood $V$ of $[\sqrt{\kappa^*}, \alpha^*]$ and $W$ of
  $\kappa^*$ such that, for all $(\alpha, \kappa) \in V \times W$,
  \begin{equation*}
    F_k(\alpha, \kappa) = 0
    \quad\text{if and only if}\quad
    \alpha=\beta_{\ell'}(\kappa) \text{ for some }0 \le \ell' \le \ell.    
  \end{equation*}
		Shrinking $W$ if
  necessary, one can assume the curves
  $\beta_0, \beta_1,\dotsc, \beta_\ell$ do not cross each other.  For
  any given $\kappa \in W$, it then suffices to count the number of
  curves one meets to reach the one emanating from $(\alpha^*,
  \kappa^*)$ starting with $\alpha = \alpha_{k,0}(\kappa) =
  \sqrt{\kappa}$ to establish that
  \begin{equation*}
    \forall \kappa \in V,\quad
    \beta_\ell(\kappa) = \alpha_{k,\ell} (\kappa),
  \end{equation*}
  whence the desired result.
\end{proof}

\begin{Lem}
  \label{kappa->0}
  Let $k \in \IN$ and $\ell > 0$.
  As $\kappa \to 0$,  $\alpha_{k,\ell}(\kappa) \to j_{k+1,\ell}$.
  Consequently
  $\alpha_{k,-\ell}(\kappa) = \kappa / \alpha_{k,\ell}(\kappa) \to 0$ if $\kappa\to0$.
\end{Lem}

\begin{proof}
Without loss of generality, we can
  restrict $\kappa$ to $\intervaloo{0, j_{k,1}^2}$ so that, as 
	we will only consider $\alpha > \sqrt{\kappa}$, we have $\kappa /
  \alpha < \sqrt{\kappa} < j_{k,1}$ and thus $J_k(\kappa / \alpha) \ne
  0$.  For such $\kappa$, one also has that $J_k(\alpha_{k,\ell}) \ne
  0$ (otherwise that would imply $J_k(\kappa / \alpha_{k,\ell} ) = 0$, see the
  proof of Lemma~\ref{alpha increase}) and so $\alpha_{k,\ell}(\kappa)
  \ne j_{k,m}$ for any $m \ge 1$.  
	
  According to formulas \eqref{A:8} and \eqref{A:4}, one has as $z
  \to 0$,
  $$
	\begin{array}{cl}
	\displaystyle
    J_k(z) = \frac{1+o(1)}{k!} (\tfrac{1}{2} z)^k ,&
    \\[3mm]
  \displaystyle
    J'_k(z) = - J_{k+1}(z) + \frac{k}{z} J_k(z)
    = \frac{1+o(1)}{2 (k-1)!} (\tfrac{1}{2} z)^{k-1}
    &\qquad \text{if } k\not=0,
    \\[3mm]
    \displaystyle
    J_0'(z)=-(1+o(1))\frac{z}{2}.&
  \end{array}
	$$
	This implies that 
	\begin{equation}
	\label{Hk}
	\lim_{z \to 0} H_k(z) = k
	\end{equation}
  where $H_k$ is defined in Lemma \ref{cserge1},
	and so, restricting further $\kappa$, one can assume
  $H_k(\kappa / \alpha)$ is bounded (as $0<\kappa / \alpha < \sqrt{\kappa}$).
    
  Let us start by showing that
  \begin{equation*}
    j_{k,1} < \alpha_{k,1}(\kappa) < j_{k,2}.
  \end{equation*}
  As $\sqrt{\kappa} < \alpha_{k,1}(\kappa)$, in order to establish the left inequality, it suffices to show
  that for all $\alpha \in \intervaloo{\sqrt{\kappa}, j_{k,1}}$,
  $F_k(\alpha) \ne 0$ ($\alpha_{k,1} \ne j_{k,1}$ was established
  above).  But, for
  $\alpha$  below the first root of $J_k$, $F_k(\alpha) \ne 0$ is
  equivalent to $\tilde F_k(\alpha, \kappa) \ne 0$ where $\tilde F_k$,
  defined in the proof of Lemma~\ref{alpha increase}, is a smooth
  function on $\intervaloo{\sqrt{\kappa}, j_{k,1}}$.  The argument is complete if one recalls
  that $\partial_\alpha \tilde F_k(\alpha, \kappa) > 0$ and that
  $\tilde F_k(\sqrt{\kappa}, \kappa) = 0$.  
	
	For the right inequality,
  we first notice that the boundedness of $H_k(\kappa / \alpha)$ and
  Lemma~\ref{cserge1} imply
  \begin{equation*}
    \lim_{\alpha \xrightarrow{>} j_{k,1}} \tilde F_k(\alpha, \kappa) = -\infty
    \quad\text{and}\quad
    \lim_{\alpha \xrightarrow{<} j_{k,2}} \tilde F_k(\alpha, \kappa) = +\infty.
  \end{equation*}
  By continuity and monotonicity, $\tilde F_k(\alpha, \kappa)$ must
  possess a unique zero $\alpha \in \intervaloo{j_{k,1}, j_{k,2}}$.
  Since we are between two consecutive roots of $J_k$, that implies
  $F_k(\alpha) = 0$ and thus the desired inequality by definition
  of~$\alpha_{k,1}$. 
\medbreak

  The same reasoning applies to $\alpha \mapsto \tilde F_k(\alpha,
  \kappa)$ on the interval $\intervaloo{j_{k,\ell}, j_{k,\ell+1}}$, $\ell \ge
  2$, thereby proving the existence of a unique root of $F_k$ in that
  interval.  Counting the number of roots below shows that this root
  is nothing but $\alpha_{k,\ell}$ therefore establishing that
  \begin{equation*}
    j_{k,\ell} < \alpha_{k,\ell}(\kappa) < j_{k,\ell+1}.
  \end{equation*}

  \sloppy %
  Now let us pass to the limit $\kappa \to 0$.  Given that
  $\alpha_{k,\ell}$ is increasing and bounded from below by a positive
  constant, we have $\alpha_{k,\ell}^* := \lim_{\kappa \xrightarrow{>} 0}
  \alpha_{k,\ell}(\kappa) \in \intervalco{j_{k,\ell}, j_{k,\ell+1}}$.
  Notice that, as $J_k\bigl(\frac{\kappa}{\alpha_{k,\ell}}\bigr)\ne 0$,
  the equation
  $F_k(\alpha_{k,\ell}) = 0$ can be rewritten as
  \begin{equation*}
    J_k(\alpha_{k,\ell}) H_k\Bigl(\frac{\kappa}{\alpha_{k,\ell}}\Bigr)
    - \alpha_{k,\ell} J'_k(\alpha_{k,\ell}) = 0.
  \end{equation*}
  Passing to the limit $\kappa \to 0$ in this equation yields, by \eqref{Hk},
  \begin{math}
    J_k(\alpha_{k,\ell}^*) k
    - \alpha_{k,\ell}^*  J'_k(\alpha_{k,\ell}^*) = 0
  \end{math}
  or equivalently, by formula \eqref{A:4},
  $J_{k+1}(\alpha_{k,\ell}^*) = 0$.  As $j_{k,\ell} \le
  \alpha_{k,\ell}^* < j_{k,\ell+1}$, the interlacing property of the
  zeros of Bessel functions (see e.g. \eqref{A:9})
  implies that $\alpha_{k,\ell}^* = j_{k+1,\ell}$.
\end{proof}

\begin{Lem}
  \label{infinity}
  For all $k\in \IN$ and all $\ell\in \IZ$, we have
  $\displaystyle\lim_{\kappa\to\infty}\alpha_{k,\ell}(\kappa)=+\infty$.
\end{Lem}

\begin{proof}
  This is obvious for $\ell \ge 0$ because $\alpha_{k,\ell}(\kappa)
  \ge \sqrt{\kappa}$.
  Assume on the contrary that there exists $\ell > 0$ such that
  $\lim_{\kappa\to\infty}\alpha_{k,-\ell}(\kappa)<+\infty$
  (recall that $\alpha_{k,-\ell}$ is increasing).  Hence,
  there exists $\kappa^*>0$ such that, for all $\kappa>\kappa^*$,
  $\alpha_{k,{-\ell}}(\kappa)$ lies between two consecutive roots of
  $J_k$ and of $J_{k}'$, i.e., $J_k(\alpha_{k,{-\ell}}(\kappa))\not=0$
  and $J_{k}'(\alpha_{k,{-\ell}}(\kappa))\ne 0$.  Because the roots of
  $J_k$ are simple, \eqref{Dbis} implies that, for all $\kappa>\kappa^*$,
  $J_k(\alpha_{k,{\ell}}(\kappa))\not=0$ and
  $J_{k+1}(\alpha_{k,{\ell}}(\kappa))\ne 0$
  (recall that $\alpha_{k,{\ell}}(\kappa)=\kappa/\alpha_{k,{-\ell}}(\kappa)$).
  This contradicts the fact that $\alpha_{k,{\ell}}$ crosses infinitely
  many roots of $J_k$ because
  $\alpha_{k,{\ell}}$ is continuous and
  $\alpha_{k,{\ell}}(\kappa) \xrightarrow[\kappa \to \infty]{} +\infty$.
\end{proof}

\medbreak

As shown in Figure~\ref{figure 1} and~\ref{figure 2}, the curves
$\alpha_{k,\ell}$ and
$\alpha_{k+1,\ell}$ cross each other.
In Proposition \ref{intersa01,a11} we will characterize their
intersection points.  This will be done in several steps given by the
following lemmas.

\begin{figure}[ht]
  \centering
  \begin{tikzpicture}[x=0.5ex, y=2.5ex]
    \draw[->] (0,0) -- (100, 0) node[below]{$\kappa$};
    \draw[->] (0,0) -- (0, 13.5) node[left]{$\alpha$};
    \draw[color=gray, dotted, thick] plot file{data/sqrt_kappa.dat}
    node[right]{$\sqrt{\kappa}$};
    \draw[darkred,fill] (0, 3.8317) circle(1pt) node[left]{$j_{1,1}$};
    \draw[darkred,fill] (0, 7.0155) circle(1pt) node[left]{$j_{1,2}$};
    \draw[darkred,fill] (0, 10.1734) circle(1pt) node[left]{$j_{1,3}$};
    \draw[darkblue,fill] (0, 5.1356) circle(1pt) node[left]{$j_{2,1}$};      
    \draw[darkblue,fill] (0, 8.4172) circle(1pt) node[left]{$j_{2,2}$};      
    \draw[darkblue,fill] (0, 11.619841) circle(1pt) node[left]{$j_{2,3}$};      
    \clip (0,0) rectangle (100, 13); 
    \begin{scope}[color=darkred]
      \draw plot file{data/alpha_0_-1.dat};
      \node[rotate=20] at (35, 5.1) {$\alpha_{0,-1}$};
      \draw plot file{data/alpha_0_1.dat};
      \node at (11, 4) {$\alpha_{0,1}$};
      \draw plot file{data/alpha_0_2.dat};
      \node at (22, 7.8) {$\alpha_{0,2}$};
      \draw plot file{data/alpha_0_3.dat};
      \node at (30, 11) {$\alpha_{0,3}$};
    \end{scope}
    \begin{scope}[color=darkblue]
      \draw plot file{data/alpha_1_-1.dat};
      \node[rotate=20] at (40, 4.3) {$\alpha_{1,-1}$};
      \draw plot file{data/alpha_1_1.dat};
      \node at (5.3, 5.6) {$\alpha_{1,1}$};
      \draw plot file{data/alpha_1_2.dat};
      \node at (9, 8.8) {$\alpha_{1,2}$};
      \draw plot file{data/alpha_1_3.dat};
      \node at (15, 12) {$\alpha_{1,3}$};
    \end{scope}
  \end{tikzpicture}
  \caption{Graphs of $\alpha_{k,\ell}$}
  \label{figure 1}
\end{figure}
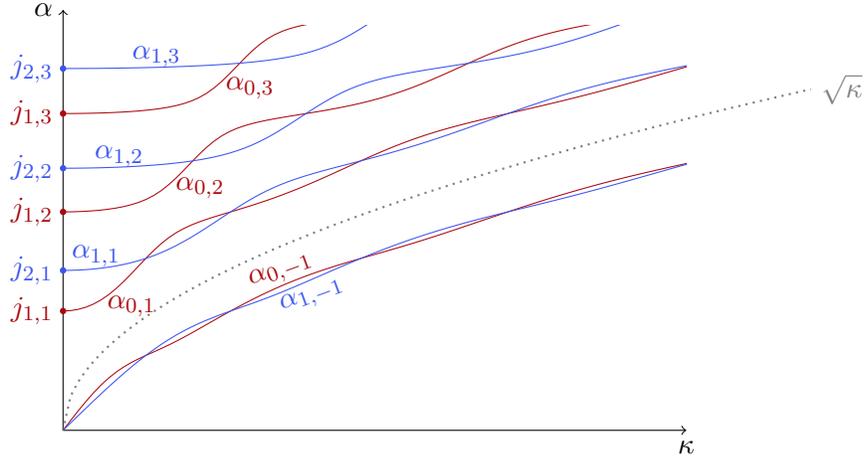

\begin{Lem}
  \label{common root}
  Let $k \in \IN$.  The functions $F_k$ and $F_{k+1}$ have a common
  positive root $\alpha \ne \sqrt{\kappa}$ if and only if
  there exists positive integers $m$ and $n$ such that $m \ne n$ and
  \begin{itemize}
  \item $\alpha = j_{k,n}$ and $\kappa / \alpha = j_{k,m}$
    (thus $\kappa = j_{k,m} \, j_{k,n}$), or
  \item $\alpha = j_{k+1,n}$ and $\kappa / \alpha = j_{k+1,m}$
    (thus $\kappa = j_{k+1,m} \, j_{k+1,n}$).
  \end{itemize}
\end{Lem}

\begin{proof}
  First recall that, using the identity \eqref{A:4}, we find that
  \begin{equation}
    \label{Fkequiv}
    F_k(\alpha)
    = \alpha J_{k}\Bigl(\frac{\kappa}{\alpha}\Bigr)  J_{k+1}(\alpha )
    -\frac{\kappa}{\alpha} J_{k}(\alpha)
    J_{k+1} \Bigl(\frac{\kappa}{\alpha} \Bigr).
  \end{equation}
  If instead one uses the identity \eqref{A:3}, we have
  \begin{equation}
    \label{Fkequiv2}
    F_{k+1}(\alpha)
    = \frac{\kappa}{\alpha}
    J_{k}\Bigl(\frac{\kappa}{\alpha}\Bigr) J_{k+1}(\alpha )
    - \alpha J_{k}(\alpha) J_{k+1} \Bigl(\frac{\kappa}{\alpha} \Bigr).
  \end{equation}

  \noindent
  ($\Leftarrow$)\ If 
	$\alpha =  j_{k,n}$ and  $\kappa / \alpha = j_{k,m}$, one easily see
  that $F_k(\alpha) = 0 = F_{k+1}(\alpha)$.
  A similar argument establish this implication when $\alpha =  j_{k+1,n}$ and  $\kappa / \alpha = j_{k+1,m}$.

  \smallskip\noindent %
  ($\Rightarrow$) Now let us prove that, if $F_k(\alpha) = 0 =
  F_{k+1}(\alpha)$ for some $0 < \alpha \ne \sqrt{\kappa}$, then
  $\alpha$ and $\kappa/\alpha$ have the desired values.  In view of
  \eqref{Fkequiv}--\eqref{Fkequiv2}, if $F_k(\alpha) = 0$, one can
  write
  \begin{equation*}
    0 =
    F_{k+1}(\alpha) = \frac{\kappa^2 - \alpha^4}{\kappa \alpha}
    J_{k+1}(\alpha) J_{k}\Bigl(\frac{\kappa}{\alpha}\Bigr).
  \end{equation*}
  Two cases can occur:
  \begin{itemize}
  \item $\alpha$ is a root of $J_{k+1}$, i.e., $\alpha = j_{k+1,n}$ for
    some $n$.  Since the zeros of $J_k$ and $J_{k+1}$
    interlace (see \eqref{A:9}), $J_k(\alpha) \ne 0$.
    Then, using the fact that $F_k(\alpha) = 0$, one deduces that
    $\kappa / \alpha$ is also a root of $J_{k+1}$, say $j_{k+1,m}$ for
    some $m$.  As $\alpha \ne \sqrt{\kappa}$, one has $n \ne m$.
  \item $\kappa / \alpha$ is a root of $J_k$.  A reasoning similar to
    the first case then shows that $\alpha$ is also a root of $J_k$
    and the conclusion readily follows.
    \qedhere
  \end{itemize}
\end{proof}

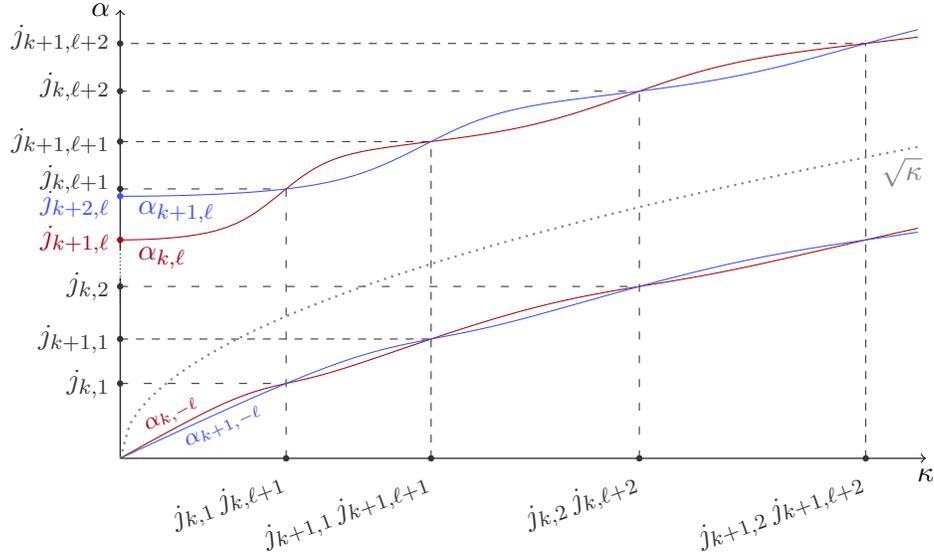
\begin{figure}
  \centering
  \begin{tikzpicture}[x=0.64ex, y=2.5ex]
    \draw[->] (0,0) -- (101, 0) node[below]{$\kappa$};
    \draw (0,0) -- (0, 5.8);
    \draw[densely dotted] (0, 5) -- (0, 7);
    \draw[->] (0, 6.7) -- (0, 14.4) node[left]{$\alpha$};
    \begin{scope}[color=black!80]
      \draw[fill] (0, 2.4048) circle(1pt);
      \draw[loosely dashed] (0,2.4048) node[left]{$j_{k,1}$} -- +(20.8107,0);
      \draw[fill] (0, 5.52008) circle(1pt);
      \draw[loosely dashed] (0,5.52008) node[left]{$j_{k,2}$} -- +(65.0902,0);

      \draw[fill] (20.8107, 0) circle(1pt);
      \draw[fill] (0, 8.65372) circle(1pt);
      \draw[loosely dashed] (20.8107, 0)
      node[below left, xshift=1ex]{\rotatebox{20}{$j_{k,1}\,j_{k,\ell+1}$}}
      -- (20.8107, 8.65372);
      \draw[loosely dashed]
      (0, 8.65372) node[left, yshift=0.8ex]{$j_{k,\ell+1}$} -- +(20.8107, 0);

      \draw[fill] (65.0902, 0) circle(1pt);
      \draw[fill] (0, 11.79153) circle(1pt);
      \draw[loosely dashed] (65.0902, 0)
      node[below left, xshift=1ex]{\rotatebox{20}{$j_{k,2}\,j_{k,\ell+2}$}}
      -- (65.0902, 11.79153)
      -- (0, 11.79153) node[left, yshift=0.4ex]{$j_{k,\ell+2}$};
    \end{scope}
    \begin{scope}[color=black!80]
      \draw[fill] (0, 3.8317) circle(1pt);
      \draw[dashed] (0, 3.8317) node[left]{$j_{k+1,1}$} -- +(38.9817, 0);

      \draw[fill] (38.9817, 0) circle(1pt);
      \draw[fill] (0, 10.17347) circle(1pt);
      \draw[dashed] (38.9817, 0)
      node[below left, xshift=1ex]{\rotatebox{20}{$j_{k+1,1}\,j_{k+1,\ell+1}$}}
      -- (38.9817, 10.17347);
      \draw[dashed]
      (0, 10.17347) node[left]{$j_{k+1,\ell+1}$} -- +(38.9817, 0);

      \draw[fill] (93.47352, 0) circle(1pt);
      \draw[fill] (0, 13.32369) circle(1pt);
      \draw[dashed] (93.47352, 0)
      node[below left, xshift=1ex]{\rotatebox{20}{$j_{k+1,2}\,j_{k+1,\ell+2}$}}
      -- (93.47352, 13.32369)
      -- (0, 13.32369) node[left, yshift=0.4ex]{$j_{k+1,\ell+2}$};
    \end{scope}
    \draw[darkred,fill] (0, 7.01559) circle(1pt) node[left]{$j_{k+1,\ell}$};
    \draw[darkblue,fill] (0, 8.417244) circle(1pt)
    node[left, yshift=-0.7ex]{$j_{k+2,\ell}$};
    \node[color=gray, below] at (98, 10.) {$\sqrt{\kappa}$};
    \clip (0,0) rectangle (100, 14);
    \draw[color=gray, dotted, thick] plot file{data/sqrt_kappa.dat};
    \begin{scope}[color=darkred]
      \draw plot file{data/alpha_0_2.dat};
      \node[below right] at (1, 7.1) {$\alpha_{k,\ell}$};
      \draw plot file{data/alpha_0_-2.dat};
      \node[rotate=28] at (6.5, 1.4) {\footnotesize $\alpha_{k,-\ell}$};
    \end{scope}
    \begin{scope}[color=darkblue]
      \draw plot file{data/alpha_1_2.dat};
      \node[right] at (1, 7.9) {$\alpha_{k+1,\ell}$};
      \draw plot file{data/alpha_1_-2.dat};
      \node[rotate=25] at (13, 1) {\footnotesize $\alpha_{k+1,-\ell}$};
    \end{scope}
  \end{tikzpicture}
  \vspace{-1ex}%
  \caption{Mapping of intervals by $\alpha_{k,\ell}$.}
  \label{figure 2}
  \label{fig:map-alpha-kl}
\end{figure}

\begin{Lem}
  \label{valalpha}
  For all $k\in \IN$, $\ell \ge 1$ and  $n\ge 1$, we have
  \begin{alignat*}{2}
    \alpha_{k,\ell}(j_{k,n}\,j_{k,\ell+n}) &= j_{k,\ell+n}
    &&= \alpha_{k+1,\ell}(j_{k,n}\,j_{k,\ell+n}),\\
    \alpha_{k,-\ell}(j_{k,n}\,j_{k,\ell+n}) & = j_{k,n}
    &&= \alpha_{k+1,-\ell}(j_{k,n}\,j_{k,\ell+n}),
    \\
    \alpha_{k,\ell}(j_{k+1,n}\,j_{k+1,\ell+n})  & = j_{k+1,\ell+n}
    &&= \alpha_{k+1,\ell}(j_{k+1,n}\,j_{k+1,\ell+n}),\\
    \alpha_{k,-\ell}(j_{k+1,n}\,j_{k+1,\ell+n}) &= j_{k+1,n}
    &&= \alpha_{k+1,-\ell}(j_{k+1,n}\,j_{k+1,\ell+n}) .
  \end{alignat*}
\end{Lem}

\begin{proof}
  Let us first deal with $\alpha_{k,\pm\ell}$ (left equalities).
  As $\alpha_{k,\ell}$ is continuous, increasing and satisfies
  $\alpha_{k,\ell}(\intervaloo{0,+\infty}) = \intervaloo{j_{k+1,\ell},+\infty}$,
  there exists an increasing sequence $(\kappa_{n})_{n \ge 1}$
  such that, for all~$n \ge 1$,
  \begin{equation*}
    \alpha_{k,\ell}(\kappa_{2n-1}) = j_{k,\ell+n}
    \quad\text{and}\quad
    \alpha_{k,\ell}(\kappa_{2n}) = j_{k+1,\ell+n}.
  \end{equation*}
  In the same way, since $\alpha_{k,-\ell}$ is increasing and
  $\alpha_{k,-\ell}(\intervaloo{0,+\infty}) = \intervaloo{0,+\infty}$,
  there exists an increasing sequence $(\tilde\kappa_{n})_{n \ge
    1}$ such that, for all~$n \ge 1$,
  \begin{equation*}
    \alpha_{k,-\ell}(\tilde\kappa_{2n-1})=j_{k,n}
    \quad\text{and}\quad
    \alpha_{k,-l}(\tilde\kappa_{2n})=j_{k+1,n}.
  \end{equation*}
  By~\eqref{Fkequiv}, for all roots $\alpha$ of $F_k$, $J_k(\alpha) = 0$ if and only if $J_k(\kappa/\alpha) = 0$, and so
  $\{\kappa_{2n-1}\mid n\ge 1\} = \{\tilde\kappa_{2n-1}\mid n\ge 1\}$.
  Similarly, $\{\kappa_{2n}\mid n\ge 1\} = \{\tilde\kappa_{2n}\mid
  n\ge 1\}$.  Since the sequences are increasing, we conclude that,
  for all $n \ge 1$, $\kappa_{n} = \tilde\kappa_{n}$.  Moreover, we
  deduce from $\kappa = \alpha_{k,-\ell}(\kappa)
  \alpha_{k,\ell}(\kappa)$ that, for all $n\ge 1$,
  $\kappa_{2n-1} = j_{k,n}\,j_{k,\ell+n}$ and
  $\kappa_{2n} = j_{k+1,n}\,j_{k+1,\ell+n}$.
This proves the left hand equalities.
  \medbreak

  To prove the equalities to the right, let us first show that
  $j_{k+2,\ell} < j_{k,\ell+1}$.
  For this, it is enough to prove that
  $J_{k+2}(j_{k,\ell'})$ and $J_{k+2}(j_{k,\ell'+1})$ have opposite
  signs for any $\ell' = 1,\dotsc, \ell$.
  Using the relations \eqref{A:1} and \eqref{A:4}, we obtain
  \begin{equation*}
    \begin{split}
      J_{k+2}(z)
      &= \frac{2(k+1)}{z} J_{k+1}(z)-J_k(z)   \\
      &= \frac{2(k+1)}{z} \Bigl( \frac{k}{z} J_k(z)
      - J_k'(z) \Bigr) - J_k(z)\\
      &= \Bigl( \frac{2(k+1)k}{z^2} - 1 \Bigr) J_k(z)
      - \frac{2(k+1)}{z} J_k'(z).
    \end{split}
  \end{equation*}
  Therefore
  $$
    J_{k+2}(j_{k,\ell'}) \, J_{k+2}(j_{k,\ell'+1})
    = 4(k+1)^2  
    \frac{J_k'(j_{k,\ell'}) J_k'(j_{k,\ell'+1})}{j_{k,\ell'} \, j_{k,\ell'+1}}
  $$
	which is negative because $j_{k,\ell'}$ and $j_{k,\ell'+1}$ are two
  consecutive simple roots of $J_k$.
	
  In a similar way to the first part, as $\alpha_{k+1,\ell}$ and
  $\alpha_{k+1,-\ell}$ are increasing and satisfy
  $\alpha_{k+1,\ell}(\intervaloo{0,+\infty}) =
  \intervaloo{j_{k+2,\ell},+\infty}$ and
  $\alpha_{k+1,-\ell}(\intervaloo{0,+\infty}) =
  \intervaloo{0,+\infty}$ and as
  $j_{k,\ell+1} > j_{k+2,\ell} > j_{k+1,\ell} > j_{k,\ell}$,
  there exist increasing sequences
  $(\kappa_{n})_{n \ge 1}$ and $(\tilde\kappa_{n})_{n \ge 1}$ such
  that, for all $n \ge 1$,
  \begin{align*}
    \alpha_{k+1,\ell}(\kappa_{2n-1}) &= j_{k,\ell+n}
    &&\text{and}&
    \alpha_{k+1,\ell}(\kappa_{2n}) &= j_{k+1,\ell+n},
    \\
    \alpha_{k+1,-\ell}(\tilde\kappa_{2n-1}) &= j_{k,n}
    &&\text{and}&
    \alpha_{k+1,-\ell}(\tilde\kappa_{2n}) &= j_{k+1,n}.
  \end{align*} 
  Now, using \eqref{Fkequiv2} and arguing as above,
  we obtain that
  $\{\kappa_{2n-1}\mid n\ge 1\} = \{\tilde\kappa_{2n-1}\mid n\ge 1\}$
  and $\{\kappa_{2n}\mid n\geq1\}=\{\tilde\kappa_{2n}\mid n\geq1\}$.
  As the sequences are increasing, one deduces that
  $\kappa_{n}=\tilde\kappa_{n}$  for all $n$.
  Finally, from
  $\kappa = \alpha_{k+1,\ell}(\kappa) \alpha_{k+1,-\ell}(\kappa)$, one
  gets that,
  for all $n\ge 1$, $\kappa_{2n-1} = j_{k,n}\,j_{k,\ell+n}$ and
  $\kappa_{2n} = j_{k+1,n}\,j_{k+1,\ell+n}$.
\end{proof}

\begin{Rem}
  \label{interlace k k+2}
  This proof establishes that the positive roots of $J_k$ and
  $J_{k+2}$ interlace (see also~\cite[Theorem~1]{Palmai2012}), namely
  \begin{equation*}
    \forall \ell \ge 1,\qquad
    j_{k,\ell} < j_{k+2,\ell} < j_{k,\ell+1}
  \end{equation*}
  (the first inequality comes from $j_{k,\ell} < j_{k+1, \ell}
  < j_{k+2,\ell}$).
\end{Rem}

\begin{Rem} 
  \label{Range of alpha_k1}
  As a byproduct of the proof, one gets that, for $\ell \ge 1$
  and $n \ge 1$,
  \begin{equation}
    \label{eq:ineq prod jkl}
    j_{k,n} j_{k,\ell+n} < j_{k+1,n} j_{k+1,\ell+n}
    < j_{k,n+1} j_{k,\ell+n+1} 
    .
  \end{equation}
  Since the functions $\alpha_{k,\ell}$ are increasing, one
  immediately deduces that
  \begin{align}
    \alpha_{k,\ell}\bigl(\intervaloo{0,j_{k,1}j_{k,\ell+1}}\bigr)
    &= \intervaloo{j_{k+1,\ell}, j_{k,\ell+1} } ,
    \notag\\
		\alpha_{k+1,\ell}\bigl(\intervaloo{0,j_{k,1}j_{k,\ell+1}}\bigr)
    &= \intervaloo{j_{k+2,\ell}, j_{k,\ell+1} } ,
    \displaybreak[1]
    \notag\\
    \alpha_{k,\ell}(I)
    = \alpha_{k+1,\ell}(I)
    &= \intervaloo{j_{k,\ell+n}, j_{k+1,\ell+n}} \notag\\
    &\qquad\text{where } I =
    \intervaloo{j_{k,n} j_{k,\ell+n},\, j_{k+1,n}j_{k+1,\ell+n}} ,
    \notag\\
    \alpha_{k,\ell}(I)
    = \alpha_{k+1,\ell}(I)
    &= \intervaloo{j_{k+1,\ell+n}, j_{k,\ell+n+1}}  \notag\\
    &\qquad\text{where } I =
    \intervaloo{j_{k+1,n} j_{k+1,\ell+n},\, j_{k,n+1} j_{k,\ell+n+1}}  .
    \notag \\
    \intertext{%
      Note further that these properties also yield
      (see figure~\ref{figure 2})}
    %
    \label{serge:25/02:1}
    \alpha_{k,\ell}(I)
    = \alpha_{k+1,\ell}(I)
    &= \intervaloo{j_{k,\ell+n}, j_{k,\ell+n+1}}\\
    &\qquad\text{where } I = \intervaloo{j_{k,n} j_{k,\ell+n},\,
      j_{k,n+1} j_{k,\ell+n+1}}   \notag
  \end{align}
  for all $n \ge 0$, with the convention that $j_{k,0} := 0$.
In the same way, we have
  \begin{align}
    \alpha_{k,-\ell}\bigl(\intervaloo{0,j_{k,1}j_{k,\ell+1}}\bigr)
    &= \intervaloo{0, j_{k,1} } ,
    \notag\\
		\alpha_{k+1,-\ell}\bigl(\intervaloo{0,j_{k,1}j_{k,\ell+1}}\bigr)
    &= \intervaloo{0, j_{k,1} } ,
    \displaybreak[1] \notag\\
    \alpha_{k,-\ell}(I)
    = \alpha_{k+1,-\ell}(I)
    &= \intervaloo{j_{k,n}, j_{k+1,n}}  \notag\\
    &\qquad\text{where } I =
    \intervaloo{j_{k,n} j_{k,\ell+n},\, j_{k+1,n}j_{k+1,\ell+n}} ,
    \displaybreak[1] \notag\\
    \alpha_{k,-\ell}(I)
    = \alpha_{k+1,-\ell}(I)
    &= \intervaloo{j_{k+1,n}, j_{k,n+1}} \notag\\
    &\qquad\text{where } I =
    \intervaloo{j_{k+1,n} j_{k+1,\ell+n},\, j_{k,n+1} j_{k,\ell+n+1}}  .
    \notag\\
    \intertext{and, in particular, for all $n \ge 0$,}
    %
    \label{serge:25/02:2}
    \alpha_{k,-\ell}(I)
    = \alpha_{k+1,-\ell}(I)
    &= \intervaloo{j_{k,n}, j_{k,n+1}} \\
    &\qquad\text{where } I = \intervaloo{j_{k,n} j_{k,\ell+n},\,
      j_{k,n+1} j_{k,\ell+n+1}}.  \notag
  \end{align}
\end{Rem}

\begin{Prop}
  \label{intersa01,a11}
  Let $k \in \IN$ and $\ell \ge 1$.
  The set of $\kappa$ such that
  $\alpha_{k,\ell} (\kappa) = \alpha_{k+1,\ell} (\kappa)$ is
  \begin{equation*}
    \{j_{k,n}j_{k,\ell +n} \mid n\ge 1\}
    \cup
    \{j_{k+1,n}j_{k+1,\ell+n}\mid n\ge 1\}.
  \end{equation*}
\end{Prop}

\begin{proof}  
  By Lemma \ref{valalpha}, we know that the elements of the set
  $\{j_{k,n}j_{k,\ell +n} \mid n\ge 1\}  \cup
  \{j_{k+1,n}j_{k+1,\ell+n}\mid n\ge 1\}$ are equality points
  of $\alpha_{k,\ell}$ and $\alpha_{k+1,\ell}$.
\medbreak

  Let us prove that there is no other point where $\alpha_{k,\ell}=
  \alpha_{k+1, \ell}$.  Lemma~\ref{common root} implies that, if $\bar
  \kappa$ is such a point, then $\alpha_{k,\ell}(\bar \kappa) =
  \alpha_{k+1,\ell}(\bar \kappa) = j_{k,m}$ or $\alpha_{k,\ell}(\bar \kappa) =
  \alpha_{k+1,\ell}(\bar \kappa) = j_{k+1,m}$ for some positive integer $m$.
  In either case, $m > \ell$ because, by Lemmas \ref{alpha increase} and \ref{kappa->0}, $\alpha_{k+1,\ell} >
  j_{k+2,\ell} > j_{k+1,\ell} > j_{k,\ell}$.
  Since $\alpha_{k,\ell}$ is increasing, hence injective,
  $\bar\kappa$ must then necessarily be one of the values given by
  Lemma~\ref{valalpha}.
\end{proof}

We will need also the following result in the next section.

\begin{Lem}
  \label{lserge25:02}%
  Let $k\in \IN$ and $\ell \in \IN \setminus\{0\}$ be fixed.
  Then the function
  $\displaystyle
  g_{k,\ell}:\intervaloo{0,+\infty} \to \IR: \kappa \mapsto
  \frac{\alpha_{k,\ell}^2(\kappa)}{\kappa}$
  is decreasing.
\end{Lem}

\begin{proof}
  Direct calculations show
  \begin{math}
    g_{k,\ell}'(\kappa) = \frac{\alpha_{k,\ell}(\kappa)}{\kappa^2}
    \bigl(2\kappa \alpha_{k,\ell}'(\kappa) - \alpha_{k,\ell}(\kappa)\bigr).
  \end{math}
  Set
  \begin{equation*}
    G(\kappa)
    := 2\kappa \alpha_{k,\ell}'(\kappa)-\alpha_{k,\ell}(\kappa).
  \end{equation*}
  Because $g_{k,\ell}$ is continuous and the intervals
  $\intervaloo{j_{k,n} j_{k,\ell+n},\, j_{k,n+1} j_{k,\ell+n+1}}$
  cover $\intervaloo{0, +\infty}$ except for isolated points,
  it is sufficient to show that, for all $n \ge 0$,
  \begin{equation*}
    \forall \kappa\in \intervaloo{j_{k,n} j_{k,\ell+n},\,
      j_{k,n+1} j_{k,\ell+n+1}},\qquad
    G(\kappa) < 0
  \end{equation*}
  with the convention that $j_{k,0} := 0$.
  For $\kappa\in \intervaloo{j_{k,n} j_{k,\ell+n},\,
    j_{k,n+1} j_{k,\ell+n+1}}$, by \eqref{serge:25/02:1},
  $\alpha_{k,\ell}(\kappa)$ is not a root of $J_k$
  and by the proof of Lemma \ref{alpha increase}, we deduce that
  \begin{equation*}
    G(\kappa)
    = \alpha
    \frac{\frac{\kappa}{\alpha} H_k'(\frac{\kappa}{\alpha})
      -\alpha H_k'(\alpha)}{
      \frac{\kappa}{\alpha} H_k'(\frac{\kappa}{\alpha})+\alpha H_k'(\alpha)},
  \end{equation*}
  where for shortness we have written $\alpha$ instead of
  $\alpha_{k,\ell}(\kappa)$.  Again the fact that
  $\alpha_{k,\ell}(\kappa)$ and
  $\frac{\kappa}{\alpha}=\alpha_{k,-\ell}(\kappa)$ are between two
  consecutive roots of $J_k$ allows to apply Lemma \ref{cserge1} and
  deduce that the above denominator is negative. Hence it remains to
  check that
  \begin{equation*}
    N(\kappa) : =\frac{\kappa}{\alpha}
    H_k'\Bigl(\frac{\kappa}{\alpha}\Bigr)
    - \alpha H_k'(\alpha)  >0,
  \end{equation*}
  for all $\kappa\in \intervaloo{j_{k,n} j_{k,\ell+n},\, j_{k,n+1}
    j_{k,\ell+n+1}}$.  With the help of the identities \eqref{eq:tilde
    H} and \eqref{eq:tilde Hbis}, we may transform $N$ into
  \begin{multline*}
    N(\kappa)
    = \frac{1}{ J_k^2(\alpha) J_k^2(\frac{\kappa}{\alpha})}
    \Bigl[ \Bigl(\alpha^2-\frac{\kappa^2}{\alpha^2} \Bigr)
    J_k^2(\alpha) J_k^2\Bigl(\frac{\kappa}{\alpha}\Bigr)  \\
    + \alpha^2 \left(J_k'(\alpha)\right)^2
    J_k^2\Bigl(\frac{\kappa}{\alpha}\Bigr)
    - \frac{\kappa^2}{\alpha^2}
    \Bigl(J_k'\Bigl(\frac{\kappa}{\alpha}\Bigr)\Bigr)^2
    J_k^2(\alpha) \Bigr].
  \end{multline*}
  As $\alpha$ is a root of \eqref{Dbis}, we arrive at
  \begin{equation*}
    N(\kappa) = \alpha^2-\frac{\kappa^2}{\alpha^2} ,
  \end{equation*}
  which is positive since $\alpha_{k,\ell}(\kappa)>\sqrt{\kappa}$
  when $\ell>0$.
\end{proof}

\begin{Rem}
  Note that, if $\kappa = j_{k,n} j_{k,\ell+n}$ for some $n>0$,
  the second case of Lemma~\ref{valalpha} implies that
  $\alpha_{k,\ell}(\kappa)$ is a root of $J_k$ and so,
  using the proof of Lemma \ref{alpha increase}, we infer
  \begin{equation*}
    \alpha_{k,\ell}'(\kappa)
    = \frac{\alpha_{k,\ell}(\kappa)}{2\kappa},
  \end{equation*}
  which means that $G(\kappa) = 0$.
\end{Rem}

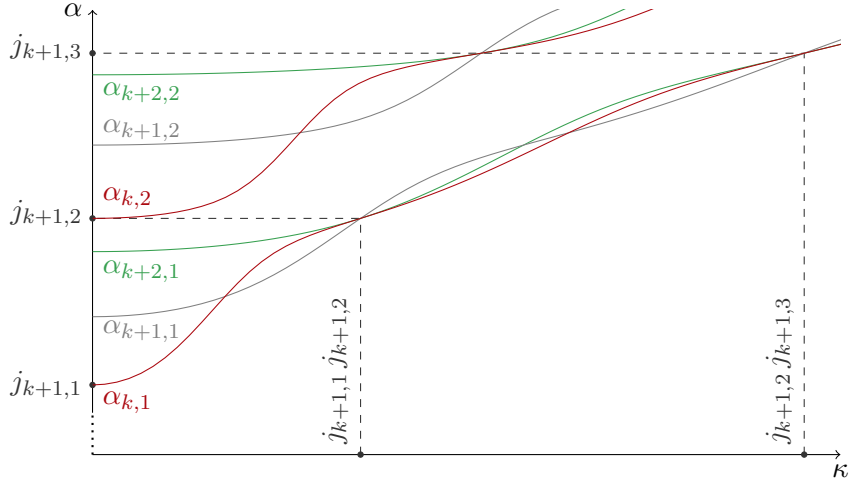
\begin{figure}[ht]
  \centering
  \newcommand{\alphamin}{2.5}
  \begin{tikzpicture}[x=0.8ex, y=4.2ex]
    \draw[->] (0, \alphamin) -- +(75, 0) node[below]{$\kappa$};
    \draw[dotted, thick] (0,\alphamin) -- (0,3.3);
    \draw[->] (0,3.3) -- (0, 11) node[left]{$\alpha$};
    \begin{scope}[color=black!80]
      \draw[fill] (0, 3.831707) circle(1pt) node[left]{$j_{k+1,1}$};

      \draw[fill] (26.88166, \alphamin) circle(1pt);
      \draw[fill] (0, 7.01558) circle(1pt);
      \draw[dashed] (26.88166, \alphamin)
      node[above left]{\rotatebox{90}{$j_{k+1,1}\,j_{k+1,2}$}}
      -- (26.88166, 7.01558)
      -- (0, 7.01558) node[left, yshift=0.4ex]{$j_{k+1,2}$};

      \draw[fill] (71.372847, \alphamin) circle(1pt);
      \draw[fill] (0, 10.173468) circle(1pt);
      \draw[dashed] (71.372847, \alphamin)
      node[above left]{\rotatebox{90}{$j_{k+1,2}\,j_{k+1,3}$}}
      -- (71.372847, 10.173468)
      -- (0, 10.173468) node[left, yshift=0.4ex]{$j_{k+1,3}$};
    \end{scope}
    \clip (0, \alphamin) rectangle (75, 11); 
    \begin{scope}[color=darkgreen]
      \draw plot file{data/alpha_2_1.dat};
      \node[right] at (0, 6.) {$\alpha_{k+2,1}$};
      \draw plot file{data/alpha_2_2.dat};
      \node[below right] at (0, 9.8) {$\alpha_{k+2,2}$};
    \end{scope}
    \begin{scope}[color=gray]
      \draw plot file{data/alpha_1_1.dat};
      \node[right] at (0, 4.85) {$\alpha_{k+1,1}$};
      \draw plot file{data/alpha_1_2.dat};
      \node[right] at (0, 8.7) {$\alpha_{k+1,2}$};
    \end{scope}
    \begin{scope}[color=darkred] 
      \draw plot file{data/alpha_0_1.dat};
      \node[right] at (0, 3.5) {$\alpha_{k,1}$};
      \draw plot file{data/alpha_0_2.dat};
      \node[right] at (0, 7.4) {$\alpha_{k,2}$};
    \end{scope}
  \end{tikzpicture}
  \caption{Illustration of $\alpha_{k,1} \le \alpha_{k+2,1}$}
  \label{fig:alpha-kl-order}
\end{figure}

Until now we have proved that the $\ell$-th curve corresponding to $k$
and $k+1$ cross each other, in particular the first ones, and we have
characterized the crossing points. In Proposition
\ref{mina01 ou a11} we will prove that the first eigenvalue
$\lambda_1(\kappa)$ corresponds to $\min\{\alpha_{0,1}(\kappa),
\alpha_{1,1}(\kappa)\}$ by the relation \eqref{rel_lambda_alpha}. To
this aim, we first prove that the other curves are above these
first two.

\begin{Prop}
  \label{alpha k k+2}
  Let $k \in \IN$.  For all $\kappa > 0$, $\alpha_{k,1}(\kappa) \le
  \alpha_{k+2,1}(\kappa)$.  Moreover, this inequality is an equality if
  and only if $\kappa = j_{k+1,n} j_{k+1,n+1}$ for some $n \ge 1$, in
  which case $\alpha_{k,1}(\kappa)
  = \alpha_{k+1,1}(\kappa) = \alpha_{k+2,1}(\kappa) = j_{k+1,n+1}$.
\end{Prop}

\begin{proof}
  Observe first that, by the proof of Lemma \ref{alpha increase},
  $\partial_{\alpha}F_k(\sqrt{\kappa})>0$
  (for the second case in this proof, one has that
  $\partial_{\alpha} F_k(\sqrt{\kappa})
  = J_k^2(\sqrt{\kappa}) \partial_\alpha
  \tilde F_k(\sqrt{\kappa}, \kappa) >0$)
  and hence, we
  know that, for $\alpha>\sqrt{\kappa}$ close to $\sqrt{\kappa}$,
  $F_k(\alpha)>0$.
  To establish that $\alpha_{k,1}(\kappa) \le \alpha_{k+2,1}(\kappa)$,
  it suffices to show that $F_k(\alpha_{k+2,1})\le 0$.  Indeed, this
  implies by the intermediate value theorem that $F_k(\cdot)=0$ has a
  solution in $\intervaloo{\sqrt{\kappa}, \alpha_{k+2,1}(\kappa)}$,
  i.e., $\sqrt{\kappa}< \alpha_{k,1}(\kappa)\leq
  \alpha_{k+2,1}(\kappa)$.

  Using formula~\eqref{Fkequiv2} for $F_{k+2}$ in which one substitutes
  $J_{k+2}$ according to the formula \eqref{A:1} and then
  using again \eqref{Fkequiv}, we find
  \begin{align}
    F_{k+2}(\alpha)
    &= \frac{\kappa}{\alpha}
    J_{k+1}\Bigl(\frac{\kappa}{\alpha}\Bigr) J_{k+2}(\alpha )
    - \alpha J_{k+1}(\alpha) J_{k+2} \Bigl(\frac{\kappa}{\alpha} \Bigr)
    \notag\\
    &= 2(k+1) \frac{\kappa^2 - \alpha^4}{\alpha^2 \kappa}
    J_{k+1}(\alpha) J_{k+1}\Bigl(\frac{\kappa}{\alpha}\Bigr)
    + F_k(\alpha).
    \label{eq:F k+2 k}
  \end{align}
  Therefore, recalling that $\kappa / \alpha_{k+2,1} = \alpha_{k+2,-1}$ and $F_{k+2}(\alpha_{k+2,1})=0$, we have
  \begin{equation*}
    F_k(\alpha_{k+2,1})
    = 2(k+1) \frac{\alpha_{k+2,1}^4 - \kappa^2}{\alpha_{k+2,1}^2 \,\kappa}
    J_{k+1}(\alpha_{k+2,1}) J_{k+1}(\alpha_{k+2,-1}) .
  \end{equation*}
  Observe that the fraction is positive since $\alpha_{k+2,1} >
  \sqrt{\kappa}$.  Moreover, by~\eqref{serge:25/02:1}
  and~\eqref{serge:25/02:2}, for all $n\in
  \IN$ and all $\kappa\in
  \intervaloo{j_{k+1,n}j_{k+1,n+1},\ j_{k+1,n+1}j_{k+1,n+2}}$,
  we have $\alpha_{k+2,1}\in
  \intervaloo{j_{k+1,n+1},j_{k+1,n+2}}$ and $\alpha_{k+2,-1}\in
  \intervaloo{j_{k+1,n},j_{k+1,n+1}}$. This implies that
  $J_{k+1}(\alpha_{k+2,1}) J_{k+1}(\alpha_{k+2,-1})
  < 0$
  and thus that $F_k(\alpha_{k+2,1}) \le 0$ for all $\kappa > 0$
  which proves the first part of the result.
\medbreak

  For the second part of the statement, it is clear from
  Lemma~\ref{valalpha} that, when $\kappa = j_{k+1,n} j_{k+1,n+1}$,
  $\alpha_{k,1}(\kappa)
  = \alpha_{k+1,1}(\kappa)
  = j_{k+1,n+1}
  = \alpha_{k+2,1}(\kappa)$.
  Let us
  show this is the only possibility.  Suppose $\kappa > 0$ is such
  that $\alpha := \alpha_{k,1}(\kappa) = \alpha_{k+2,1}(\kappa)$.  In
  view of equation~\eqref{eq:F k+2 k}, $\alpha$ or $\kappa / \alpha$
  is a root of $J_{k+1}$.  In the latter case, using~\eqref{Fkequiv}
  and $F_k(\alpha) = 0$, one deduces that $\alpha$ must also be a root
  of $J_{k+1}$.  Thus, in both cases, $\alpha_{k,1}(\kappa) =
  \alpha_{k+2,1}(\kappa) = j_{k+1,m}$ for some $m \ge 1$.
  In fact $m \ge 2$ because $\alpha_{k+2,1} > j_{k+3,1} > j_{k+1,1}$.
  Then, the injectivity of $\alpha_{k,1}$ and Lemma~\ref{valalpha}
  imply that $\kappa$ has the desired form.
\end{proof}

\begin{Prop}
  \label{mina01 ou a11}
  For all $\kappa>0$, we have
  \begin{equation}
    \label{eq:min alpha}
    \Bar\alpha(\kappa) :=
    \min\bigl\{\alpha_{k,\ell}(\kappa) \bigm| k\in \IN, \ell\ge 1 \bigr\}
    = \min\bigl\{\alpha_{0,1}(\kappa), \alpha_{1,1}(\kappa) \bigr\}.
  \end{equation}
  Moreover, the first eigenvalue $\lambda_1(\kappa)$ is given by
  $\lambda_1(\kappa) = \Bar\alpha^2(\kappa) + \kappa^2 /
  \Bar\alpha^2(\kappa)$.
\end{Prop}

\begin{proof}
  It is obvious that
  $$
  \min\{\alpha_{k,\ell} \mid k\in \IN,\ \ell\geq 1\}
  = \min\{\alpha_{k,1} \mid k\in \IN\}.
  $$
  Since Proposition~\ref{alpha k k+2} asserts that, for all $k\in
  \IN$, $\alpha_{k+2,1} \ge \alpha_{k,1}$, \eqref{eq:min alpha}
  is established.  To conclude the proof it suffices to recall
  the relation~\eqref{rel_lambda_alpha} and to notice that
  $\alpha \mapsto \alpha^2 + \kappa^2/\alpha^2$ is increasing
  on $\intervaloo{\sqrt{\kappa}, +\infty}$
  (hence the smallest eigenvalue corresponds to the
  smallest~$\alpha_{k,\ell}$).
\end{proof}

In Theorem \ref{1st eigen} we will prove that the first eigenvalue of
\eqref{eq:pbm} given by Proposition~\ref{mina01 ou a11} comes alternatively
from $\alpha_{0,1}$ and $\alpha_{1,1}$. To this aim, it remains to
prove that, as shown in Figure~\ref{figure 1}, the curves
$\alpha_{0,1}$ and $\alpha_{1,1}$ cross each other non-tangentially.

\begin{Lem}
  \label{der inter}
  Let $k \in \IN$ and $n\ge 1$.
  \begin{itemize}
  \item
    If $\kappa = j_{k,n}j_{k,n+1}$, then
    $\partial_\kappa \alpha_{k,1}(\kappa)
    > \partial_\kappa\alpha_{k+1,1}(\kappa)$.
  \item  If $\kappa = j_{k+1,n}j_{k+1,n+1}$, then
    $\partial_\kappa\alpha_{k,1}(\kappa)
    < \partial_\kappa\alpha_{k+1,1}(\kappa)$.
  \end{itemize}
\end{Lem}

\begin{proof}
  We will use the computations in the proof of Lemma~\ref{alpha
    increase} to evaluate the derivatives of $\alpha_{k,1}$
  and $\alpha_{k+1,1}$.
  Recall that, for $ n\geq 1$, we have
  $\alpha_{k,1}(j_{k,n}j_{k,n+1}) =
  \alpha_{k+1,1}(j_{k,n}j_{k,n+1}) = j_{k,n+1}$.  On one hand,
  the first case in the proof of Lemma~\ref{alpha increase} implies
  $\partial_\kappa\alpha_{k,1}(j_{k,n}j_{k,n+1}) =
  j_{k,n+1}/(2 j_{k,n}j_{k,n+1})= (2 j_{k,n})^{-1}$.  On
  the other hand, for the derivative of $\alpha_{k+1,1}$ at
  $j_{k,n}j_{k,n+1}$, we are in the second case
  of the proof of Lemma~\ref{alpha increase} and
  \begin{equation*}
    \partial_\kappa\alpha_{k+1,1}(j_{k,n}j_{k,n+1})
    = \frac{\frac{1}{j_{k,n+1}} H_{k+1}'(j_{k,n})}{
      \frac{j_{k,n}}{j_{k,n+1}} H_{k+1}'(j_{k,n}) + H_{k+1}'(j_{k,n+1})}.  
  \end{equation*}
  By  Lemma \ref{cserge1}, we know that, for all $\bar n \ge 1$,
  \begin{equation*}
    H_{k+1}'(j_{k, \bar n})
    = \frac{-\tilde H_{k+1}(j_{k, \bar n})}{
      j_{k, \bar n} \, J_{k+1}^2(j_{k, \bar n})},    
  \end{equation*}
  and
  \begin{equation*}
    \tilde H_{k+1}(j_{k, \bar n})
    = \bigl(j_{k, \bar n}^2- (k+1)^2\bigr)  J_{k+1}^2(j_{k, \bar n})
    + j_{k, \bar n}^2 \bigl(J_{k+1}'(j_{k, \bar n})\bigr)^2.
  \end{equation*}
  As, by~\eqref{A:3}, $J_{k+1}'(j_{k, \bar n}) = J_k(j_{k, \bar n}) - \frac{k+1}{j_{k,\bar n}}
  J_{k+1}(j_{k, \bar n}) = -\frac{k+1}{j_{k, \bar n}} J_{k+1}(j_{k,\bar n})$,
  we deduce
  that
  \begin{equation*}
    \tilde H_{k+1}(j_{k, \bar n})
    = j_{k, \bar n}^2 J_{k+1}^2(j_{k, \bar n})
    \qquad\text{and, finally,}\qquad
    H_{k+1}'(j_{k, \bar n})= - j_{k, \bar n}.
  \end{equation*}
  This implies that
  \begin{equation*}
    \partial_\kappa\alpha_{k+1,1}(j_{k,n}j_{k,n+1})
    = \frac{\frac{- j_{k,n}}{j_{k,n+1}}}{
      \frac{-j_{k,n}^2}{j_{k,n+1}}- j_{k,n+1}}
    = \frac{ j_{k,n}}{j_{k,n}^2+ j_{k,n+1}^2}.    
  \end{equation*}
  We can then conclude that 
  \begin{equation*}
    \partial_\kappa\alpha_{k+1,1}(j_{k,n}j_{k,n+1})
    = \frac{j_{k,n}}{j_{k,n}^2+ j_{k,n+1}^2}
    < \frac{j_{k,n}}{2j_{k,n}^2}
    = \partial_\kappa\alpha_{k,1}(j_{k,n}j_{k,n+1}).    
  \end{equation*}
  The argument is similar if $\kappa\in \{j_{k+1,n}j_{k+1,n+1}\mid
  n\geq 1\}$.
\end{proof}

Now we can give our first two main results which characterize the
first eigenvalue and the first eigenspace with respect to the value of
$\kappa$.  In Theorem~\ref{1st eigen}, we deal with the case
$\alpha_{0,1} \ne \alpha_{1,1}$ while the case $\alpha_{0,1} =
\alpha_{1,1}$ is considered in Theorem~\ref{1st eigen2}.

\begin{Thm}
  \label{1st eigen}
  Denote $R_{k,\ell}$ the function defined by equation~\eqref{eq R} with
  $\alpha = \alpha_{k,\ell}$ given by Theorem~\ref{thm:alpha kl} and
  $(c,d)$ being a non-zero element of the
  one dimensional space of solutions to~\eqref{eq:R bd cond}.

  For all $\kappa\in \intervalco{0,j_{0,1}j_{0,2}} \cup  \bigcup_{n\ge 1}
  \intervaloo{j_{1,n}j_{1,n+1}, \,j_{0,n+1}j_{0,n+2}}$,
  the first eigenvalue is given by $\lambda_1(\kappa) =
  \alpha_{0,1}^2(\kappa) + \kappa^2 / \alpha_{0,1}^2(\kappa)$
  and the eigenfunctions are multiples~of
  \begin{equation*}
    x \mapsto R_{0,1}(\abs{x})
  \end{equation*}
  and are thus radial.
  Consequently, the first eigenspace has dimension~$1$.

  For all $\kappa\in \bigcup_{n\geq 0}
  \intervaloo{j_{0,n+1}j_{0,n+2}, \,j_{1,n+1}j_{1,n+2}}$, the first
  eigenvalue is given by $\lambda_1(\kappa) = \alpha_{1,1}^2(\kappa) +
    \kappa^2 / \alpha_{1,1}^2(\kappa)$ and the
  eigenfunctions have the form 
	$$
	R_{1,1}(r) (c_1\cos\theta + c_2\sin\theta)
	$$
  for any $c_1$ and $c_2$.
  In this case, the first eigenspace has dimension~$2$.
%
\end{Thm}

\begin{proof}
  Note that, when $\alpha = \alpha_{k,\ell}$, the system \eqref{eq:R
    bd cond} is degenerate.  Moreover, $J_k(\alpha)$ and $J'_k(\alpha)$
  cannot vanish together.  Thus the dimention of the space of
  solutions to the system \eqref{eq:R bd cond} is exactly~$1$.

  By Proposition~\ref{mina01 ou a11},
  $\lambda_1 = \alpha^2+\kappa^2/\alpha^2$ with
  $\alpha = \min\{\alpha_{0,1}(\kappa), \alpha_{1,1}(\kappa)\}$.
  Thus Proposition~\ref{intersa01,a11} and
  Lemma~\ref{der inter} (with $k=0$) imply the claims about $\lambda_1$.
  Moreover by Proposition \ref{alpha k k+2}, we know that,
  for the values of $\kappa$ considered in the statement,
  $\alpha_{0,1}<\alpha_{2,1}$.
  To conclude the proof, it remains to establish that,
  in the second case, $\alpha_{1,1} < \alpha_{3,1}$.
  We will show that if $\alpha_{1,1}=\alpha_{3,1}$ then
  $\alpha_{0,1}<\alpha_{1,1}$.

  By Proposition \ref{alpha k k+2}, we know that
  $\alpha_{1,1}(\kappa)=\alpha_{3,1}(\kappa)$ if and only if
  $\kappa=j_{2,n}j_{2,n+1}$ for some $n\in \IN^*$, in which
  case
  \begin{equation}
\label{eq*}    
\alpha_{1,1}(j_{2,n}j_{2,n+1})
    = \alpha_{2,1}(j_{2,n}j_{2,n+1})
    = \alpha_{3,1}(j_{2,n}j_{2,n+1})
    = j_{2,n+1} .
  \end{equation}
  On the other hand, again by Proposition~\ref{alpha k k+2}, we have
  $\alpha_{0,1}(\kappa)\leq \alpha_{2,1}(\kappa)$ and
  $\alpha_{0,1}(\kappa)=\alpha_{2,1}(\kappa)$ if and only if
  $\kappa=j_{1,m}j_{1,m+1}$ for some $m\in \IN^*$.
  Thanks to \eqref{eq:ineq prod jkl} and \eqref{eq*},
  this implies the conclusion that 
  \begin{equation*}
    \alpha_{0,1}(j_{2,n}j_{2,n+1})
    < \alpha_{2,1}(j_{2,n}j_{2,n+1})
    = \alpha_{1,1}(j_{2,n}j_{2,n+1}).    \qedhere
  \end{equation*}
\end{proof}


\begin{Thm} 
  \label{1st eigen2}
  Denote $R_{k,\ell}$ the function defined by equation~\eqref{eq R} with
  $\alpha = \alpha_{k,\ell}$ given by Theorem~\ref{thm:alpha kl} and
  $(c,d)$ being a non-zero element of the
  one dimensional space of solutions to~\eqref{eq:R bd cond}.

  If $\kappa = j_{0,n}j_{0,n+1}$ for some $n \ge 1$, then
  $\alpha_{0,1} = \alpha_{1,1} < \alpha_{k,\ell}$ for all $(k,\ell)$
  different from $(0,1)$ and $(1,1)$.  The eigenfunctions have the
  form
  \begin{equation*}
    c_1 R_{0,1}(r) + R_{1,1}(r) (c_2\cos\theta + c_3\sin\theta),
    \qquad c_1, c_2, c_3 \in \IR.
  \end{equation*}

  If $\kappa = j_{1,n}j_{1,n+1}$ for some $n \ge 1$, then
  $\alpha_{0,1} = \alpha_{1,1} = \alpha_{2,1} < \alpha_{k,\ell}$
  for all $(k,\ell) \notin \{ (0,1),\ (1,1),\ (2,1) \}$.
  The eigenfunctions have the form
  \begin{equation*}
    c_1 R_{0,1}(r) + R_{1,1}(r) (c_2\cos\theta + c_3\sin\theta)
    + R_{2,1}(r) \bigl(c_4 \cos(2\theta) + c_5 \sin(2\theta) \bigr).
  \end{equation*}
  where $c_1,\dotsc,c_5$ vary in $\IR$.
\end{Thm}

\begin{proof}
  First consider $\kappa = j_{0,n}j_{0,n+1}$ for some $n \ge 1$.
  Using Lemma~\ref{valalpha}, one has $\alpha_{0,1}(j_{0,n}j_{0,n+1}) = \alpha_{1,1} (j_{0,n}j_{0,n+1})=
  j_{0,n+1}$.  Proposition~\ref{alpha k k+2} implies that
  $\alpha_{0,1}(j_{0,n}j_{0,n+1}) < \alpha_{2,1}(j_{0,n}j_{0,n+1})$ as, if they were equal, then
  $\alpha_{0,1}(j_{0,n}j_{0,n+1}) = \alpha_{2,1}(j_{0,n}j_{0,n+1}) = j_{1,m}$ which is impossible because
  the positive roots of $J_0$ and $J_1$ interlace.

  A similar argument shows
  $\alpha_{1,1} (j_{0,n}j_{0,n+1})< \alpha_{3,1}(j_{0,n}j_{0,n+1})$
  because the roots of $J_0$ and $J_2$ interlace
  (see Remark~\ref{interlace k k+2}).
  Using again
  Proposition~\ref{alpha k k+2}, it is then easy to conclude that no other
  $\alpha_{k,\ell}$ is equal to $\alpha_{0,1} = \alpha_{1,1}$.
  The form of the eigenfunctions readily follows from
  Proposition~\ref{eigenfunctions}.

  Now, let $\kappa = j_{1,n}j_{1,n+1}$ for some $n \ge 1$.
  Proposition~\ref{alpha k k+2} says that
  $\alpha_{0,1}(j_{1,n}j_{1,n+1}) = \alpha_{1,1} (j_{1,n}j_{1,n+1})
  = \alpha_{2,1}(j_{1,n}j_{1,n+1}) = j_{1,n+1}$.
  Moreover, by the same arguments as above, one gets
  $\alpha_{2,1}(j_{1,n}j_{1,n+1})
  < \alpha_{4,1}(j_{1,n}j_{1,n+1})$ as well as $\alpha_{1,1} (j_{1,n}j_{1,n+1})< \alpha_{3,1}(j_{1,n}j_{1,n+1})$ and then
  conclude that no other $\alpha_{k,\ell}$ is equal to $\alpha_{0,1} =
  \alpha_{1,1} = \alpha_{2,1}$.  Again, the form of the eigenfunctions
  follows easily.
\end{proof}

\section{Nodal properties of the first eigenfunction}
\label{sec:nodal prop}

In this section, we will give further nodal properties of the eigenfunctions  with respect to the $\kappa$-intervals.

\begin{Lem}
  \label{R01}
  Let $0 < \kappa \le j_{0,1} j_{0,2}$ and $R_{0,1}$ be defined as in
  Theorem~\ref{1st eigen} (or~\ref{1st eigen2}).  Then $r \mapsto
  \abs{R_{0,1}(r)}$ is positive in $\intervalco{0,1}$ and
  decreasing.
\end{Lem}

\begin{proof}
  Theorem~\ref{1st eigen} says that
  \begin{equation*}
    R_{0,1}(r)
    = c\,  J_0(\alpha_{0,1} r)
    + d \, J_0\Bigl(\frac{\kappa}{\alpha_{0,1}} r\Bigr)
  \end{equation*}
  where  $(c,d)$ is a nontrivial solution to~\eqref{eq:R bd cond}
  with $\alpha = \alpha_{0,1}$.
  Remark~\ref{Range of alpha_k1} imply
  that, for $0 < \kappa \le j_{0,1}j_{0,2}$, we have
  $0< \alpha_{0,-1}(\kappa) \le j_{0,1} < j_{1,1} < \alpha_{0,1}(\kappa)
  \le j_{0,2}$.
  Thus $J_0'(\alpha_{0,1}) = -J_1(\alpha_{0,1}) > 0$ and
  $J_0'(\kappa/\alpha_{0,1}) = -J_1(\kappa/\alpha_{0,1}) < 0$ and
  hence the second equation of~\eqref{eq:R bd cond} is non-degenerate
  and a possibility is to choose w.l.o.g.
  \begin{equation*}
    c := -\frac{\kappa}{\alpha_{0,1}}
    J_0'\Bigl(\frac{\kappa}{\alpha_{0,1}}\Bigr) > 0
    \qquad\text{and}\qquad
    d := \alpha_{0,1} J'_0(\alpha_{0,1}) > 0.
  \end{equation*}
  We want to show that $v(r) := \partial_rR_{0,1}(r) <0$ for all $r \in
  \intervaloo{0,1}$. As $R_{0,1}(1)=0$ we then obtain also $R_{0,1}>0$
  on~$\intervalco{0,1}$.
	
  Observe that $v$ is given by
  \begin{equation}
    \label{eq:d R01}
    v(r)
    = -\Bigl[c\, \alpha_{0,1} J_1(\alpha_{0,1} r)
    + d \, \frac{\kappa}{\alpha_{0,1}}
    J_1\Bigl(\frac{\kappa}{\alpha_{0,1}} r\Bigr)\Bigr].
  \end{equation}
  Since $\frac{\kappa}{\alpha_{0,1}} r \in \intervalco{0, j_{1,1}}$,
  we have $J_1\bigl(\frac{\kappa}{\alpha_{0,1}} r\bigr) > 0$ for all $r \in
  \intervalcc{0,1}$.  If $J_1(\alpha_{0,1} r) \ge 0$,
  which is the case when $r \in
  \intervalcc{0, j_{1,1}/\alpha_{0,1}}$,
  then clearly $v(r) <0$.

  For $r \in \intervaloo{j_{1,1}/\alpha_{0,1}, 1}$,
  $\alpha_{0,1} r \in \intervaloo{j_{1,1}, \alpha_{0,1}} \subseteq
  \intervaloo{j_{1,1}, j_{1,2}}$.  Thus
  $J_1(\alpha_{0,1} r) < 0$ and the negativity of $v$ is not
  straightforward.
  Suppose on the contrary there exists a $r^* \in
  \intervaloo{j_{1,1}/\alpha_{0,1}, 1}$ such that $v(r^*) = 0$.
  A simple computation using \eqref{A:6}  shows that $v$ solves
  \begin{equation}
    \label{partial R}
    \begin{split}
      -\partial_r^2 v- \frac{1}{r} \partial_r v
      + \Bigl(\frac{1}{r^2}- \frac{\kappa^2}{\alpha_{0,1}^2}\Bigr) v
      &= -c \alpha_{0,1} \Bigl( \alpha_{0,1}^2
      - \frac{\kappa^2}{\alpha_{0,1}^2} \Bigr)
      J_1(\alpha_{0,1} r),
      \\
      v(r^*)=0, &\quad v(1)=0.
    \end{split}
  \end{equation}
The right hand side is positive on $\intervaloc{r^*, 1}$.  Moreover,
  the problem can be rewritten under the form
  \begin{equation}
    \label{Laplacian partial R}
    \begin{array}{c}
      \displaystyle
			-\Delta v
      + \Bigl(\frac{1}{r^2}- \frac{\kappa^2}{\alpha_{0,1}^2}\Bigr) v
      = -c \alpha_{0,1} \Bigl( \alpha_{0,1}^2
      - \frac{\kappa^2}{\alpha_{0,1}^2} \Bigr)
      J_1(\alpha_{0,1} \abs{x}),
      \\
      v=0 \quad \text{ on } \partial A^*,
    \end{array}
  \end{equation}
  where $A^*:= \{x\in \IR^2 \mid r^*< \abs{x} < 1\}$.
  Let us prove that $v\ge 0$ on $A^*$.  
  
  Recall that ${\kappa^2}/{\alpha_{0,1}^2} = \alpha_{0,-1}^2 < j_{1,1}^2$ where
  $j_{1,1}^2$ is the first eigenvalue of $-\Delta +\frac{1}{r^2}$ on
  the unit ball with zero Dirichlet boundary conditions (with
  eigenfunction $J_{1}(j_{1,1}r)$).  As the first eigenvalue of $-\Delta
  +\frac{1}{r^2}$ on the unit ball is less than the first eigenvalue of $-\Delta
  +\frac{1}{r^2}$ on the annulus $A^*$, 
	we deduce, by the maximum principle, that $v\geq0$ on $A^*$
	(see \cite{Walter} or \cite[Theorem 2.8]{Du}).
	
  This implies that $\partial_r v(1)\leq 0$, i.e., $\partial^2_r
  R_{0,1}(1)\le 0$.  Moreover, $\partial_r v(1) \ne 0$ because,
  otherwise, \eqref{partial R} evaluated at $r=1$ would give
  $-\partial_r^2v(1) > 0$ which would imply that $v(r) < 0$
  for $r$ close to~$1$.  Thus $v'(1)=\partial_r^2 R_{0,1}(1) < 0$.

  Since $R_{0,1}$ satisfies
  \begin{equation*}
    -\partial_r^2 R_{0,1} - \frac{1}{r} \partial_r R_{0,1}
    - \Bigl(\frac{\kappa}{\alpha_{0,1}}\Bigr)^2 R_{0,1}
    = c \Bigl( \alpha_{0,1}^2 - \frac{\kappa^2}{\alpha_{0,1}^2} \Bigr)
    J_0(\alpha_{0,1} r),    
  \end{equation*}
  the evaluation in $r=1$ gives a contradiction, as $R_{0,1}(1)=0$,
  $\partial_r R_{0,1}(1)=0$ and $J_0(\alpha_{0,1}) \le 0$ (recall that
  $\alpha_{0,1}\in \intervaloc{j_{1,1}, j_{0,2}}$).
	
  In conclusion $v=\partial_rR_{0,1}<0$ on $\intervaloo{0,1}$ and
  hence $R_{0,1}>0$ on $\intervalco{0,1}$.
\end{proof}

\begin{Rem}
  For $\kappa > j_{0,1} j_{0,2}$, the function $R_{0,1}$ changes sign
  as illustrated by Figure~\ref{fig:R01}.
\end{Rem}

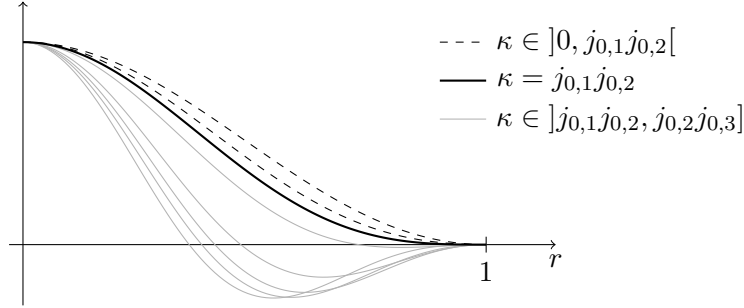
\begin{figure}[ht]
  \centering
  \begin{tikzpicture}[x=16em, y=7em]
    \draw[->] (-0.03, 0) -- (1.15, 0) node[below]{$r$};
    \draw[->] (0, -0.3) -- (0, 1.2);
    \begin{scope}[dashed]
      \input{data/R01-1.tex}
      \draw (0.9, 1) -- +(0.1, 0) node[right]{%
        $\kappa \in \intervaloo{0, j_{0,1} j_{0,2}}$};
    \end{scope}
    \begin{scope}[color=black!30]
      \input{data/R01-3.tex}
      \draw (0.9, 0.62) -- +(0.1, 0) node[right, color=black]{%
        $\kappa \in \intervaloc{j_{0,1} j_{0,2}, j_{0,2} j_{0,3}}$};
    \end{scope}
    \draw[thick] plot file{data/R01-2.dat};
    \draw[thick] (0.9, 0.81) -- +(0.1, 0) node[right]{%
      $\kappa = j_{0,1} j_{0,2}$};

    \draw (1, 3pt) -- (1, -3pt) node[below]{$1$};
  \end{tikzpicture}

  \caption{Graph of $R_{0,1}$ for various values of~$\kappa$.}
  \label{fig:R01}
\end{figure}

\begin{Lem}
  \label{Rk1}
  Let $0< \kappa \le j_{k,1} j_{k,2}$ and $R_{k,1}$ be defined as in
  Theorem~\ref{1st eigen} (or~\ref{1st eigen2}).  Then
  $\abs{R_{k,1}(r)}$ is positive for~$r \in \intervaloo{0,1}$.
\end{Lem}

\begin{proof}
  Recall that  by equation~\eqref{eq R}, we know that
  \begin{equation*}
    R_{k,1}(r)=cJ_k(\alpha_{k,1}r)+d J_k(\alpha_{k,-1} r)
  \end{equation*}
  where the real numbers $c$ and $d$ solve the linear degenerate
  system~\eqref{eq:R bd cond} with  $\alpha = \alpha_{k,1}$. 
  Observe that, by Remark \ref{Range of alpha_k1}, we have
  $0 < \alpha_{k,-1} \le j_{k,1} < j_{k+1,1} < \alpha_{k,1} \le
  j_{k,2} < j_{k+1,2}$,
  and hence $J_{k+1}(\alpha_{k,1}) < 0$
  and $J_{k+1}(\alpha_{k,-1}) > 0$.  Using \eqref{A:4}, one deduces
  that any solution $(c,d)$ to~\eqref{eq:R bd cond} must also satisfy
  \begin{equation*}
    c \, \alpha_{k,1} J_{k+1}(\alpha_{k,1})
    + d \frac{\kappa}{\alpha_{k,1}} J_{k+1}\Bigl(
    \frac{\kappa}{\alpha_{k,1}} \Bigr) = 0.
  \end{equation*}
  Thus one can take for example for $c$ and $d$:
  \begin{equation*}
    c := \frac{\kappa}{\alpha_{k,1}} J_{k+1}\Bigl(
    \frac{\kappa}{\alpha_{k,1}} \Bigr) > 0
    \qquad\text{and}\qquad
    d := -\alpha_{k,1} J_{k+1}(\alpha_{k,1}) > 0.
  \end{equation*}
  We want to show that $R_{k,1} >0$ on $\intervaloo{0,1}$.
  Since, for all $r \in \intervaloo{0,1}$,
  $\alpha_{k,-1} r \in \intervaloo{0, j_{k,1}}$, we have
  $J_k\bigl(\alpha_{k,-1} r\bigr) > 0$.
  If $r$ is such that $J_k(\alpha_{k,1} r) \ge 0$, i.e., if
  $r \in \intervaloc{0, j_{k,1}/\alpha_{k,1}}$,
  then clearly $R_{k,1}(r) > 0$.

  For $r \in \intervaloo{j_{k,1}/\alpha_{k,1}, 1}$, we have
  $\alpha_{k,1} r \in \intervaloo{j_{k,1}, \alpha_{k,1}}
  \subseteq \intervaloo{j_{k,1}, j_{k,2}}$.  Thus
  $J_k(\alpha_{k,1} r) < 0$ and the positivity of $R_{k,1}$ is not
  straightforward.
  Suppose on the contrary there exists a 
  $r^* \in \intervaloo{j_{k,1}/\alpha_{k,1}, 1}$ such that
  $R_{k,1}(r^*) = 0$.
  A simple computation using \eqref{A:6}  shows:
  \begin{equation}
    \label{Laplacian R alpha11}
    \begin{cases}
      -\Delta \bigl(R_{k,1}\sin(k\theta) \bigr)
      - \alpha_{k,-1}^2 R_{k,1}\sin(k\theta) \\[1\jot]
      \hspace{6em}
      = c \bigl( \alpha_{k,1}^2 - \alpha_{k,-1}^2 \bigr)
      J_k(\alpha_{k,1} r) \sin(k\theta),
      \quad \text{ in } A^+_k,
      \\[2mm]
      R_{k,1}(r) \sin(k\theta)=0,
      \quad\text{ on }\partial A^+_k,
    \end{cases}
  \end{equation}
  where $A^+_k := \bigl\{(r\cos(\theta), r\sin(\theta))\bigm|
  r^*<r<1,\ \theta\in \intervaloo{0, \pi/k} \bigr\}$.

  The right hand side is negative for $r\in \intervaloo{r^*, 1}$
  and $\theta\in \intervaloo{0, \pi/k}$.
  Since $\alpha_{k,-1}^2 \le j_{k,1}^2$,
  where $j_{k,1}^2$ is the first eigenvalue of $-\Delta$ on
  \begin{equation*}
    D^+
    := \bigl\{ (r\cos(k\theta), r\sin(k\theta)) \bigm|
    0<r<1,\
    \theta\in \intervaloo{0, \pi/k} \bigr\},
  \end{equation*}
  with zero Dirichlet boundary conditions 
  (with positive first eigenfunction $J_k(j_{k,1}r)\sin(k \theta)$),
  which is less than the first eigenvalue of $-\Delta$ on 
  $A^+_k \subsetneq D^+$, the maximum principle applies
  (see~\cite{Walter} or~\cite[Theorem 2.8]{Du})
  and we conclude that $R_{k,1}(r) \sin(k\theta) < 0$ on
  $A^+_k$.
  Evaluating~\eqref{Laplacian R alpha11} for $r = 1$ and taking into account
  the clamped boundary conditions $R_{k,1}(1) = 0 = \partial_r R_{k,1}(1)$, one
  deduces $\partial_r^2 R_{k,1}(1)
  = -c ( \alpha_{k,1}^2 - \alpha_{k,-1}^2 )
  J_k(\alpha_{k,1}) \ge 0$.
  If $\partial_r^2 R_{k,1}(1) > 0$, this contradicts $R_{k,1}<0$.
  If $\partial_r^2 R_{k,1}(1) = 0$, i.e., $\alpha_{k,1} = j_{k,2}$,
  differentiating \eqref{Laplacian R alpha11} w.r.t.\ $r$ and
  evaluating at $r=1$ yields
  \begin{equation*}
    \partial_r^3 R_{k,1}(1)
    = -c ( \alpha_{k,1}^2 - \alpha_{k,-1}^2 )\alpha_{k,1} J_k'(j_{k,2})
    < 0
  \end{equation*}
  which again contradicts $R_{k,1}(r)<0$.
  This proves that $R_{k,1}>0$ on $\intervaloo{0,1}$.
\end{proof}

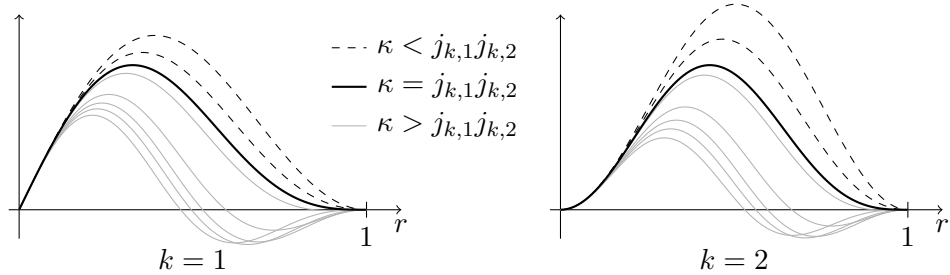
\begin{figure}[hbt]
  \centering
  \begin{tikzpicture}[x=12em, y=5em]
    \draw[->] (-0.03, 0) -- (1.1, 0) node[below]{$r$};
    \draw[->] (0, -0.2) -- (0, 1.35);
    \draw (1, 3pt) -- (1, -3pt) node[below]{$1$};
    \node[below] at (0.5, -0.2) {$k=1$};
    \begin{scope}[dashed]
      \input{data/R11-1.tex}
      \draw (0.9, 1.12) -- +(0.1, 0) node[right]{%
        $\kappa < j_{k,1} j_{k,2}$};
    \end{scope}
    \begin{scope}[color=black!30]
      \input{data/R11-3.tex}
      \draw (0.9, 0.58) -- +(0.1, 0) node[right, color=black]{%
        $\kappa > j_{k,1} j_{k,2}$};
    \end{scope}
    \draw[thick] plot file{data/R11-2.dat};
    \draw[thick] (0.9, 0.85) -- +(0.1, 0) node[right]{%
      $\kappa = j_{k,1} j_{k,2}$};

    \begin{scope}[xshift=18.7em]
      \draw[->] (-0.03, 0) -- (1.1, 0) node[below]{$r$};
      \draw[->] (0, -0.2) -- (0, 1.35);
      \draw (1, 3pt) -- (1, -3pt) node[below]{$1$};
      \node[below] at (0.5, -0.2) {$k=2$};
      \begin{scope}[dashed]
        \input{data/R21-1.tex}
      \end{scope}
      \begin{scope}[color=black!30]
        \input{data/R21-3.tex}
      \end{scope}
      \draw[thick] plot file{data/R21-2.dat};
    \end{scope}
  \end{tikzpicture}

  \caption{Graph of $R_{k,1}$ for various values of~$\kappa$.}
  \label{fig:Rk1}
\end{figure}

\begin{Rem}
  When $k > 0$, $R_{k,1}$ is no longer decreasing
  (see Figure~\ref{fig:Rk1}) because
  $R_{k,1}(0) = 0$ and $R_{k,1}(1) = 0$.
\end{Rem}

Hence we have proved so far the following result.

\begin{Thm}
  \label{Thmprincipal2}
  If $0\leq \kappa < j_{0,1} j_{0,2}$, the first eigenspace is of
  dimension~$1$, any first eigenfunction $\varphi_1$
  is radial
  and $|\varphi_1|$ is positive in $\Omega$ and decreasing w.r.t.\
  $r = \abs{x}$.

  If $j_{0,1} j_{0,2}< \kappa < j_{1,1} j_{1,2}$, the first
  eigenfunctions have the form $R_{1,1}(r) \cdot (c_1 \cos\theta + c_2
  \sin\theta)$ with $R_{1,1}(r)>0$ for $r \in \intervaloo{0,1}$
  and hence have two nodal domains that are half balls.
\end{Thm}

\begin{proof}
  The case $\kappa=0$ can be deduced from Theorem
  \ref{tserge2}. Consider then the case $\kappa>0$.
  When $\kappa \in \intervaloo{0, j_{0,1} j_{0,2}}
  \cup \intervaloo{j_{0,1} j_{0,2}, j_{1,1} j_{1,2}}$, Theorem~\ref{1st
    eigen} says that the eigenfunctions are the desired form;
  Lemmas~\ref{R01} and~\ref{Rk1} complete the proof.
\end{proof}


In order to study the evolution of $R_{k,1}$ in the next intervals, we
first prove that $R_{k,1}$ changes sign in every interval of the form
$[\frac{j_{k,i}}{\alpha_{k,1}},\frac{j_{k,i+1}}{\alpha_{k,1}}]$. In a
second step, we will prove that $R_{k,1}$ will change sign only once
on this interval. This will allow us to deduce on the exact number
of root of $\varphi_1$ according to the value of~$\kappa$.

\begin{Lem}
  \label{cserge25:02}%
  Let $k\in \IN$ and $n \in \IN\setminus\{0,1\}$ be fixed.
  Then for all $i\in \{1, \dotsc, n-1\}$ and all
  $\kappa\in \intervaloo{j_{k,n-1} j_{k,n},\, j_{k,n} j_{k,n+1}}$,
  \begin{equation*}
    R_{k,1} \Bigl(\frac{j_{k,i}}{\alpha_{k,1}}\Bigr)\,
    R_{k,1} \Bigl(\frac{j_{k,i+1}}{\alpha_{k,1}} \Bigr)
    <0.
  \end{equation*}
  where $R_{k,1}$ is defined as in Theorem~\ref{1st eigen}.
\end{Lem}

\begin{proof}
First observe that
  \begin{equation*}
    R_{k,1}\Bigl(\frac{j_{k,i}}{\alpha_{k,1}}\Bigr) \,
    R_{k,1}\Bigl(\frac{j_{k,i+1}}{\alpha_{k,1}}\Bigr)
    =
    d^2 J_k\Bigl(j_{k,i}\frac{\kappa}{\alpha_{k,1}^2(\kappa)}\Bigr) \,
    J_k\Bigl(j_{k,i+1}\frac{\kappa}{\alpha_{k,1}^2(\kappa)} \Bigr).
  \end{equation*}
  Note that $d \ne 0$ because otherwise \eqref{eq:R bd cond} would
  boil down to $c J_k(\alpha) = 0 = c J'_k(\alpha)$ and so $c = 0$,
  contradicting the fact that $(c,d)$ must be non-trivial.
  We will prove that
  \begin{equation*}
    J_k\Bigl(j_{k,i}\frac{\kappa}{\alpha_{k,1}^2(\kappa)} \Bigr) \,
    J_k\Bigl(j_{k,i+1}\frac{\kappa}{\alpha_{k,1}^2(\kappa)} \Bigr)
    <0.
  \end{equation*}
  Set $h(\kappa) := {\kappa}/{\alpha_{k,1}^2(\kappa)}$.
As $i\leq n-1$, we clearly have
  \begin{equation*}
   j_{k,i-1} \, j_{k,i}  <  j_{k,i} \, j_{k,i+1} < \kappa,
 \end{equation*}
 and, because $h$ is increasing thanks to Lemma~\ref{lserge25:02},
$$
h(j_{k,i-1} j_{k,i})<h(j_{k,i} j_{k,i+1})<h(\kappa).
$$
As, by Lemma \ref{valalpha},  $h(j_{k,i-1} j_{k,i})=\frac{j_{k,i-1}}{j_{k,i}}$ and $h(j_{k,i} j_{k,i+1})=\frac{j_{k,i}}{j_{k,i+1}}$, we deduce that
  \begin{equation*}
    h(\kappa)j_{k,i} > j_{k,i-1}
    \quad\text{and}\quad
    h(\kappa)j_{k,i+1} > j_{k,i}.
  \end{equation*}
  Moreover, $h(\kappa)<1$ because
  $\alpha_{k,1}(\kappa)>\sqrt{\kappa}$.
  We conclude that
  \begin{equation*}
    j_{k,i-1} < h(\kappa)j_{k,i} < 
    j_{k,i} < h(\kappa)j_{k,i+1} < j_{k,i+1}.
  \end{equation*}
This means that $h(\kappa)j_{k,i}$ and $h(\kappa)j_{k,i+1}$ are in two 
consecutive intervals of zeros of $J_k$ and the conclusion follows.
\end{proof}

\begin{Rem}
  The previous Lemma does not readily extend to $\alpha_{k,\ell}$ with
  $\ell > 1$.  The left graph on Figure~\ref{fig:R0l kappa large}
  illustrates the result while the right one shows that even the
  number of sign changes of $R_{k,\ell}$, for $\ell > 1$,
  does not correspond to the
  number of points $( {j_{k,i}}/{\alpha_{k,\ell}} )_{i=1}^n$.
\end{Rem}

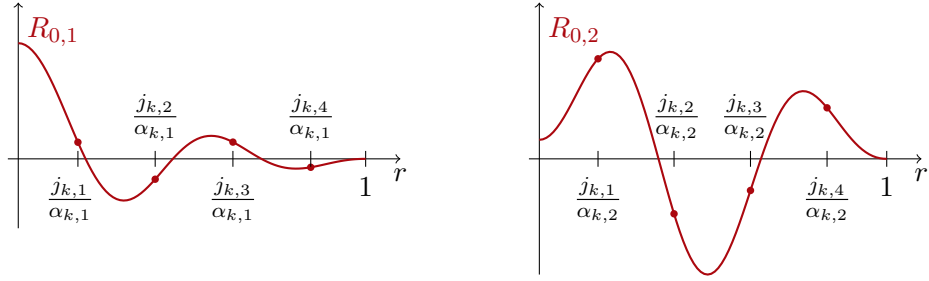
\begin{figure}[hbt]
  \centering
  \begin{tikzpicture}[x=12em, y=4em]
    \draw[->] (-0.03, 0) -- (1.1, 0) node[below]{$r$};
    \draw[->] (0, -0.6) -- (0, 1.35);
    \draw (1, 3pt) -- (1, -3pt) node[below]{$1$};
    \begin{scope}[color=darkred, thick]
      \input{data/R01-sign4.tex}
      \node at (0.1, 1.1) {$R_{0,1}$};
    \end{scope}
    \draw (0.171709, 3pt) -- +(0, -6pt) node[below, xshift=-3pt]{%
      $\frac{j_{k,1}}{\alpha_{k,1}}$};
    \draw (0.394144, -3pt) -- +(0, 6pt) node[above, xshift=0pt]{%
      $\frac{j_{k,2}}{\alpha_{k,1}}$};
    \draw (0.617893, 3pt) -- +(0, -6pt) node[below, xshift=0pt]{%
      $\frac{j_{k,3}}{\alpha_{k,1}}$};
    \draw (0.841939, -3pt) -- +(0, 6pt) node[above, xshift=0pt]{%
      $\frac{j_{k,4}}{\alpha_{k,1}}$};

    \begin{scope}[xshift=18em]
      \draw[->] (-0.03, 0) -- (1.1, 0) node[below]{$r$};
      \draw[->] (0, -1.) -- (0, 1.35);
      \draw (1, 3pt) -- (1, -3pt) node[below]{$1$};
      \begin{scope}[color=darkred, thick]
        \input{data/R02-sign4.tex}
        \node at (0.1, 1.1) {$R_{0,2}$};
      \end{scope}
      \draw (0.168909, 3pt) -- +(0, -6pt) node[below]{%
        $\frac{j_{k,1}}{\alpha_{k,2}}$};
      \draw (0.387717, -3pt) -- +(0, 6pt) node[above, xshift=1pt]{%
        $\frac{j_{k,2}}{\alpha_{k,2}}$};
      \draw (0.607818, -3pt) -- +(0, 6pt) node[above, xshift=-2pt]{%
        $\frac{j_{k,3}}{\alpha_{k,2}}$};
      \draw (0.828210, 3pt) -- +(0, -6pt) node[below]{%
        $\frac{j_{k,4}}{\alpha_{k,2}}$};
    \end{scope}
  \end{tikzpicture}

  \caption{Graph of $R_{k,\ell}$ for a $\kappa \in
    \intervaloo{j_{k,3} j_{k,4}, j_{k,4} j_{k,5}}$ and $k=0$.}
  \label{fig:R0l kappa large}
\end{figure}

\begin{Lem}
  \label{Lem Ev1}
  Let $k\in \IN$, $n \in \IN^{*}$, $\kappa\in\intervaloo{j_{k,n-1}
    j_{k,n},\, j_{k,n} j_{k,n+1}}$, and $R_{k,1}$ be as in
  Theorem~\ref{1st eigen} with $c > 0$ and $d > 0$.
  Let $i\in \{0,\dotsc, n\}$ and
  $\frac{j_{k,i}}{\alpha_{k,1}} \le r_1 < r_2 \le
  \frac{j_{k,i+1}}{\alpha_{k,1}}$
  (with the convention that $j_{k,0} := 0$).
  \begin{itemize}
  \item If $i$ is odd and
    $R_{k,1}(r_1)\le 0$ and $R_{k,1}(r_2)\le 0$, then $R_{k,1}(r) <
    0$ on $\intervaloo{r_1,r_2}$.  Moreover, if $R_{k,1}(r_1) = 0$
    then $R_{k,1}'(r_1)<0$ and if $R_{k,1}(r_2) = 0$ then
    $R_{k,1}'(r_2)>0$.
  \item If $i$ is even and $R_{k,1}(r_1)\ge 0$ and $R_{k,1}(r_2) \ge 0$, then
    $R_{k,1}(r) > 0$ on $\intervaloo{r_1,r_2}$.  Moreover, if
    $R_{k,1}(r_1) = 0$ then $R_{k,1}'(r_1)>0$ and if
    $R_{k,1}(r_2) = 0$ then $R_{k,1}'(r_2)<0$.
  \end{itemize}
\end{Lem}

\begin{Rem}
  Given the interval where $\kappa$ lies, Remark~\ref{Range of
    alpha_k1} asserts that $j_{k,n-1} < \alpha_{k,-1} < j_{k,n}
  < \alpha_{k,1} < j_{k,n+1}$.  Therefore $J_k(\alpha_{k,1}) \ne 0$
  and one can use the first equation of the degenerate
  system~\eqref{eq:R bd cond} to find a nontrivial solution $(c,d)$.
  Moreover, since $J_k(\alpha_{k,1})$ and $J_k(\alpha_{k,-1})$ have
  opposite signs, one can always choose $c > 0$ and $d > 0$.
\end{Rem}

\begin{proof}
  Observe that $J_k(\alpha_{k,1}r)\sin (k\theta)$ is a
  solution to
  \begin{equation*}
    \begin{cases}
      -\Delta \varphi = \alpha_{k,1}^2\,\varphi, &\text{in } A_i^+,
      \\[1\jot]
      \varphi=0,&\text{on }\partial A_i^+.
    \end{cases}
  \end{equation*}
  where $A_i^+ := \bigl\{ (r\cos(\theta), r\sin(\theta)) \in \IR^2 \bigm|
  \frac{j_{k,i}}{\alpha_{k,1}} < r < \frac{j_{k,i+1}}{\alpha_{k,1}}$
  and $0 < \theta < \frac{\pi}{k} \bigr\}$.
  Moreover, as $J_k(\alpha_{k,1}r)\sin(k\theta)$ does not change sign
  in $A_i^+$,
  it is the first eigenfunction of~$-\Delta$ in $A_i^+$.
  As $\alpha_{k,-1}^2 < \alpha_{k,1}^2$, the maximum principle is
  valid for \eqref{Laplacian R alpha11} with $A_k^+$ replaced by
  $A_i^+$ (see \cite{Walter} or \cite[Theorem~2.8]{Du}).
  The result then follows from the fact that, when $i$ is odd
  (resp.\ even), the right hand side of~\eqref{Laplacian R alpha11}
  is negative (resp.\ positive) on $\intervaloo{r_1,r_2}$.
\end{proof}

It remains to study the sign of $R_{k,1}$ for $r$ close to~$1$.

\begin{Lem}
  \label{Lem Ev2}
  Let $k$, $n$, $\kappa$, and $R_{k,1}$ be as in Lemma~\ref{Lem Ev1}.
  When $n$ is odd (resp.\ even) then, for all $r\in \bigintervaloo{
    \frac{j_{k,n}}{\alpha_{k,1}}, 1}$, $R_{k,1}(r) > 0$ (resp.\
  $R_{k,1}(r) < 0$).
\end{Lem}

\begin{proof}
Observe that $R_{k,\ell}(1)=0$, $\partial_r R_{k,\ell}(1)=0$ and
using the easily checked identity
  \begin{equation*}
    \partial_r^2 R_{k,1}(r)
    = -\frac{1}{r} \partial_r R_{k,1}(r)
    + \Bigl(\frac{k^2}{r^2} - \alpha_{k,1}^2 \Bigr) R_{k,1}(r)
    + d \bigl(\alpha_{k,1}^2 - \alpha_{k,-1}^2 \bigr)
    J_k\bigl(\alpha_{k,-1} r\bigr),
  \end{equation*}
we get
  \begin{equation*}
    \partial_r^2 R_{k,1}(1)
    = d (\alpha_{k,1}^2 - \alpha_{k,-1}^2) \, J_k(\alpha_{k,-1}).
  \end{equation*}
By our choice of $d>0$, we see that
  $\partial_r^2 R_{k,1}(1)$ has the same sign as $J_k(\alpha_{k,-1})$.
  Since $\alpha_{k,-1}$ belongs to
  $\intervaloo{j_{k,n-1}, j_{k,n}}$, we deduce that, for $n-1$ even
  (resp.\ $n-1$ odd), $\partial_r^2 R_{k,1}(1)>0$
  (resp.\ $\partial_r^2 R_{k,1}(1) < 0$).  This implies the
  existence of $\epsilon > 0$ such that $R_{k,1} > 0$ (resp.\
  $R_{k,1} < 0$) on $\intervalco{1-\epsilon, 1}$.

  If $R_{k,1}$ has a root
  $r_1 \in \intervaloo{\frac{j_{k,n}}{\alpha_{k,1}}, 1}$, we have a
  contradiction with Lemma~\ref{Lem Ev1} applied with $i=n$ and
  $r_2=1$.
\end{proof}

\begin{Prop}
  \label{Prop Ev5}
  Let $k\in \IN$, $n \in \IN^{*}$ fixed and
  $\kappa\in\intervaloo{j_{k,n-1} j_{k,n},\, j_{k,n} j_{k,n+1}}$
  (i.e., $\alpha_{k,1}\in \intervaloo{j_{k,n}, j_{k,n+1}}$) with $n\in
  \IN$, $n\ge 2$, then $R_{k,1}$ has exactly $n-1$ simple zeros
  in~$\intervaloo{0,1}$.
\end{Prop}
\begin{proof}
Recall that $R_{k,1}(r)=cJ_k(\alpha_{k,1} r)+ d J_k(\alpha_{k,-1} r)$,   with $c>0$ and $d>0$
  and hence $R_{k,1}>0$ on $\intervaloo{0,
    \frac{j_{k,1}}{\alpha_{k,1}}}$.  On the
  other hand by Lemma~\ref{Lem Ev2}, we know that $R_{k,1}(r)$ has no
  root on $\intervalco{\frac{j_{k,n}}{\alpha_{k,1}}, 1}$.

Moreover, 
by Lemma  \ref{cserge25:02}, we know that, for all $i\in \{1, \ldots, n-1\}$,
  \begin{equation*}
    R\Bigl(\frac{j_{k,i}}{\alpha_{k,1}}\Bigr)
    R\Bigl(\frac{j_{k,i+1}}{\alpha_{k,1}} \Bigr) < 0.
  \end{equation*}
Hence by
  Lemmas  \ref{Lem Ev1} and \ref{Lem Ev2} we deduce that $R_{k,1}$
  has exactly one root on $\intervaloo{\frac{j_{k,i}}{\alpha_{k,1}},
    \frac{j_{k,i+1}}{\alpha_{k,1}}}$
for all $i\in \{1, \ldots, n-1\}$. This proves the result.
\end{proof}

From Theorem~\ref{1st eigen} and Proposition~\ref{Prop Ev5},
it is easy to derive Theorem~\ref{1st eigen nodal}.

\section{Extension to any dimension}
\label{sec:any dim}

The buckling problem \eqref{eq:pbm} in the unit ball of $\IR^N$, with
$N\geq 3$, can be treated as before.  Indeed using spherical
coordinates, we first look for a solution $u$ to
$$
(\Delta+\alpha^2)u=0,
$$
with $\alpha\geq 0$, in the form
$$
u=R(r) S(\theta),
$$
where $S$ is a spherical harmonic function, that is the
restriction to the unit sphere of an harmonic homogeneous
polynomial of degree $k \in \IN$.  Expressing $\Delta$ is spherical
coordinates (see \cite[p.~38]{Muller:66} or \cite{Grumiau-Troestler:10}),
one finds that $R$ must satisfy
\begin{equation*}
  \partial_r^2 R + \frac{N-1}{r} \partial_r R
  + \Bigl(\alpha^2 - \frac{k(k+N-2)}{r^2}\Bigr) R=0.
\end{equation*}
Performing the change of unknown
$$
R(r):=r^{-\frac{N-2}{2}} B(r),
$$
one sees that $B$ satisfies the Bessel-like equation
\begin{equation*}
  \partial_r^2 B + \frac{1}{r} \partial_r B
  + \Bigl(\alpha^2 -\frac{\nu_k^2}{r^2}\Bigr) B = 0,
\end{equation*}
where $\nu_k^2 = k(k+N-2) + \bigl(\frac{N-2}{2}\bigr)^2 = \bigl(k +
\frac{N-2}{2}\bigr)^2$.
Hence, if $\alpha > 0$, $B$ is a linear combination of
$r \mapsto J_{\nu_k}(\alpha r)$ and $r \mapsto Y_{\nu_k}(\alpha r)$,
while if $\alpha=0$, $B$ is a linear combination of $r^{\nu_k}$ and
$r^{-\nu_k}$.  So the results of Lemma~\ref{lserge1} remain valid in
dimension~$N$ if one replaces $r \mapsto J_k(\alpha r)$ by
\begin{equation*}
  r \mapsto r^{-\frac{N-2}{2}} J_{\nu_k}(\alpha r),
  \qquad\text{with }
  \nu_k=k+\frac{N-2}{2},
\end{equation*}
$r J'_k(\alpha r)$ by $r^{1- \frac{N-2}{2}} J_{\nu_k}'(\alpha r)$,
and similarly for the other functions
(except for $r^{\pm k}$ which stay unchanged).

For the case $\kappa = 0$, it is easily found that the eigenvalues are
given by $j_{\nu_k+1, \ell}^2$, $k \in \IN$, $\ell \in \IN^{*}$.

For $\kappa > 0$, the boundary conditions yield the system
\begin{equation}
  \label{eq:bd cond ND}
  \begin{cases}
    c\, J_{\nu_k}(\alpha)
    + d\, J_{\nu_k}\bigl(\frac{\kappa}{\alpha} \bigr) = 0,
    \\[1\jot]
    c \,\alpha J'_{\nu_k}(\alpha )
    + d\, \frac{\kappa}{\alpha}
    J'_{\nu_k}\bigl(\frac{\kappa}{\alpha} \bigr) = 0.
  \end{cases}
\end{equation}
As this is nothing but~\eqref{eq:R bd cond} with $\abs{k}$ replaced by
$\nu_k$, for a nontrivial solution $(c,d)$ to exist,
$\alpha$ must be a root of $F_k$ defined as
in~\eqref{Dbis} also with $\abs{k}$ replaced by $\nu_k$.  Because
$\nu_{k+1} = \nu_k + 1$ and the proofs of the properties of roots
$\alpha_{k,\ell}$ of $F_k$ do not use the fact that $k$ is an integer,
they remain valid in this context (with $j_{k,n}$ replaced by
$j_{\nu_k, n}$).

The radial part $R_{k,\ell}$ of the eigenfunctions is now given by
\begin{equation}
  \label{eq:Rkl ND}
  R_{k,\ell}(r)
  = r^{-\frac{N-2}{2}}
  \left( c J_{\nu_k}(\alpha_{k,\ell} \, r)
  + d J_{\nu_k}\Bigl(\frac{\kappa}{\alpha_{k,\ell}}\, r\Bigr) \right),
\end{equation}
where $(c,d)$ is a non-trivial solution to the degenerate
system~\eqref{eq:bd cond ND} with $\alpha = \alpha_{k,\ell}$.
For Lemma~\ref{R01}, using \eqref{A:4} and $\nu_{k+1} = \nu_k + 1$,
one easily shows that
\begin{equation*}
  \partial_r R_{0,1}(r)
  = - r^{-\frac{N-2}{2}} \Bigl[ c \, \alpha_{0,1} J_{\nu_1}(\alpha_{0,1}\, r)
  + d \, \frac{\kappa}{\alpha_{0,1}}
  J_{\nu_1}\Bigl(\frac{\kappa}{\alpha_{0,1}}\, r\Bigr) \Bigr]
\end{equation*}
instead of~\eqref{eq:d R01}.  The rest of the proof adapts in an
obvious fashion.
The rest of the section does not use a special value for $k$ nor
depends on $k$ being an integer.  In conclusion, the following theorem
holds in any dimensions.

\begin{Thm}
  \label{1st eigen nodal ND}
  Denote $R_{k,\ell}$ a function defined by equation~\eqref{eq:Rkl ND}
  with $(c,d)$ a non-trivial solution of~\eqref{eq:bd cond ND} with
  $\alpha = \alpha_{k,\ell}$ where $\alpha_{k,\ell}$ the $\ell$-th
  positive root of
  \begin{math}
    F_k(\alpha) := 
    \frac{\kappa}{\alpha} J_{\nu_k}(\alpha)
    J'_{\nu_k} \bigl(\frac{\kappa}{\alpha} \bigr)
    - \alpha J_{\nu_k} \bigl(\frac{\kappa}{\alpha}\bigr)
    J'_{\nu_k}(\alpha)
  \end{math}
  greater than $\sqrt{\kappa}$.
  \begin{itemize}
  \item If $\kappa\in \intervalco{0,j_{\nu_0,1}j_{\nu_0,2}}$, the
    first eigenvalue is simple and is given by $\lambda_1(\kappa) =
    \alpha_{\nu_0,1}^2(\kappa) + \kappa^2 /
    \alpha_{\nu_0,1}^2(\kappa)$ and the eigenfunctions $\varphi_1$ are
    radial, one-signed and $\abs{\varphi_1}$ is decreasing with
    respect to $r$.

  \item If $\kappa\in \intervaloo{j_{\nu_1,n}j_{\nu_1,n+1}, \,
      j_{\nu_0,n+1}j_{\nu_0,n+2}}$, for some $n\ge 1$, the first
    eigenvalue is simple and given by $\lambda_1(\kappa) =
    \alpha_{\nu_0,1}^2(\kappa) + \kappa^2 /
    \alpha_{\nu_0,1}^2(\kappa)$ and the eigenfunctions are radial and
    have $n+1$ nodal regions.

  \item If $\kappa \in \intervaloo{j_{\nu_0,n+1}j_{\nu_0,n+2}, \,
      j_{\nu_1,n+1}j_{\nu_1,n+2}}$, for some $n\ge 0$, the first
    eigenvalue is given by $\lambda_1(\kappa) =
    \alpha_{\nu_1,1}^2(\kappa) + \kappa^2 /
    \alpha_{\nu_1,1}^2(\kappa)$ and the eigenfunctions $\varphi_1$
    have the form
    \begin{equation*}
      R_{1,1}(r) S\Bigl(\frac{x}{\abs{x}}\Bigr) ,
      \qquad
      S \text{ is a spherical harmonic of degree } 1.
    \end{equation*}
    Moreover the function $R_{1,1}$ has $n$ simple zeros
    in $\intervaloo{0,1}$,
    i.e., $\varphi_1$ has $2(n+1)$ nodal regions.
  \end{itemize}
\end{Thm}

\appendix
\section{Appendix: Bessel functions}

As a convenience to the reader, we gather in this section various
properties of Bessel functions (see for instance~\cite{dlmf}) that are
used in this paper.

\subsection*{Recurrence Relations and Derivatives}

The Bessel functions  $J_{\nu}$ satisfies
\allowdisplaybreaks
\begin{gather}
  \label{A:1}
  {\nu} J_{\nu}(z)=\frac{z}{2} \bigl(J_{{\nu}-1}(z)+J_{{\nu}+1}(z)\bigr),
  \\[1\jot]
  \label{A:3}
  J_{\nu}'(z)=J_{{\nu}-1}(z)-\frac{{\nu}}{z}J_{{\nu}}(z),
  \\[1\jot]
  \label{A:4}
  J_{\nu}'(z)=-J_{{\nu}+1}(z)+\frac{{\nu}}{z}J_{{\nu}}(z),
  \\[1\jot]
  \label{A:5}
  J_0'(z)=-J_{1}(z),
  \\[1\jot]
  \label{A:6}
  z^2J_{\nu}''(z)+ zJ_{\nu}'(z)+(z^2-{\nu}^2) J_{\nu}(z)=0.
\end{gather}

\subsection*{Asymptotic behaviour}

When $\nu$ is fixed and $z\to\infty$ with
$\abs{\arg(z)} \le \pi-\delta$, we have
\begin{equation}
  \label{A:7}
  J_{\nu}(z) =\sqrt{\frac{2}{\pi z}}
  \Bigl( \cos\bigl(z-\frac{1}{2}\nu\pi-\frac{1}{4}\pi \bigr)
  + \exp(|\Im z|)o(1) \Bigr).
\end{equation}
For any given $\nu \ne -1,-2,-3,\dotsc$,
\begin{equation}
  \label{A:8}
  J_{\nu}(z) = \bigl(1 + o(1)\bigr)
  \Bigl(\frac12 z\Bigr)^{\nu} {\bigm/} \Gamma(\nu+1)
  \qquad\text{when } z \to 0.
\end{equation}

\subsection*{Zeros}

When $\nu\geq0$, the zeros of $J_{\nu}$ are simple
and interlace according to the inequalities
\begin{equation}
  \label{A:9}
  j_{\nu,1} <j_{\nu+1,1} <j_{\nu,2} <j_{\nu,2} <j_{\nu,3}	<\cdots
\end{equation}

\bigskip
\bibliographystyle{abbrv}
\bibliography{DNT}

\end{document}

%% file: data/R01-1.tex
\draw plot coordinates { (0.000000, 1.000000) (0.010101, 0.999733) (0.020202, 0.998932) (0.030303, 0.997599) (0.040404, 0.995735) (0.050505, 0.993341) (0.060606, 0.990421) (0.070707, 0.986978) (0.080808, 0.983016) (0.090909, 0.978538) (0.101010, 0.973551) (0.111111, 0.968060) (0.121212, 0.962071) (0.131313, 0.955590) (0.141414, 0.948625) (0.151515, 0.941184) (0.161616, 0.933275) (0.171717, 0.924907) (0.181818, 0.916089) (0.191919, 0.906831) (0.202020, 0.897143) (0.212121, 0.887037) (0.222222, 0.876523) (0.232323, 0.865613) (0.242424, 0.854319) (0.252525, 0.842654) (0.262626, 0.830631) (0.272727, 0.818263) (0.282828, 0.805563) (0.292929, 0.792546) (0.303030, 0.779226) (0.313131, 0.765618) (0.323232, 0.751737) (0.333333, 0.737597) (0.343434, 0.723215) (0.353535, 0.708606) (0.363636, 0.693787) (0.373737, 0.678772) (0.383838, 0.663579) (0.393939, 0.648224) (0.404040, 0.632724) (0.414141, 0.617096) (0.424242, 0.601355) (0.434343, 0.585520) (0.444444, 0.569607) (0.454545, 0.553634) (0.464646, 0.537616) (0.474747, 0.521572) (0.484848, 0.505519) (0.494949, 0.489472) (0.505051, 0.473450) (0.515152, 0.457468) (0.525253, 0.441545) (0.535354, 0.425695) (0.545455, 0.409936) (0.555556, 0.394284) (0.565657, 0.378755) (0.575758, 0.363365) (0.585859, 0.348130) (0.595960, 0.333064) (0.606061, 0.318183) (0.616162, 0.303502) (0.626263, 0.289035) (0.636364, 0.274797) (0.646465, 0.260802) (0.656566, 0.247062) (0.666667, 0.233592) (0.676768, 0.220404) (0.686869, 0.207511) (0.696970, 0.194924) (0.707071, 0.182656) (0.717172, 0.170717) (0.727273, 0.159118) (0.737374, 0.147869) (0.747475, 0.136980) (0.757576, 0.126462) (0.767677, 0.116321) (0.777778, 0.106567) (0.787879, 0.097208) (0.797980, 0.088251) (0.808081, 0.079702) (0.818182, 0.071568) (0.828283, 0.063855) (0.838384, 0.056568) (0.848485, 0.049711) (0.858586, 0.043289) (0.868687, 0.037305) (0.878788, 0.031763) (0.888889, 0.026664) (0.898990, 0.022010) (0.909091, 0.017804) (0.919192, 0.014045) (0.929293, 0.010734) (0.939394, 0.007870) (0.949495, 0.005453) (0.959596, 0.003482) (0.969697, 0.001953) (0.979798, 0.000866) (0.989899, 0.000216) (1.000000, 0.000000) };
\draw plot coordinates { (0.000000, 1.000000) (0.010101, 0.999664) (0.020202, 0.998655) (0.030303, 0.996976) (0.040404, 0.994629) (0.050505, 0.991618) (0.060606, 0.987948) (0.070707, 0.983623) (0.080808, 0.978652) (0.090909, 0.973042) (0.101010, 0.966802) (0.111111, 0.959941) (0.121212, 0.952470) (0.131313, 0.944402) (0.141414, 0.935748) (0.151515, 0.926522) (0.161616, 0.916738) (0.171717, 0.906412) (0.181818, 0.895560) (0.191919, 0.884198) (0.202020, 0.872343) (0.212121, 0.860015) (0.222222, 0.847231) (0.232323, 0.834012) (0.242424, 0.820376) (0.252525, 0.806346) (0.262626, 0.791941) (0.272727, 0.777185) (0.282828, 0.762097) (0.292929, 0.746702) (0.303030, 0.731021) (0.313131, 0.715077) (0.323232, 0.698895) (0.333333, 0.682497) (0.343434, 0.665907) (0.353535, 0.649148) (0.363636, 0.632245) (0.373737, 0.615221) (0.383838, 0.598099) (0.393939, 0.580905) (0.404040, 0.563660) (0.414141, 0.546389) (0.424242, 0.529115) (0.434343, 0.511861) (0.444444, 0.494648) (0.454545, 0.477501) (0.464646, 0.460440) (0.474747, 0.443487) (0.484848, 0.426662) (0.494949, 0.409988) (0.505051, 0.393483) (0.515152, 0.377167) (0.525253, 0.361059) (0.535354, 0.345177) (0.545455, 0.329538) (0.555556, 0.314160) (0.565657, 0.299059) (0.575758, 0.284249) (0.585859, 0.269745) (0.595960, 0.255562) (0.606061, 0.241711) (0.616162, 0.228206) (0.626263, 0.215057) (0.636364, 0.202274) (0.646465, 0.189868) (0.656566, 0.177846) (0.666667, 0.166217) (0.676768, 0.154986) (0.686869, 0.144161) (0.696970, 0.133745) (0.707071, 0.123742) (0.717172, 0.114157) (0.727273, 0.104989) (0.737374, 0.096242) (0.747475, 0.087915) (0.757576, 0.080007) (0.767677, 0.072516) (0.777778, 0.065442) (0.787879, 0.058779) (0.797980, 0.052524) (0.808081, 0.046671) (0.818182, 0.041216) (0.828283, 0.036150) (0.838384, 0.031468) (0.848485, 0.027160) (0.858586, 0.023219) (0.868687, 0.019633) (0.878788, 0.016393) (0.888889, 0.013489) (0.898990, 0.010908) (0.909091, 0.008638) (0.919192, 0.006668) (0.929293, 0.004983) (0.939394, 0.003570) (0.949495, 0.002416) (0.959596, 0.001505) (0.969697, 0.000823) (0.979798, 0.000355) (0.989899, 0.000086) (1.000000, -0.000000) };

%% file: data/R01-3.tex
\draw plot coordinates { (0.000000, 1.000000) (0.010101, 0.999489) (0.020202, 0.997956) (0.030303, 0.995405) (0.040404, 0.991842) (0.050505, 0.987275) (0.060606, 0.981716) (0.070707, 0.975177) (0.080808, 0.967674) (0.090909, 0.959225) (0.101010, 0.949848) (0.111111, 0.939567) (0.121212, 0.928406) (0.131313, 0.916390) (0.141414, 0.903547) (0.151515, 0.889907) (0.161616, 0.875501) (0.171717, 0.860364) (0.181818, 0.844529) (0.191919, 0.828033) (0.202020, 0.810912) (0.212121, 0.793207) (0.222222, 0.774957) (0.232323, 0.756201) (0.242424, 0.736984) (0.252525, 0.717345) (0.262626, 0.697330) (0.272727, 0.676980) (0.282828, 0.656341) (0.292929, 0.635456) (0.303030, 0.614370) (0.313131, 0.593127) (0.323232, 0.571770) (0.333333, 0.550345) (0.343434, 0.528893) (0.353535, 0.507459) (0.363636, 0.486084) (0.373737, 0.464810) (0.383838, 0.443678) (0.393939, 0.422727) (0.404040, 0.401996) (0.414141, 0.381523) (0.424242, 0.361342) (0.434343, 0.341490) (0.444444, 0.321998) (0.454545, 0.302900) (0.464646, 0.284224) (0.474747, 0.265999) (0.484848, 0.248252) (0.494949, 0.231007) (0.505051, 0.214286) (0.515152, 0.198111) (0.525253, 0.182499) (0.535354, 0.167469) (0.545455, 0.153034) (0.555556, 0.139207) (0.565657, 0.125998) (0.575758, 0.113416) (0.585859, 0.101466) (0.595960, 0.090154) (0.606061, 0.079480) (0.616162, 0.069444) (0.626263, 0.060045) (0.636364, 0.051279) (0.646465, 0.043138) (0.656566, 0.035615) (0.666667, 0.028700) (0.676768, 0.022382) (0.686869, 0.016645) (0.696970, 0.011476) (0.707071, 0.006857) (0.717172, 0.002770) (0.727273, -0.000805) (0.737374, -0.003890) (0.747475, -0.006507) (0.757576, -0.008680) (0.767677, -0.010434) (0.777778, -0.011794) (0.787879, -0.012788) (0.797980, -0.013443) (0.808081, -0.013787) (0.818182, -0.013849) (0.828283, -0.013658) (0.838384, -0.013244) (0.848485, -0.012636) (0.858586, -0.011866) (0.868687, -0.010961) (0.878788, -0.009953) (0.888889, -0.008870) (0.898990, -0.007743) (0.909091, -0.006599) (0.919192, -0.005468) (0.929293, -0.004376) (0.939394, -0.003352) (0.949495, -0.002420) (0.959596, -0.001606) (0.969697, -0.000935) (0.979798, -0.000429) (0.989899, -0.000111) (1.000000, 0.000000) };
\draw plot coordinates { (0.000000, 1.000000) (0.010101, 0.999187) (0.020202, 0.996749) (0.030303, 0.992695) (0.040404, 0.987038) (0.050505, 0.979794) (0.060606, 0.970988) (0.070707, 0.960645) (0.080808, 0.948799) (0.090909, 0.935487) (0.101010, 0.920750) (0.111111, 0.904633) (0.121212, 0.887188) (0.131313, 0.868467) (0.141414, 0.848528) (0.151515, 0.827433) (0.161616, 0.805246) (0.171717, 0.782035) (0.181818, 0.757871) (0.191919, 0.732825) (0.202020, 0.706974) (0.212121, 0.680395) (0.222222, 0.653166) (0.232323, 0.625369) (0.242424, 0.597084) (0.252525, 0.568393) (0.262626, 0.539380) (0.272727, 0.510127) (0.282828, 0.480718) (0.292929, 0.451234) (0.303030, 0.421757) (0.313131, 0.392369) (0.323232, 0.363148) (0.333333, 0.334172) (0.343434, 0.305517) (0.353535, 0.277257) (0.363636, 0.249463) (0.373737, 0.222204) (0.383838, 0.195547) (0.393939, 0.169553) (0.404040, 0.144283) (0.414141, 0.119792) (0.424242, 0.096134) (0.434343, 0.073355) (0.444444, 0.051502) (0.454545, 0.030614) (0.464646, 0.010729) (0.474747, -0.008123) (0.484848, -0.025914) (0.494949, -0.042620) (0.505051, -0.058223) (0.515152, -0.072708) (0.525253, -0.086066) (0.535354, -0.098292) (0.545455, -0.109385) (0.555556, -0.119348) (0.565657, -0.128190) (0.575758, -0.135922) (0.585859, -0.142562) (0.595960, -0.148129) (0.606061, -0.152647) (0.616162, -0.156145) (0.626263, -0.158655) (0.636364, -0.160210) (0.646465, -0.160849) (0.656566, -0.160613) (0.666667, -0.159545) (0.676768, -0.157693) (0.686869, -0.155104) (0.696970, -0.151830) (0.707071, -0.147923) (0.717172, -0.143436) (0.727273, -0.138426) (0.737374, -0.132948) (0.747475, -0.127060) (0.757576, -0.120820) (0.767677, -0.114285) (0.777778, -0.107514) (0.787879, -0.100563) (0.797980, -0.093491) (0.808081, -0.086353) (0.818182, -0.079204) (0.828283, -0.072100) (0.838384, -0.065091) (0.848485, -0.058230) (0.858586, -0.051564) (0.868687, -0.045141) (0.878788, -0.039004) (0.888889, -0.033196) (0.898990, -0.027756) (0.909091, -0.022720) (0.919192, -0.018121) (0.929293, -0.013990) (0.939394, -0.010353) (0.949495, -0.007234) (0.959596, -0.004653) (0.969697, -0.002628) (0.979798, -0.001172) (0.989899, -0.000294) (1.000000, -0.000000) };
\draw plot coordinates { (0.000000, 1.000000) (0.010101, 0.999025) (0.020202, 0.996105) (0.030303, 0.991250) (0.040404, 0.984476) (0.050505, 0.975808) (0.060606, 0.965277) (0.070707, 0.952918) (0.080808, 0.938775) (0.090909, 0.922898) (0.101010, 0.905343) (0.111111, 0.886170) (0.121212, 0.865445) (0.131313, 0.843241) (0.141414, 0.819635) (0.151515, 0.794706) (0.161616, 0.768541) (0.171717, 0.741229) (0.181818, 0.712862) (0.191919, 0.683536) (0.202020, 0.653348) (0.212121, 0.622401) (0.222222, 0.590795) (0.232323, 0.558635) (0.242424, 0.526025) (0.252525, 0.493071) (0.262626, 0.459878) (0.272727, 0.426550) (0.282828, 0.393193) (0.292929, 0.359909) (0.303030, 0.326800) (0.313131, 0.293965) (0.323232, 0.261502) (0.333333, 0.229504) (0.343434, 0.198063) (0.353535, 0.167266) (0.363636, 0.137199) (0.373737, 0.107939) (0.383838, 0.079564) (0.393939, 0.052144) (0.404040, 0.025744) (0.414141, 0.000425) (0.424242, -0.023757) (0.434343, -0.046753) (0.444444, -0.068518) (0.454545, -0.089013) (0.464646, -0.108207) (0.474747, -0.126072) (0.484848, -0.142588) (0.494949, -0.157739) (0.505051, -0.171517) (0.515152, -0.183919) (0.525253, -0.194948) (0.535354, -0.204612) (0.545455, -0.212924) (0.555556, -0.219906) (0.565657, -0.225581) (0.575758, -0.229979) (0.585859, -0.233134) (0.595960, -0.235087) (0.606061, -0.235880) (0.616162, -0.235561) (0.626263, -0.234182) (0.636364, -0.231797) (0.646465, -0.228464) (0.656566, -0.224244) (0.666667, -0.219200) (0.676768, -0.213397) (0.686869, -0.206903) (0.696970, -0.199786) (0.707071, -0.192115) (0.717172, -0.183961) (0.727273, -0.175394) (0.737374, -0.166484) (0.747475, -0.157301) (0.757576, -0.147915) (0.767677, -0.138394) (0.777778, -0.128804) (0.787879, -0.119210) (0.797980, -0.109675) (0.808081, -0.100260) (0.818182, -0.091022) (0.828283, -0.082016) (0.838384, -0.073295) (0.848485, -0.064907) (0.858586, -0.056896) (0.868687, -0.049304) (0.878788, -0.042169) (0.888889, -0.035523) (0.898990, -0.029396) (0.909091, -0.023813) (0.919192, -0.018794) (0.929293, -0.014356) (0.939394, -0.010510) (0.949495, -0.007264) (0.959596, -0.004621) (0.969697, -0.002581) (0.979798, -0.001138) (0.989899, -0.000282) (1.000000, 0.000000) };
\draw plot coordinates { (0.000000, 1.000000) (0.010101, 0.998893) (0.020202, 0.995578) (0.030303, 0.990067) (0.040404, 0.982382) (0.050505, 0.972553) (0.060606, 0.960620) (0.070707, 0.946628) (0.080808, 0.930634) (0.090909, 0.912699) (0.101010, 0.892893) (0.111111, 0.871294) (0.121212, 0.847986) (0.131313, 0.823058) (0.141414, 0.796606) (0.151515, 0.768733) (0.161616, 0.739544) (0.171717, 0.709149) (0.181818, 0.677664) (0.191919, 0.645207) (0.202020, 0.611897) (0.212121, 0.577858) (0.222222, 0.543214) (0.232323, 0.508092) (0.242424, 0.472615) (0.252525, 0.436912) (0.262626, 0.401106) (0.272727, 0.365322) (0.282828, 0.329681) (0.292929, 0.294304) (0.303030, 0.259307) (0.313131, 0.224804) (0.323232, 0.190904) (0.333333, 0.157712) (0.343434, 0.125328) (0.353535, 0.093848) (0.363636, 0.063361) (0.373737, 0.033950) (0.383838, 0.005694) (0.393939, -0.021338) (0.404040, -0.047081) (0.414141, -0.071478) (0.424242, -0.094479) (0.434343, -0.116042) (0.444444, -0.136131) (0.454545, -0.154719) (0.464646, -0.171784) (0.474747, -0.187315) (0.484848, -0.201304) (0.494949, -0.213754) (0.505051, -0.224673) (0.515152, -0.234076) (0.525253, -0.241986) (0.535354, -0.248430) (0.545455, -0.253444) (0.555556, -0.257067) (0.565657, -0.259347) (0.575758, -0.260333) (0.585859, -0.260081) (0.595960, -0.258653) (0.606061, -0.256111) (0.616162, -0.252524) (0.626263, -0.247962) (0.636364, -0.242499) (0.646465, -0.236210) (0.656566, -0.229172) (0.666667, -0.221463) (0.676768, -0.213163) (0.686869, -0.204350) (0.696970, -0.195103) (0.707071, -0.185502) (0.717172, -0.175623) (0.727273, -0.165541) (0.737374, -0.155332) (0.747475, -0.145065) (0.757576, -0.134810) (0.767677, -0.124633) (0.777778, -0.114596) (0.787879, -0.104758) (0.797980, -0.095173) (0.808081, -0.085891) (0.818182, -0.076960) (0.828283, -0.068420) (0.838384, -0.060309) (0.848485, -0.052658) (0.858586, -0.045495) (0.868687, -0.038841) (0.878788, -0.032714) (0.888889, -0.027126) (0.898990, -0.022084) (0.909091, -0.017590) (0.919192, -0.013642) (0.929293, -0.010234) (0.939394, -0.007353) (0.949495, -0.004984) (0.959596, -0.003107) (0.969697, -0.001699) (0.979798, -0.000733) (0.989899, -0.000177) (1.000000, 0.000000) };
\draw plot coordinates { (0.000000, 1.000000) (0.010101, 0.998720) (0.020202, 0.994885) (0.030303, 0.988514) (0.040404, 0.979635) (0.050505, 0.968288) (0.060606, 0.954527) (0.070707, 0.938413) (0.080808, 0.920019) (0.090909, 0.899429) (0.101010, 0.876737) (0.111111, 0.852043) (0.121212, 0.825460) (0.131313, 0.797104) (0.141414, 0.767103) (0.151515, 0.735589) (0.161616, 0.702700) (0.171717, 0.668579) (0.181818, 0.633375) (0.191919, 0.597236) (0.202020, 0.560319) (0.212121, 0.522777) (0.222222, 0.484766) (0.232323, 0.446445) (0.242424, 0.407967) (0.252525, 0.369487) (0.262626, 0.331156) (0.272727, 0.293124) (0.282828, 0.255534) (0.292929, 0.218526) (0.303030, 0.182234) (0.313131, 0.146788) (0.323232, 0.112309) (0.333333, 0.078912) (0.343434, 0.046703) (0.353535, 0.015781) (0.363636, -0.013764) (0.373737, -0.041850) (0.383838, -0.068406) (0.393939, -0.093369) (0.404040, -0.116687) (0.414141, -0.138317) (0.424242, -0.158225) (0.434343, -0.176388) (0.444444, -0.192793) (0.454545, -0.207435) (0.464646, -0.220321) (0.474747, -0.231465) (0.484848, -0.240890) (0.494949, -0.248629) (0.505051, -0.254722) (0.515152, -0.259217) (0.525253, -0.262169) (0.535354, -0.263641) (0.545455, -0.263700) (0.555556, -0.262421) (0.565657, -0.259881) (0.575758, -0.256164) (0.585859, -0.251357) (0.595960, -0.245550) (0.606061, -0.238834) (0.616162, -0.231304) (0.626263, -0.223054) (0.636364, -0.214179) (0.646465, -0.204775) (0.656566, -0.194934) (0.666667, -0.184751) (0.676768, -0.174315) (0.686869, -0.163713) (0.696970, -0.153030) (0.707071, -0.142347) (0.717172, -0.131740) (0.727273, -0.121280) (0.737374, -0.111034) (0.747475, -0.101065) (0.757576, -0.091427) (0.767677, -0.082170) (0.777778, -0.073339) (0.787879, -0.064970) (0.797980, -0.057096) (0.808081, -0.049741) (0.818182, -0.042923) (0.828283, -0.036655) (0.838384, -0.030943) (0.848485, -0.025787) (0.858586, -0.021181) (0.868687, -0.017114) (0.878788, -0.013569) (0.888889, -0.010524) (0.898990, -0.007954) (0.909091, -0.005826) (0.919192, -0.004107) (0.929293, -0.002759) (0.939394, -0.001740) (0.949495, -0.001008) (0.959596, -0.000516) (0.969697, -0.000217) (0.979798, -0.000064) (0.989899, -0.000008) (1.000000, 0.000000) };

%% file: data/R11-1.tex
\draw plot coordinates { (0.000000, 0.000000) (0.010101, 0.049911) (0.020202, 0.099733) (0.030303, 0.149377) (0.040404, 0.198755) (0.050505, 0.247779) (0.060606, 0.296363) (0.070707, 0.344419) (0.080808, 0.391865) (0.090909, 0.438615) (0.101010, 0.484589) (0.111111, 0.529706) (0.121212, 0.573888) (0.131313, 0.617058) (0.141414, 0.659144) (0.151515, 0.700073) (0.161616, 0.739777) (0.171717, 0.778189) (0.181818, 0.815247) (0.191919, 0.850889) (0.202020, 0.885060) (0.212121, 0.917705) (0.222222, 0.948775) (0.232323, 0.978221) (0.242424, 1.006001) (0.252525, 1.032075) (0.262626, 1.056407) (0.272727, 1.078966) (0.282828, 1.099723) (0.292929, 1.118654) (0.303030, 1.135740) (0.313131, 1.150964) (0.323232, 1.164314) (0.333333, 1.175783) (0.343434, 1.185367) (0.353535, 1.193066) (0.363636, 1.198886) (0.373737, 1.202834) (0.383838, 1.204923) (0.393939, 1.205170) (0.404040, 1.203596) (0.414141, 1.200226) (0.424242, 1.195087) (0.434343, 1.188213) (0.444444, 1.179639) (0.454545, 1.169404) (0.464646, 1.157552) (0.474747, 1.144129) (0.484848, 1.129186) (0.494949, 1.112774) (0.505051, 1.094950) (0.515152, 1.075773) (0.525253, 1.055304) (0.535354, 1.033607) (0.545455, 1.010750) (0.555556, 0.986801) (0.565657, 0.961830) (0.575758, 0.935912) (0.585859, 0.909121) (0.595960, 0.881533) (0.606061, 0.853226) (0.616162, 0.824279) (0.626263, 0.794772) (0.636364, 0.764787) (0.646465, 0.734404) (0.656566, 0.703707) (0.666667, 0.672776) (0.676768, 0.641696) (0.686869, 0.610548) (0.696970, 0.579414) (0.707071, 0.548377) (0.717172, 0.517516) (0.727273, 0.486913) (0.737374, 0.456647) (0.747475, 0.426795) (0.757576, 0.397434) (0.767677, 0.368639) (0.777778, 0.340483) (0.787879, 0.313039) (0.797980, 0.286374) (0.808081, 0.260557) (0.818182, 0.235651) (0.828283, 0.211721) (0.838384, 0.188824) (0.848485, 0.167019) (0.858586, 0.146359) (0.868687, 0.126895) (0.878788, 0.108676) (0.888889, 0.091746) (0.898990, 0.076146) (0.909091, 0.061916) (0.919192, 0.049089) (0.929293, 0.037698) (0.939394, 0.027769) (0.949495, 0.019326) (0.959596, 0.012391) (0.969697, 0.006980) (0.979798, 0.003105) (0.989899, 0.000777) (1.000000, 0.000000) };
\draw plot coordinates { (0.000000, 0.000000) (0.010101, 0.049907) (0.020202, 0.099703) (0.030303, 0.149278) (0.040404, 0.198522) (0.050505, 0.247325) (0.060606, 0.295581) (0.070707, 0.343183) (0.080808, 0.390028) (0.090909, 0.436013) (0.101010, 0.481041) (0.111111, 0.525014) (0.121212, 0.567840) (0.131313, 0.609429) (0.141414, 0.649695) (0.151515, 0.688557) (0.161616, 0.725936) (0.171717, 0.761760) (0.181818, 0.795959) (0.191919, 0.828470) (0.202020, 0.859233) (0.212121, 0.888196) (0.222222, 0.915309) (0.232323, 0.940529) (0.242424, 0.963819) (0.252525, 0.985147) (0.262626, 1.004486) (0.272727, 1.021815) (0.282828, 1.037121) (0.292929, 1.050392) (0.303030, 1.061627) (0.313131, 1.070828) (0.323232, 1.078001) (0.333333, 1.083161) (0.343434, 1.086327) (0.353535, 1.087523) (0.363636, 1.086779) (0.373737, 1.084129) (0.383838, 1.079614) (0.393939, 1.073278) (0.404040, 1.065170) (0.414141, 1.055345) (0.424242, 1.043859) (0.434343, 1.030775) (0.444444, 1.016159) (0.454545, 1.000080) (0.464646, 0.982610) (0.474747, 0.963825) (0.484848, 0.943803) (0.494949, 0.922624) (0.505051, 0.900371) (0.515152, 0.877128) (0.525253, 0.852983) (0.535354, 0.828021) (0.545455, 0.802330) (0.555556, 0.776001) (0.565657, 0.749121) (0.575758, 0.721780) (0.585859, 0.694067) (0.595960, 0.666069) (0.606061, 0.637874) (0.616162, 0.609568) (0.626263, 0.581235) (0.636364, 0.552958) (0.646465, 0.524818) (0.656566, 0.496893) (0.666667, 0.469260) (0.676768, 0.441992) (0.686869, 0.415158) (0.696970, 0.388826) (0.707071, 0.363059) (0.717172, 0.337918) (0.727273, 0.313458) (0.737374, 0.289733) (0.747475, 0.266790) (0.757576, 0.244674) (0.767677, 0.223425) (0.777778, 0.203077) (0.787879, 0.183664) (0.797980, 0.165210) (0.808081, 0.147739) (0.818182, 0.131268) (0.828283, 0.115811) (0.838384, 0.101375) (0.848485, 0.087964) (0.858586, 0.075580) (0.868687, 0.064216) (0.878788, 0.053865) (0.888889, 0.044512) (0.898990, 0.036141) (0.909091, 0.028730) (0.919192, 0.022254) (0.929293, 0.016685) (0.939394, 0.011991) (0.949495, 0.008136) (0.959596, 0.005081) (0.969697, 0.002786) (0.979798, 0.001205) (0.989899, 0.000293) (1.000000, -0.000000) };

%% file: data/R11-3.tex
\draw plot coordinates { (0.000000, 0.000000) (0.010101, 0.049901) (0.020202, 0.099654) (0.030303, 0.149112) (0.040404, 0.198129) (0.050505, 0.246561) (0.060606, 0.294266) (0.070707, 0.341105) (0.080808, 0.386943) (0.090909, 0.431648) (0.101010, 0.475094) (0.111111, 0.517159) (0.121212, 0.557727) (0.131313, 0.596688) (0.141414, 0.633938) (0.151515, 0.669381) (0.161616, 0.702927) (0.171717, 0.734493) (0.181818, 0.764007) (0.191919, 0.791401) (0.202020, 0.816618) (0.212121, 0.839609) (0.222222, 0.860333) (0.232323, 0.878757) (0.242424, 0.894860) (0.252525, 0.908627) (0.262626, 0.920051) (0.272727, 0.929138) (0.282828, 0.935898) (0.292929, 0.940353) (0.303030, 0.942532) (0.313131, 0.942471) (0.323232, 0.940216) (0.333333, 0.935819) (0.343434, 0.929342) (0.353535, 0.920850) (0.363636, 0.910418) (0.373737, 0.898126) (0.383838, 0.884059) (0.393939, 0.868310) (0.404040, 0.850973) (0.414141, 0.832150) (0.424242, 0.811945) (0.434343, 0.790466) (0.444444, 0.767823) (0.454545, 0.744129) (0.464646, 0.719500) (0.474747, 0.694051) (0.484848, 0.667900) (0.494949, 0.641162) (0.505051, 0.613954) (0.515152, 0.586392) (0.525253, 0.558589) (0.535354, 0.530659) (0.545455, 0.502709) (0.555556, 0.474848) (0.565657, 0.447178) (0.575758, 0.419798) (0.585859, 0.392805) (0.595960, 0.366287) (0.606061, 0.340331) (0.616162, 0.315017) (0.626263, 0.290420) (0.636364, 0.266607) (0.646465, 0.243641) (0.656566, 0.221579) (0.666667, 0.200470) (0.676768, 0.180357) (0.686869, 0.161276) (0.696970, 0.143257) (0.707071, 0.126322) (0.717172, 0.110486) (0.727273, 0.095760) (0.737374, 0.082144) (0.747475, 0.069635) (0.757576, 0.058223) (0.767677, 0.047889) (0.777778, 0.038612) (0.787879, 0.030361) (0.797980, 0.023104) (0.808081, 0.016800) (0.818182, 0.011405) (0.828283, 0.006869) (0.838384, 0.003141) (0.848485, 0.000162) (0.858586, -0.002128) (0.868687, -0.003792) (0.878788, -0.004896) (0.888889, -0.005508) (0.898990, -0.005698) (0.909091, -0.005536) (0.919192, -0.005095) (0.929293, -0.004447) (0.939394, -0.003662) (0.949495, -0.002814) (0.959596, -0.001971) (0.969697, -0.001202) (0.979798, -0.000575) (0.989899, -0.000154) (1.000000, -0.000000) };
\draw plot coordinates { (0.000000, 0.000000) (0.010101, 0.049892) (0.020202, 0.099579) (0.030303, 0.148861) (0.040404, 0.197536) (0.050505, 0.245407) (0.060606, 0.292280) (0.070707, 0.337969) (0.080808, 0.382291) (0.090909, 0.425071) (0.101010, 0.466143) (0.111111, 0.505349) (0.121212, 0.542542) (0.131313, 0.577583) (0.141414, 0.610345) (0.151515, 0.640715) (0.161616, 0.668588) (0.171717, 0.693876) (0.181818, 0.716501) (0.191919, 0.736400) (0.202020, 0.753523) (0.212121, 0.767835) (0.222222, 0.779313) (0.232323, 0.787950) (0.242424, 0.793751) (0.252525, 0.796737) (0.262626, 0.796942) (0.272727, 0.794410) (0.282828, 0.789203) (0.292929, 0.781392) (0.303030, 0.771061) (0.313131, 0.758305) (0.323232, 0.743230) (0.333333, 0.725952) (0.343434, 0.706597) (0.353535, 0.685298) (0.363636, 0.662196) (0.373737, 0.637440) (0.383838, 0.611184) (0.393939, 0.583587) (0.404040, 0.554811) (0.414141, 0.525022) (0.424242, 0.494389) (0.434343, 0.463080) (0.444444, 0.431265) (0.454545, 0.399112) (0.464646, 0.366786) (0.474747, 0.334452) (0.484848, 0.302269) (0.494949, 0.270393) (0.505051, 0.238974) (0.515152, 0.208155) (0.525253, 0.178072) (0.535354, 0.148856) (0.545455, 0.120626) (0.555556, 0.093495) (0.565657, 0.067564) (0.575758, 0.042928) (0.585859, 0.019668) (0.595960, -0.002145) (0.606061, -0.022449) (0.616162, -0.041194) (0.626263, -0.058342) (0.636364, -0.073865) (0.646465, -0.087747) (0.656566, -0.099983) (0.666667, -0.110578) (0.676768, -0.119548) (0.686869, -0.126921) (0.696970, -0.132732) (0.707071, -0.137028) (0.717172, -0.139865) (0.727273, -0.141305) (0.737374, -0.141421) (0.747475, -0.140293) (0.757576, -0.138005) (0.767677, -0.134649) (0.777778, -0.130323) (0.787879, -0.125127) (0.797980, -0.119168) (0.808081, -0.112552) (0.818182, -0.105390) (0.828283, -0.097793) (0.838384, -0.089872) (0.848485, -0.081739) (0.858586, -0.073503) (0.868687, -0.065273) (0.878788, -0.057154) (0.888889, -0.049247) (0.898990, -0.041651) (0.909091, -0.034457) (0.919192, -0.027753) (0.929293, -0.021620) (0.939394, -0.016133) (0.949495, -0.011359) (0.959596, -0.007358) (0.969697, -0.004182) (0.979798, -0.001875) (0.989899, -0.000472) (1.000000, -0.000000) };
\draw plot coordinates { (0.000000, 0.000000) (0.010101, 0.049886) (0.020202, 0.099536) (0.030303, 0.148715) (0.040404, 0.197190) (0.050505, 0.244735) (0.060606, 0.291126) (0.070707, 0.336147) (0.080808, 0.379592) (0.090909, 0.421260) (0.101010, 0.460966) (0.111111, 0.498531) (0.121212, 0.533791) (0.131313, 0.566596) (0.141414, 0.596808) (0.151515, 0.624307) (0.161616, 0.648986) (0.171717, 0.670754) (0.181818, 0.689539) (0.191919, 0.705283) (0.202020, 0.717947) (0.212121, 0.727509) (0.222222, 0.733964) (0.232323, 0.737323) (0.242424, 0.737615) (0.252525, 0.734887) (0.262626, 0.729199) (0.272727, 0.720629) (0.282828, 0.709269) (0.292929, 0.695226) (0.303030, 0.678619) (0.313131, 0.659582) (0.323232, 0.638259) (0.333333, 0.614805) (0.343434, 0.589385) (0.353535, 0.562171) (0.363636, 0.533345) (0.373737, 0.503092) (0.383838, 0.471604) (0.393939, 0.439076) (0.404040, 0.405705) (0.414141, 0.371688) (0.424242, 0.337224) (0.434343, 0.302510) (0.444444, 0.267738) (0.454545, 0.233099) (0.464646, 0.198778) (0.474747, 0.164953) (0.484848, 0.131795) (0.494949, 0.099468) (0.505051, 0.068127) (0.515152, 0.037914) (0.525253, 0.008963) (0.535354, -0.018603) (0.545455, -0.044675) (0.555556, -0.069154) (0.565657, -0.091957) (0.575758, -0.113013) (0.585859, -0.132263) (0.595960, -0.149663) (0.606061, -0.165182) (0.616162, -0.178804) (0.626263, -0.190524) (0.636364, -0.200353) (0.646465, -0.208311) (0.656566, -0.214434) (0.666667, -0.218769) (0.676768, -0.221372) (0.686869, -0.222312) (0.696970, -0.221666) (0.707071, -0.219523) (0.717172, -0.215976) (0.727273, -0.211130) (0.737374, -0.205091) (0.747475, -0.197976) (0.757576, -0.189902) (0.767677, -0.180991) (0.777778, -0.171369) (0.787879, -0.161161) (0.797980, -0.150492) (0.808081, -0.139490) (0.818182, -0.128276) (0.828283, -0.116973) (0.838384, -0.105699) (0.848485, -0.094566) (0.858586, -0.083682) (0.868687, -0.073150) (0.878788, -0.063064) (0.888889, -0.053513) (0.898990, -0.044576) (0.909091, -0.036324) (0.919192, -0.028820) (0.929293, -0.022117) (0.939394, -0.016257) (0.949495, -0.011274) (0.959596, -0.007193) (0.969697, -0.004026) (0.979798, -0.001777) (0.989899, -0.000441) (1.000000, 0.000000) };
\draw plot coordinates { (0.000000, 0.000000) (0.010101, 0.049881) (0.020202, 0.099499) (0.030303, 0.148590) (0.040404, 0.196896) (0.050505, 0.244162) (0.060606, 0.290143) (0.070707, 0.334599) (0.080808, 0.377302) (0.090909, 0.418035) (0.101010, 0.456593) (0.111111, 0.492785) (0.121212, 0.526438) (0.131313, 0.557391) (0.141414, 0.585505) (0.151515, 0.610654) (0.161616, 0.632736) (0.171717, 0.651665) (0.181818, 0.667376) (0.191919, 0.679824) (0.202020, 0.688983) (0.212121, 0.694851) (0.222222, 0.697441) (0.232323, 0.696789) (0.242424, 0.692951) (0.252525, 0.685999) (0.262626, 0.676025) (0.272727, 0.663139) (0.282828, 0.647465) (0.292929, 0.629144) (0.303030, 0.608332) (0.313131, 0.585195) (0.323232, 0.559913) (0.333333, 0.532676) (0.343434, 0.503682) (0.353535, 0.473137) (0.363636, 0.441253) (0.373737, 0.408246) (0.383838, 0.374334) (0.393939, 0.339737) (0.404040, 0.304674) (0.414141, 0.269363) (0.424242, 0.234017) (0.434343, 0.198845) (0.444444, 0.164050) (0.454545, 0.129828) (0.464646, 0.096364) (0.474747, 0.063835) (0.484848, 0.032406) (0.494949, 0.002231) (0.505051, -0.026550) (0.515152, -0.053811) (0.525253, -0.079438) (0.535354, -0.103334) (0.545455, -0.125415) (0.555556, -0.145616) (0.565657, -0.163884) (0.575758, -0.180182) (0.585859, -0.194490) (0.595960, -0.206801) (0.606061, -0.217126) (0.616162, -0.225488) (0.626263, -0.231923) (0.636364, -0.236482) (0.646465, -0.239229) (0.656566, -0.240236) (0.666667, -0.239590) (0.676768, -0.237384) (0.686869, -0.233722) (0.696970, -0.228715) (0.707071, -0.222477) (0.717172, -0.215133) (0.727273, -0.206808) (0.737374, -0.197629) (0.747475, -0.187728) (0.757576, -0.177234) (0.767677, -0.166278) (0.777778, -0.154987) (0.787879, -0.143485) (0.797980, -0.131894) (0.808081, -0.120329) (0.818182, -0.108899) (0.828283, -0.097707) (0.838384, -0.086849) (0.848485, -0.076411) (0.858586, -0.066472) (0.868687, -0.057100) (0.878788, -0.048354) (0.888889, -0.040284) (0.898990, -0.032928) (0.909091, -0.026314) (0.919192, -0.020462) (0.929293, -0.015378) (0.939394, -0.011062) (0.949495, -0.007501) (0.959596, -0.004675) (0.969697, -0.002553) (0.979798, -0.001099) (0.989899, -0.000265) (1.000000, 0.000000) };
\draw plot coordinates { (0.000000, 0.000000) (0.010101, 0.049875) (0.020202, 0.099450) (0.030303, 0.148426) (0.040404, 0.196508) (0.050505, 0.243411) (0.060606, 0.288854) (0.070707, 0.332571) (0.080808, 0.374308) (0.090909, 0.413823) (0.101010, 0.450893) (0.111111, 0.485314) (0.121212, 0.516898) (0.131313, 0.545481) (0.141414, 0.570919) (0.151515, 0.593091) (0.161616, 0.611901) (0.171717, 0.627276) (0.181818, 0.639168) (0.191919, 0.647552) (0.202020, 0.652429) (0.212121, 0.653825) (0.222222, 0.651788) (0.232323, 0.646390) (0.242424, 0.637726) (0.252525, 0.625912) (0.262626, 0.611084) (0.272727, 0.593396) (0.282828, 0.573022) (0.292929, 0.550150) (0.303030, 0.524982) (0.313131, 0.497733) (0.323232, 0.468629) (0.333333, 0.437904) (0.343434, 0.405798) (0.353535, 0.372556) (0.363636, 0.338427) (0.373737, 0.303659) (0.383838, 0.268499) (0.393939, 0.233189) (0.404040, 0.197970) (0.414141, 0.163071) (0.424242, 0.128715) (0.434343, 0.095115) (0.444444, 0.062470) (0.454545, 0.030968) (0.464646, 0.000780) (0.474747, -0.027935) (0.484848, -0.055037) (0.494949, -0.080404) (0.505051, -0.103930) (0.515152, -0.125527) (0.525253, -0.145127) (0.535354, -0.162679) (0.545455, -0.178152) (0.555556, -0.191532) (0.565657, -0.202823) (0.575758, -0.212046) (0.585859, -0.219239) (0.595960, -0.224456) (0.606061, -0.227764) (0.616162, -0.229245) (0.626263, -0.228993) (0.636364, -0.227112) (0.646465, -0.223718) (0.656566, -0.218932) (0.666667, -0.212885) (0.676768, -0.205711) (0.686869, -0.197550) (0.696970, -0.188541) (0.707071, -0.178826) (0.717172, -0.168547) (0.727273, -0.157842) (0.737374, -0.146847) (0.747475, -0.135692) (0.757576, -0.124501) (0.767677, -0.113393) (0.777778, -0.102476) (0.787879, -0.091851) (0.797980, -0.081609) (0.808081, -0.071829) (0.818182, -0.062582) (0.828283, -0.053924) (0.838384, -0.045902) (0.848485, -0.038550) (0.858586, -0.031892) (0.868687, -0.025940) (0.878788, -0.020692) (0.888889, -0.016138) (0.898990, -0.012257) (0.909091, -0.009019) (0.919192, -0.006384) (0.929293, -0.004303) (0.939394, -0.002723) (0.949495, -0.001581) (0.959596, -0.000811) (0.969697, -0.000342) (0.979798, -0.000101) (0.989899, -0.000013) (1.000000, 0.000000) };

%% file: data/R21-1.tex
\draw plot coordinates { (0.000000, 0.000000) (0.010101, 0.001398) (0.020202, 0.005586) (0.030303, 0.012547) (0.040404, 0.022255) (0.050505, 0.034671) (0.060606, 0.049747) (0.070707, 0.067423) (0.080808, 0.087630) (0.090909, 0.110289) (0.101010, 0.135311) (0.111111, 0.162598) (0.121212, 0.192042) (0.131313, 0.223529) (0.141414, 0.256933) (0.151515, 0.292126) (0.161616, 0.328967) (0.171717, 0.367313) (0.181818, 0.407013) (0.191919, 0.447911) (0.202020, 0.489846) (0.212121, 0.532653) (0.222222, 0.576164) (0.232323, 0.620209) (0.242424, 0.664613) (0.252525, 0.709202) (0.262626, 0.753800) (0.272727, 0.798232) (0.282828, 0.842323) (0.292929, 0.885899) (0.303030, 0.928788) (0.313131, 0.970820) (0.323232, 1.011831) (0.333333, 1.051657) (0.343434, 1.090142) (0.353535, 1.127132) (0.363636, 1.162482) (0.373737, 1.196052) (0.383838, 1.227707) (0.393939, 1.257322) (0.404040, 1.284780) (0.414141, 1.309970) (0.424242, 1.332792) (0.434343, 1.353154) (0.444444, 1.370975) (0.454545, 1.386182) (0.464646, 1.398713) (0.474747, 1.408518) (0.484848, 1.415555) (0.494949, 1.419795) (0.505051, 1.421219) (0.515152, 1.419818) (0.525253, 1.415596) (0.535354, 1.408568) (0.545455, 1.398759) (0.555556, 1.386204) (0.565657, 1.370952) (0.575758, 1.353060) (0.585859, 1.332596) (0.595960, 1.309640) (0.606061, 1.284279) (0.616162, 1.256611) (0.626263, 1.226744) (0.636364, 1.194795) (0.646465, 1.160887) (0.656566, 1.125154) (0.666667, 1.087735) (0.676768, 1.048778) (0.686869, 1.008435) (0.696970, 0.966867) (0.707071, 0.924237) (0.717172, 0.880715) (0.727273, 0.836473) (0.737374, 0.791689) (0.747475, 0.746541) (0.757576, 0.701210) (0.767677, 0.655880) (0.777778, 0.610732) (0.787879, 0.565951) (0.797980, 0.521719) (0.808081, 0.478217) (0.818182, 0.435624) (0.828283, 0.394116) (0.838384, 0.353866) (0.848485, 0.315044) (0.858586, 0.277813) (0.868687, 0.242334) (0.878788, 0.208758) (0.888889, 0.177234) (0.898990, 0.147900) (0.909091, 0.120891) (0.919192, 0.096329) (0.929293, 0.074333) (0.939394, 0.055008) (0.949495, 0.038454) (0.959596, 0.024759) (0.969697, 0.014003) (0.979798, 0.006254) (0.989899, 0.001570) (1.000000, -0.000000) };
\draw plot coordinates { (0.000000, 0.000000) (0.010101, 0.001398) (0.020202, 0.005584) (0.030303, 0.012539) (0.040404, 0.022228) (0.050505, 0.034605) (0.060606, 0.049611) (0.070707, 0.067172) (0.080808, 0.087204) (0.090909, 0.109610) (0.101010, 0.134283) (0.111111, 0.161103) (0.121212, 0.189942) (0.131313, 0.220661) (0.141414, 0.253112) (0.151515, 0.287140) (0.161616, 0.322582) (0.171717, 0.359268) (0.181818, 0.397023) (0.191919, 0.435667) (0.202020, 0.475017) (0.212121, 0.514885) (0.222222, 0.555082) (0.232323, 0.595420) (0.242424, 0.635706) (0.252525, 0.675753) (0.262626, 0.715372) (0.272727, 0.754379) (0.282828, 0.792592) (0.292929, 0.829833) (0.303030, 0.865932) (0.313131, 0.900722) (0.323232, 0.934046) (0.333333, 0.965751) (0.343434, 0.995695) (0.353535, 1.023744) (0.363636, 1.049775) (0.373737, 1.073674) (0.383838, 1.095338) (0.393939, 1.114674) (0.404040, 1.131603) (0.414141, 1.146056) (0.424242, 1.157977) (0.434343, 1.167323) (0.444444, 1.174061) (0.454545, 1.178173) (0.464646, 1.179653) (0.474747, 1.178507) (0.484848, 1.174753) (0.494949, 1.168424) (0.505051, 1.159560) (0.515152, 1.148218) (0.525253, 1.134462) (0.535354, 1.118368) (0.545455, 1.100024) (0.555556, 1.079527) (0.565657, 1.056981) (0.575758, 1.032500) (0.585859, 1.006208) (0.595960, 0.978234) (0.606061, 0.948712) (0.616162, 0.917784) (0.626263, 0.885597) (0.636364, 0.852301) (0.646465, 0.818048) (0.656566, 0.782995) (0.666667, 0.747298) (0.676768, 0.711116) (0.686869, 0.674606) (0.696970, 0.637924) (0.707071, 0.601226) (0.717172, 0.564661) (0.727273, 0.528380) (0.737374, 0.492526) (0.747475, 0.457237) (0.757576, 0.422648) (0.767677, 0.388884) (0.777778, 0.356065) (0.787879, 0.324304) (0.797980, 0.293704) (0.808081, 0.264361) (0.818182, 0.236360) (0.828283, 0.209779) (0.838384, 0.184683) (0.848485, 0.161131) (0.858586, 0.139168) (0.868687, 0.118830) (0.878788, 0.100143) (0.888889, 0.083122) (0.898990, 0.067772) (0.909091, 0.054086) (0.919192, 0.042049) (0.929293, 0.031634) (0.939394, 0.022805) (0.949495, 0.015518) (0.959596, 0.009717) (0.969697, 0.005340) (0.979798, 0.002315) (0.989899, 0.000564) (1.000000, 0.000000) };

%% file: data/R21-3.tex
\draw plot coordinates { (0.000000, 0.000000) (0.010101, 0.001398) (0.020202, 0.005581) (0.030303, 0.012526) (0.040404, 0.022189) (0.050505, 0.034509) (0.060606, 0.049413) (0.070707, 0.066808) (0.080808, 0.086586) (0.090909, 0.108628) (0.101010, 0.132797) (0.111111, 0.158945) (0.121212, 0.186914) (0.131313, 0.216532) (0.141414, 0.247618) (0.151515, 0.279984) (0.161616, 0.313433) (0.171717, 0.347763) (0.181818, 0.382766) (0.191919, 0.418233) (0.202020, 0.453949) (0.212121, 0.489702) (0.222222, 0.525279) (0.232323, 0.560469) (0.242424, 0.595064) (0.252525, 0.628861) (0.262626, 0.661664) (0.272727, 0.693281) (0.282828, 0.723532) (0.292929, 0.752243) (0.303030, 0.779253) (0.313131, 0.804410) (0.323232, 0.827578) (0.333333, 0.848629) (0.343434, 0.867452) (0.353535, 0.883949) (0.363636, 0.898039) (0.373737, 0.909653) (0.383838, 0.918740) (0.393939, 0.925263) (0.404040, 0.929203) (0.414141, 0.930553) (0.424242, 0.929326) (0.434343, 0.925548) (0.444444, 0.919260) (0.454545, 0.910517) (0.464646, 0.899391) (0.474747, 0.885963) (0.484848, 0.870331) (0.494949, 0.852602) (0.505051, 0.832896) (0.515152, 0.811340) (0.525253, 0.788074) (0.535354, 0.763244) (0.545455, 0.737002) (0.555556, 0.709508) (0.565657, 0.680924) (0.575758, 0.651418) (0.585859, 0.621158) (0.595960, 0.590315) (0.606061, 0.559058) (0.616162, 0.527555) (0.626263, 0.495972) (0.636364, 0.464471) (0.646465, 0.433210) (0.656566, 0.402340) (0.666667, 0.372005) (0.676768, 0.342342) (0.686869, 0.313479) (0.696970, 0.285535) (0.707071, 0.258619) (0.717172, 0.232827) (0.727273, 0.208248) (0.737374, 0.184954) (0.747475, 0.163010) (0.757576, 0.142464) (0.767677, 0.123355) (0.777778, 0.105708) (0.787879, 0.089534) (0.797980, 0.074834) (0.808081, 0.061595) (0.818182, 0.049793) (0.828283, 0.039393) (0.838384, 0.030346) (0.848485, 0.022596) (0.858586, 0.016075) (0.868687, 0.010708) (0.878788, 0.006409) (0.888889, 0.003086) (0.898990, 0.000641) (0.909091, -0.001031) (0.919192, -0.002038) (0.929293, -0.002495) (0.939394, -0.002515) (0.949495, -0.002216) (0.959596, -0.001716) (0.969697, -0.001130) (0.979798, -0.000574) (0.989899, -0.000161) (1.000000, -0.000000) };
\draw plot coordinates { (0.000000, 0.000000) (0.010101, 0.001397) (0.020202, 0.005578) (0.030303, 0.012509) (0.040404, 0.022136) (0.050505, 0.034382) (0.060606, 0.049149) (0.070707, 0.066322) (0.080808, 0.085763) (0.090909, 0.107320) (0.101010, 0.130821) (0.111111, 0.156081) (0.121212, 0.182899) (0.131313, 0.211065) (0.141414, 0.240356) (0.151515, 0.270542) (0.161616, 0.301384) (0.171717, 0.332641) (0.181818, 0.364067) (0.191919, 0.395416) (0.202020, 0.426442) (0.212121, 0.456903) (0.222222, 0.486562) (0.232323, 0.515187) (0.242424, 0.542558) (0.252525, 0.568461) (0.262626, 0.592696) (0.272727, 0.615078) (0.282828, 0.635433) (0.292929, 0.653608) (0.303030, 0.669463) (0.313131, 0.682879) (0.323232, 0.693754) (0.333333, 0.702008) (0.343434, 0.707580) (0.353535, 0.710430) (0.363636, 0.710538) (0.373737, 0.707907) (0.383838, 0.702559) (0.393939, 0.694536) (0.404040, 0.683900) (0.414141, 0.670734) (0.424242, 0.655137) (0.434343, 0.637225) (0.444444, 0.617133) (0.454545, 0.595009) (0.464646, 0.571014) (0.474747, 0.545321) (0.484848, 0.518116) (0.494949, 0.489592) (0.505051, 0.459949) (0.515152, 0.429393) (0.525253, 0.398134) (0.535354, 0.366384) (0.545455, 0.334354) (0.555556, 0.302255) (0.565657, 0.270293) (0.575758, 0.238670) (0.585859, 0.207582) (0.595960, 0.177214) (0.606061, 0.147744) (0.616162, 0.119338) (0.626263, 0.092150) (0.636364, 0.066319) (0.646465, 0.041971) (0.656566, 0.019216) (0.666667, -0.001852) (0.676768, -0.021155) (0.686869, -0.038634) (0.696970, -0.054246) (0.707071, -0.067965) (0.717172, -0.079784) (0.727273, -0.089711) (0.737374, -0.097774) (0.747475, -0.104014) (0.757576, -0.108490) (0.767677, -0.111276) (0.777778, -0.112457) (0.787879, -0.112135) (0.797980, -0.110420) (0.808081, -0.107436) (0.818182, -0.103313) (0.828283, -0.098192) (0.838384, -0.092216) (0.848485, -0.085538) (0.858586, -0.078309) (0.868687, -0.070686) (0.878788, -0.062823) (0.888889, -0.054875) (0.898990, -0.046992) (0.909091, -0.039321) (0.919192, -0.032001) (0.929293, -0.025167) (0.939394, -0.018941) (0.949495, -0.013440) (0.959596, -0.008767) (0.969697, -0.005014) (0.979798, -0.002261) (0.989899, -0.000572) (1.000000, -0.000000) };
\draw plot coordinates { (0.000000, 0.000000) (0.010101, 0.001397) (0.020202, 0.005576) (0.030303, 0.012499) (0.040404, 0.022104) (0.050505, 0.034303) (0.060606, 0.048987) (0.070707, 0.066023) (0.080808, 0.085257) (0.090909, 0.106516) (0.101010, 0.129609) (0.111111, 0.154326) (0.121212, 0.180445) (0.131313, 0.207730) (0.141414, 0.235936) (0.151515, 0.264808) (0.161616, 0.294086) (0.171717, 0.323507) (0.181818, 0.352805) (0.191919, 0.381716) (0.202020, 0.409981) (0.212121, 0.437344) (0.222222, 0.463559) (0.232323, 0.488389) (0.242424, 0.511610) (0.252525, 0.533013) (0.262626, 0.552402) (0.272727, 0.569602) (0.282828, 0.584456) (0.292929, 0.596826) (0.303030, 0.606597) (0.313131, 0.613677) (0.323232, 0.617995) (0.333333, 0.619506) (0.343434, 0.618187) (0.353535, 0.614039) (0.363636, 0.607089) (0.373737, 0.597384) (0.383838, 0.584996) (0.393939, 0.570018) (0.404040, 0.552563) (0.414141, 0.532766) (0.424242, 0.510777) (0.434343, 0.486766) (0.444444, 0.460916) (0.454545, 0.433424) (0.464646, 0.404498) (0.474747, 0.374356) (0.484848, 0.343222) (0.494949, 0.311327) (0.505051, 0.278903) (0.515152, 0.246184) (0.525253, 0.213403) (0.535354, 0.180790) (0.545455, 0.148568) (0.555556, 0.116954) (0.565657, 0.086155) (0.575758, 0.056368) (0.585859, 0.027777) (0.595960, 0.000550) (0.606061, -0.025158) (0.616162, -0.049209) (0.626263, -0.071484) (0.636364, -0.091882) (0.646465, -0.110320) (0.656566, -0.126737) (0.666667, -0.141088) (0.676768, -0.153352) (0.686869, -0.163525) (0.696970, -0.171623) (0.707071, -0.177682) (0.717172, -0.181754) (0.727273, -0.183910) (0.737374, -0.184235) (0.747475, -0.182832) (0.757576, -0.179815) (0.767677, -0.175311) (0.777778, -0.169459) (0.787879, -0.162404) (0.797980, -0.154300) (0.808081, -0.145308) (0.818182, -0.135589) (0.828283, -0.125309) (0.838384, -0.114633) (0.848485, -0.103724) (0.858586, -0.092741) (0.868687, -0.081840) (0.878788, -0.071168) (0.888889, -0.060865) (0.898990, -0.051061) (0.909091, -0.041874) (0.919192, -0.033412) (0.929293, -0.025768) (0.939394, -0.019023) (0.949495, -0.013241) (0.959596, -0.008474) (0.969697, -0.004754) (0.979798, -0.002102) (0.989899, -0.000522) (1.000000, 0.000000) };
\draw plot coordinates { (0.000000, 0.000000) (0.010101, 0.001397) (0.020202, 0.005574) (0.030303, 0.012490) (0.040404, 0.022075) (0.050505, 0.034234) (0.060606, 0.048844) (0.070707, 0.065761) (0.080808, 0.084814) (0.090909, 0.105815) (0.101010, 0.128553) (0.111111, 0.152801) (0.121212, 0.178316) (0.131313, 0.204846) (0.141414, 0.232124) (0.151515, 0.259879) (0.161616, 0.287834) (0.171717, 0.315710) (0.181818, 0.343228) (0.191919, 0.370115) (0.202020, 0.396102) (0.212121, 0.420929) (0.222222, 0.444347) (0.232323, 0.466123) (0.242424, 0.486038) (0.252525, 0.503890) (0.262626, 0.519499) (0.272727, 0.532705) (0.282828, 0.543372) (0.292929, 0.551387) (0.303030, 0.556664) (0.313131, 0.559140) (0.323232, 0.558780) (0.333333, 0.555575) (0.343434, 0.549542) (0.353535, 0.540724) (0.363636, 0.529188) (0.373737, 0.515026) (0.383838, 0.498355) (0.393939, 0.479310) (0.404040, 0.458050) (0.414141, 0.434748) (0.424242, 0.409597) (0.434343, 0.382804) (0.444444, 0.354587) (0.454545, 0.325175) (0.464646, 0.294802) (0.474747, 0.263711) (0.484848, 0.232144) (0.494949, 0.200346) (0.505051, 0.168559) (0.515152, 0.137019) (0.525253, 0.105956) (0.535354, 0.075592) (0.545455, 0.046137) (0.555556, 0.017788) (0.565657, -0.009272) (0.575758, -0.034878) (0.585859, -0.058880) (0.595960, -0.081150) (0.606061, -0.101575) (0.616162, -0.120067) (0.626263, -0.136557) (0.636364, -0.150996) (0.646465, -0.163358) (0.656566, -0.173637) (0.666667, -0.181847) (0.676768, -0.188021) (0.686869, -0.192213) (0.696970, -0.194494) (0.707071, -0.194950) (0.717172, -0.193684) (0.727273, -0.190812) (0.737374, -0.186461) (0.747475, -0.180769) (0.757576, -0.173883) (0.767677, -0.165955) (0.777778, -0.157141) (0.787879, -0.147602) (0.797980, -0.137497) (0.808081, -0.126984) (0.818182, -0.116220) (0.828283, -0.105353) (0.838384, -0.094528) (0.848485, -0.083880) (0.858586, -0.073535) (0.868687, -0.063606) (0.878788, -0.054197) (0.888889, -0.045397) (0.898990, -0.037282) (0.909091, -0.029912) (0.919192, -0.023335) (0.929293, -0.017582) (0.939394, -0.012670) (0.949495, -0.008601) (0.959596, -0.005362) (0.969697, -0.002927) (0.979798, -0.001258) (0.989899, -0.000303) (1.000000, 0.000000) };
\draw plot coordinates { (0.000000, 0.000000) (0.010101, 0.001397) (0.020202, 0.005572) (0.030303, 0.012478) (0.040404, 0.022038) (0.050505, 0.034143) (0.060606, 0.048658) (0.070707, 0.065419) (0.080808, 0.084238) (0.090909, 0.104903) (0.101010, 0.127181) (0.111111, 0.150824) (0.121212, 0.175564) (0.131313, 0.201124) (0.141414, 0.227217) (0.151515, 0.253550) (0.161616, 0.279828) (0.171717, 0.305756) (0.181818, 0.331043) (0.191919, 0.355405) (0.202020, 0.378568) (0.212121, 0.400273) (0.222222, 0.420275) (0.232323, 0.438349) (0.242424, 0.454288) (0.252525, 0.467913) (0.262626, 0.479066) (0.272727, 0.487617) (0.282828, 0.493464) (0.292929, 0.496533) (0.303030, 0.496780) (0.313131, 0.494189) (0.323232, 0.488775) (0.333333, 0.480580) (0.343434, 0.469677) (0.353535, 0.456164) (0.363636, 0.440164) (0.373737, 0.421825) (0.383838, 0.401318) (0.393939, 0.378832) (0.404040, 0.354575) (0.414141, 0.328769) (0.424242, 0.301650) (0.434343, 0.273461) (0.444444, 0.244453) (0.454545, 0.214882) (0.464646, 0.185004) (0.474747, 0.155072) (0.484848, 0.125336) (0.494949, 0.096038) (0.505051, 0.067409) (0.515152, 0.039670) (0.525253, 0.013025) (0.535354, -0.012337) (0.545455, -0.036246) (0.555556, -0.058551) (0.565657, -0.079122) (0.575758, -0.097854) (0.585859, -0.114660) (0.595960, -0.129478) (0.606061, -0.142271) (0.616162, -0.153022) (0.626263, -0.161739) (0.636364, -0.168449) (0.646465, -0.173202) (0.656566, -0.176067) (0.666667, -0.177130) (0.676768, -0.176494) (0.686869, -0.174276) (0.696970, -0.170607) (0.707071, -0.165627) (0.717172, -0.159486) (0.727273, -0.152337) (0.737374, -0.144340) (0.747475, -0.135655) (0.757576, -0.126443) (0.767677, -0.116859) (0.777778, -0.107057) (0.787879, -0.097181) (0.797980, -0.087370) (0.808081, -0.077749) (0.818182, -0.068435) (0.828283, -0.059529) (0.838384, -0.051121) (0.848485, -0.043284) (0.858586, -0.036079) (0.868687, -0.029548) (0.878788, -0.023719) (0.888889, -0.018606) (0.898990, -0.014206) (0.909091, -0.010502) (0.919192, -0.007464) (0.929293, -0.005050) (0.939394, -0.003205) (0.949495, -0.001866) (0.959596, -0.000959) (0.969697, -0.000405) (0.979798, -0.000120) (0.989899, -0.000015) (1.000000, 0.000000) };

%% file: data/R01-sign4.tex
\draw[fill] (0.171709, 0.143470) circle(1pt);
\draw[fill] (0.394144, -0.176125) circle(1pt);
\draw[fill] (0.617893, 0.144705) circle(1pt);
\draw[fill] (0.841939, -0.072885) circle(1pt);
\draw plot coordinates { (0.000000, 1.000000) (0.010101, 0.995954) (0.020202, 0.983868) (0.030303, 0.963898) (0.040404, 0.936298) (0.050505, 0.901420) (0.060606, 0.859709) (0.070707, 0.811694) (0.080808, 0.757980) (0.090909, 0.699242) (0.101010, 0.636212) (0.111111, 0.569671) (0.121212, 0.500434) (0.131313, 0.429343) (0.141414, 0.357250) (0.151515, 0.285008) (0.161616, 0.213458) (0.171717, 0.143416) (0.181818, 0.075665) (0.191919, 0.010942) (0.202020, -0.050075) (0.212121, -0.106768) (0.222222, -0.158593) (0.232323, -0.205085) (0.242424, -0.245862) (0.252525, -0.280631) (0.262626, -0.309188) (0.272727, -0.331422) (0.282828, -0.347310) (0.292929, -0.356922) (0.303030, -0.360412) (0.313131, -0.358014) (0.323232, -0.350041) (0.333333, -0.336874) (0.343434, -0.318955) (0.353535, -0.296783) (0.363636, -0.270896) (0.373737, -0.241870) (0.383838, -0.210307) (0.393939, -0.176820) (0.404040, -0.142029) (0.414141, -0.106550) (0.424242, -0.070983) (0.434343, -0.035904) (0.444444, -0.001857) (0.454545, 0.030652) (0.464646, 0.061165) (0.474747, 0.089275) (0.484848, 0.114633) (0.494949, 0.136950) (0.505051, 0.156001) (0.515152, 0.171625) (0.525253, 0.183725) (0.535354, 0.192269) (0.545455, 0.197285) (0.555556, 0.198864) (0.565657, 0.197147) (0.575758, 0.192329) (0.585859, 0.184648) (0.595960, 0.174382) (0.606061, 0.161840) (0.616162, 0.147357) (0.626263, 0.131287) (0.636364, 0.113993) (0.646465, 0.095843) (0.656566, 0.077203) (0.666667, 0.058427) (0.676768, 0.039855) (0.686869, 0.021804) (0.696970, 0.004565) (0.707071, -0.011602) (0.717172, -0.026472) (0.727273, -0.039857) (0.737374, -0.051609) (0.747475, -0.061621) (0.757576, -0.069824) (0.767677, -0.076194) (0.777778, -0.080740) (0.787879, -0.083512) (0.797980, -0.084592) (0.808081, -0.084094) (0.818182, -0.082156) (0.828283, -0.078941) (0.838384, -0.074628) (0.848485, -0.069410) (0.858586, -0.063486) (0.868687, -0.057061) (0.878788, -0.050334) (0.888889, -0.043502) (0.898990, -0.036748) (0.909091, -0.030241) (0.919192, -0.024133) (0.929293, -0.018554) (0.939394, -0.013610) (0.949495, -0.009382) (0.959596, -0.005927) (0.969697, -0.003271) (0.979798, -0.001418) (0.989899, -0.000344) (1.000000, 0.000000) };

%% file: data/R02-sign4.tex
\draw[fill] (0.168909, 0.864800) circle(1pt);
\draw[fill] (0.387717, -0.474941) circle(1pt);
\draw[fill] (0.607818, -0.273197) circle(1pt);
\draw[fill] (0.828210, 0.441052) circle(1pt);
\draw plot coordinates { (0.000000, 0.164411) (0.010101, 0.168752) (0.020202, 0.181681) (0.030303, 0.202917) (0.040404, 0.232000) (0.050505, 0.268296) (0.060606, 0.311013) (0.070707, 0.359211) (0.080808, 0.411824) (0.090909, 0.467679) (0.101010, 0.525516) (0.111111, 0.584016) (0.121212, 0.641820) (0.131313, 0.697562) (0.141414, 0.749888) (0.151515, 0.797487) (0.161616, 0.839114) (0.171717, 0.873612) (0.181818, 0.899938) (0.191919, 0.917181) (0.202020, 0.924579) (0.212121, 0.921537) (0.222222, 0.907636) (0.232323, 0.882646) (0.242424, 0.846526) (0.252525, 0.799431) (0.262626, 0.741708) (0.272727, 0.673893) (0.282828, 0.596700) (0.292929, 0.511012) (0.303030, 0.417865) (0.313131, 0.318431) (0.323232, 0.214000) (0.333333, 0.105955) (0.343434, -0.004249) (0.353535, -0.115108) (0.363636, -0.225101) (0.373737, -0.332708) (0.383838, -0.436440) (0.393939, -0.534863) (0.404040, -0.626622) (0.414141, -0.710462) (0.424242, -0.785253) (0.434343, -0.850003) (0.444444, -0.903876) (0.454545, -0.946207) (0.464646, -0.976507) (0.474747, -0.994475) (0.484848, -1.000000) (0.494949, -0.993159) (0.505051, -0.974215) (0.515152, -0.943612) (0.525253, -0.901965) (0.535354, -0.850048) (0.545455, -0.788775) (0.555556, -0.719190) (0.565657, -0.642442) (0.575758, -0.559769) (0.585859, -0.472470) (0.595960, -0.381889) (0.606061, -0.289384) (0.616162, -0.196313) (0.626263, -0.104003) (0.636364, -0.013735) (0.646465, 0.073282) (0.656566, 0.155927) (0.666667, 0.233184) (0.676768, 0.304158) (0.686869, 0.368084) (0.696970, 0.424343) (0.707071, 0.472460) (0.717172, 0.512117) (0.727273, 0.543147) (0.737374, 0.565535) (0.747475, 0.579415) (0.757576, 0.585057) (0.767677, 0.582867) (0.777778, 0.573363) (0.787879, 0.557172) (0.797980, 0.535008) (0.808081, 0.507656) (0.818182, 0.475958) (0.828283, 0.440788) (0.838384, 0.403040) (0.848485, 0.363605) (0.858586, 0.323355) (0.868687, 0.283127) (0.878788, 0.243702) (0.888889, 0.205799) (0.898990, 0.170053) (0.909091, 0.137012) (0.919192, 0.107123) (0.929293, 0.080731) (0.939394, 0.058069) (0.949495, 0.039261) (0.959596, 0.024324) (0.969697, 0.013165) (0.979798, 0.005594) (0.989899, 0.001328) (1.000000, 0.000000) };